\documentclass[11pt,a4paper,reqno]{amsart}
\usepackage[T1]{fontenc}
\usepackage{amsmath}
\usepackage{amsthm}
\usepackage{amsfonts}
\usepackage{amssymb}
\usepackage{graphicx}
\usepackage{amsbsy}
\usepackage{mathrsfs}
\usepackage{bbm}
\newtheorem{theorem}{Theorem}[section]
\newtheorem{lemma}[theorem]{Lemma}

\newtheorem{proposition}[theorem]{Proposition}

\newtheorem{definition}[theorem]{Definition}
\theoremstyle{remark}
\newtheorem{remark}[theorem]{\it \bf{Remark}\/}

\numberwithin{equation}{section}
\catcode`@=11
\def\section{\@startsection{section}{1}%
  \z@{1.5\linespacing\@plus\linespacing}{.5\linespacing}%
  {\normalfont\bfseries\large\centering}}
\catcode`@=12
\newcommand{\be}{\begin{equation}}
\newcommand{\ee}{\end{equation}}
\newcommand{\bea}{\begin{eqnarray}}
\newcommand{\eea}{\end{eqnarray}}
\newcommand{\bee}{\begin{eqnarray*}}
\newcommand{\eee}{\end{eqnarray*}}

\def\pa{\partial}

\def\RR{\mathbb{R}}

\def\hw{\hat{w}}

\def\fref#1{{\rm (\ref{#1})}}

\catcode`@=11
\def\supess{\mathop{\operator@font Sup\,ess}}
\catcode`@=12

\def\RR{\mathbb{R}}

\def\g#1{{\bf #1}}
\def\e{\varepsilon}

\def\fref#1{{\rm (\ref{#1})}}

\def\R2+{\RR ^2_+}

\def\tt{\tilde{T}}

\def\lsl{\frac{\lambda_s}{\lambda}}

\def\pa{\partial}
\def\hb{\hat{b}}

\def\Mod{\rm Mod}

\def\lim{\mathop{\rm lim}}

\def\R{\Bbb R}
\def\e{\varepsilon}
\def\l{\lambda}

\def\log{{\rm log}}

\def\et{\tilde{\e}}

\def\lsl{\frac{\lambda_s}{\lambda}}

\def\qbt{P_{B_1}}

\def\S{\Sigma}
\def\T{\Theta}

\def\pa{\partial}

\def\et{\tilde{\e}}

\def\eh{\hat{\e}}

\def\pa{\partial}

\def\Mm{\mathcal M}

\def\pu {\phi_u}
\def\wh{\hat{w}}
\def\bh{\hat{b}}
\title[]{Stable blow-up dynamic for the parabolic-parabolic Patlak-Keller-Segel model}
\author[R. Schweyer]{R\'emi Schweyer}
\address{Laboratoire Analyse, G\'eom\'etrie et Mod\'elisation, Universit\'e de Cergy-Pontoise, France}
\email{remi.schweyer@u-cergy.fr}

\begin{document}
\maketitle

\begin{abstract} We consider the parabolic-parabolic  two-dimensional Patlak-Keller-Segel problem. We prove the existence of stable blow-up dynamics in finite time in the radial case. We extend in this article the result of \cite{raphael2012b} for the parabolic-elliptic case. 
\end{abstract}
\section{Introduction}
\subsection{Setting of the problem}
The Patlak-Keller-Segel model, which is studied in this article, is one of the simplest modelisation of chemotaxis, suggested by Patlak in 1953 \cite{patlak1953} and by Keller and Segel in \cite{keller1970}, \cite{keller1971} and \cite{keller1971b}. More precisely, the cells of living organisms communicate to each other through chemical species. When these chemical species cause a collective movement of the cells, chemotaxis is the term applied to describe this phenomenon. Chemotaxis plays a crucial role in a large number of biological situations, like angiogenesis, embryonic development or formation of colonies of bacteria.  The interested reader can refer to \cite{perthame2007}, \cite{tindall2008}, \cite{tindall2008a} and \cite{hillen2009} for more detailed information. \par
In this paper, we study a particular case of chemotaxis, in which the chemical specy is directly product by the cells.  This implies a strong coupling between the spatio-temporal dynamics of cells and the chemoattractant. The most famous example is the amoeba {\it Dictystelium disco\"ideum}, which has attracted a considerable attention for the past fourty years. A complete review is available in \cite{horstmann2003} and \cite{horstmann2004}.
The following model is suggested by Nanjundiah in \cite{nanjundiah1973} :
\be
\label{kps2}
\left\{ \begin{array}{l} \pa_t u = \nabla.( \kappa \nabla u + \chi u \nabla v),\\
\pa_t v = \eta \Delta v - \beta u + \alpha v, \\
u(0,x) =u_0>0,\\
v(0,x) = v_0>0,
\end{array} \right. \ \ (t,x)\in \mathbb R^+ \times \Omega.
\ee
where $\Omega$ is an open set of $\mathbb R^2$ or the whole plane, $u$ is the density of amoeba, $v$ the concentration of the chemoattractant, $\kappa$ and $\eta$ are respectively diffusion coefficients for the amoeba ant the chemoattractant, $\chi$ the sensitivity of the amoeba to the chemoattractant, $\beta$ the rate of production of chemoattractant per amoeba and $\alpha$ the rate of destruction of the chemoattractant. All above quantities are positive. For biological reasons, the last rate $\alpha$ is very small, and in a first approximation, we can consider $\alpha=0$. Using a suitable rescaling and an adimensionalization, we obtain:
\be
\label{kps}
\left\{ \begin{array}{l} \pa_t u = \nabla.(\nabla u + u \nabla v),\\
c \pa_t v = \Delta v - u, \\
u(0,x) =u_0>0,\\
v(0,x) = v_0>0,
\end{array} \right. \ \ (t,x)\in \mathbb R^+ \times \mathbb R^2.
\ee
where $c=\frac{\kappa}{\eta}$ is the difference of the time scales of the diffusive processes undergone by u and v. \par
\subsection{On the parabolic-elliptic model}
\label{dkiuend}
In this subsection we take $c=0$, and the system \fref{kps} becomes:
\be
\label{kpsee}
(PKSpe)\ \ \left\{\begin{array}{lll}
\pa_tu=\nabla\cdot(\nabla u+u\nabla v),\\
v=\frac{1}{2\pi}\log |x|\star u\\
u_{|t=0}=u_0> 0
\end{array} \right .\ \ (t,x)\in \Bbb R\times \Bbb R^2,\
\ee
In a pioneering work \cite{jager1992} is proved for the first time that the corresponding solution to \fref{kpsee} with small enough mass is global in time, and blow-up can occur for initial data with large mass. A important step was taken in \cite{beckner1993} and \cite{carlen1992} with the establishment of the logarithmic Hardy-Littlewood-Sobolev inequality:
 $\forall u\geq 0$ with $\int u=M$, 
\be
\label{logsobolev}
\int u\log u+\frac{4\pi}{M}\int \phi_u u\geq M\left[1+\log\pi-\log M\right]
\ee
where given $u$ we defined: $$\phi_u=\frac{1}{2\pi}\log |x|\star u.$$
The left term coincides with the free energy for $M=8\pi$, and thus, the free energy is lower bounded for this mass. Furthermore $Q$ is up to symmetry the unique minimizer, where $Q$ is a radial explicit profile defined by:
\be
Q(x) = \frac{8}{(1+|x|^2)^2}.
\ee
This is the key ingredient in the proof that all solutions of \fref{kpsee} such that $M\leq 8\pi$ are global in time. \cite{diaz1998}, \cite{dolbeault2004} and \cite{blanchet2006}. Moreover, for $M<8\pi$, zero is a local universal attractor. The dynamics of these solutions is sharply described in \cite{campos2012}. \par
Now, the solutions with enough good decay, {\it ie} solutions with finite second moment $\int_{\mathbb R^2} |x|^2u_0(x)dx < \infty$, satisfy the virial law: $$\frac{d}{dt}\int |x|^2u(t,x)dx=4\left(1-\frac{M}{8\pi}\right)M.$$ This argument gives two informations. First, if the mass of the solution is $8\pi$ and the second moment is finite, then, this second moment is preserved. Secondly, if the mass is larger than $8\pi$ and the second moment is finite, the solution blows-up in finite time. \par
In the case of $M=8\pi$, there are two conservation laws : the mass and the second moment. The problem is both $L^1$ critical, and energy critical. In \cite{blanchet2008}, the authors prove in the case of finite second moment, the solutions grow-up and converge to a Dirac mass at infinity time. The argument is not constructive and gives no information about the rate of convergence. The article \cite{senba2009} answers this question. In the case of bounded domain, the situation is described in \cite{biler2005} and \cite{kavallaris2009}. \par
In the case of infinite second moment is proved in \cite{blanchet2012} the existence of global solutions converging to the ground state, and the rate of convergence is proved in \cite{carlen2013}. The case of very slow decay with finite mass $8\pi$ is an open question. \par
For the solution with large mass, {\it ie} $M>8\pi$, we have seen that if the second moment is finite, the solution blows-up in finite time. In fact, this argument can be extended in the case of all solutions with large mass. The problem is still poorly understood. In the radial case, a first example of blow-up solution in finite time is described in \cite{herrero1996} using formal matching asymptotic, and completed in \cite{velazquez2002} and \cite{velazquez2006}. Recently, in \cite{raphael2012b} is given a sharp description of the blow-up dynamics, and in particular its stability. 
\subsection{On the parabolic-parabolic model}
\label{dkiuend1}
In \cite{nagai2000} is proved by fixed point arguments the existence and the uniqueness of local non negative smooth solutions of \fref{kps}. Moreover, if we consider initial data with enough fast decay, {\it ie} $u_0 \in L^1(\mathbb R^2)$, there is the conservation of the mass :
\be
\label{consmasse}
\int_{\Bbb R^2}u(t,x)dx=\int_{\Bbb R^2}u(0,x)dx =M
\ee
Now, if $(u,v)$ is solution to the problem \fref{kps}, then the rescaled solution
\bee
\left|\begin{array}{l} u_\l\\ v_\l \end{array} \right. = \left|\begin{array}{l} \l^2 u(\l^2t, \l x) \\ v(\l^2t,\l x) \end{array} \right.
\eee
is also a solution. Furthermore, the $L^1$ norm in unchanged by this scaling :
\bee
\int_{\Bbb R^2}u_\l(t,x)dx=\int_{\Bbb R^2}u(\l^2 t,x)dx
\eee
The problem is thus $L^1$ critical. The second important quantity is the free energy functionnal :
\bee
E(u,v) = \int _{\mathbb R^2} u(x)\log u(x) dx +  \int _{\mathbb R^2} u(x)v(x) dx -  \frac 12\int _{\mathbb R^2} v(x) \Delta v(x) dx
\eee
This functional play a crucial rule in nonlinear diffusion and kinetic models \cite{arnold2004}. This energy is dissipated by the flow. Moreover, the problem is almost energy critical in the following sense:
\bee
E(u_\l,v_\l) = E(u,v) + M \left( 2 - \frac M {4\pi} \right) \log \l.
\eee
For initial data with small mass $M<8\pi$, corresponding solutions of \fref{kps} exist globally in time \cite{calvez2008}. See \cite{nagai1997} and \cite{biler1998} for similar results on a disk. \par
Now, the threshold effect of the mass on the dynamics of the solutions of \fref{kps} is more imprecise than the parabolic-elliptic case. Indeed, self-similar solutions with $M>8\pi$ are exhibited (\cite{naito2006} in the case $c=1$, \cite{biler2011} for all $c>0$). Hence, in opposite of the parabolic-elliptic case, solution with large mass can exist globally in time. \par
A blow-up dynamic in finite time is proved in a bounded domain in \cite{herrero1997}, using formal matching asymptotic as \cite{herrero1996} for the parabolic-elliptic case. 
\subsection{Main result}
Our result concerns the parabolic-parabolic model, in the case $c=1$, {\it ie} :
\be
\label{kps11}
\left\{ \begin{array}{l} \pa_t u = \nabla.(\nabla u + u \nabla v),\\
\pa_t v = \Delta v - u, \\
u(0,x) =u_0>0,\\
v(0,x) = v_0>0,
\end{array} \right. \ \ (t,x)\in \mathbb R^+ \times \mathbb R^2.
\ee
In the continuation of \cite{raphael2012b} for the parabolic-elliptic case, we obtain a sharp description of the blow-up dynamics in finite time, for data with small super critical mass. In particular, we obtain the rate of convergence and the stability of the blow-up, for small perturbation in the energy space. We use a similar approach like in \cite{schweyer2012} for the energy critical semilinear heat flow, and in \cite{raphael2013b} and \cite{raphael2013} for the harmonic heat flow. \par
Let the weighted $H^2$ space:
\be
\label{defhtwow}
\|\e\|_{H^2_Q}=\|(1+r^2)\Delta \e\|_{L^2}+\left\|(1+r)\nabla \e\right\|_{L^2}+\|\e\|_{L^2}
\ee
and the weighted $H^3$ space:
\be
\label{defhtwow2}
\|\eta\|_{N}=\|(1+r)\nabla\Delta \eta\|_{L^2}+\left\|\Delta \eta\right\|_{L^2}.
\ee
We introduce the energy norm
\be
\label{energyspcaee}
\|\boldsymbol \e\|_{\mathcal E}=\|\e\|_{H^2_Q}+ \|\eta\|_N + \|\e\|_{L^1} = \|(\e,\eta)\|_{W_Q} + \|\e\|_{L^1}.
\ee
\begin{theorem}[Stable chemotactic blow up]
\label{thmmain}
There exists a set of initial data of the form $${ \bf u_0} = \left| \begin{array}{l} u_0 \\ v_0 \end{array} \right.={ \bf Q}+ \boldsymbol \e_0, \ \boldsymbol \e_0 = \left| \begin{array}{l} \e_0 \\ \eta_0 \end{array} \right.\in\mathcal E, \ \ u_0> 0, \ \ v_0>0, \ \ \|\boldsymbol \e_0\|_{\mathcal E}\ll 1$$ such that the corresponding solution $u\in \mathcal C ([0,T),\mathcal E)$ to \fref{kps} with $c=1$ satisfies the following:\\
{\em (i) Small super critical mass}: $$ 8\pi<\int u_0<8\pi+ \alpha^*$$ for some $0<\alpha^*\ll1$ which can be chosen arbitrarily small;\\
{\em (ii) Blow up} : the solution blows up in finite time $0<T<+\infty$;\\
{\em (iii) Universality of the blow up bubble}: the solution admits for all times $t\in [0,T)$ a decomposition 
\be
{\bf u}(t,x) = \left| \begin{array}{l} \frac{1}{\l^2(t)}( Q+ \e)\left(t,\frac{x}{\l(t)}\right) \\
  ( \phi_Q+ \eta)\left(t,\frac{x}{\l(t)}\right)
   \end{array} \right. 
\ee with 
\be
\label{boievbebeo}
\| (\e(t),\eta(t))\|_{W_Q}\to 0\ \ \mbox{as}\ \ t\to T
\ee and the universal blow up speed: 
\be
\label{rate}
\lambda(t)=\sqrt{T-t}e^{-\sqrt{\frac{|\log (T-t)|}{2}}+O(1)}\ \ \mbox{as} \ \ t\to T.
\ee
{\em (iv) Stability}: the above blow up dynamics is stable by small perturbation of the data in $\mathcal E$: $${\bf f_0}= \left|\begin{array}{l} f_0\\ g_0\end{array} \right.f_0> 0, \ g_0>0 \ \ \|{\bf f_0}-{\bf u_0}\|_{\mathcal E}<\epsilon({\bf u_0}).$$
\end{theorem}
{\it Comments on the result}\\ \par
\begin{enumerate}
\item To our knowledge, this is the first proof of the existence of blow-up dynamics in finite time for the parabolic-parabolic Patlak-Keller-Segel in the whole plane. Very recently, Mizogushi and Winkler have obtain the existence of blow-up solution with a virial argument, in the case where $\alpha$, the rate of destruction of the chemoattractant, is non negative. Moreover this kind of obstructive argument does not come with a sharp description of the blow up bubble like \fref{rate}.
\item In this work, for the sake of simplicity, we have chosen $c=1$, where we recall that $c$ is the difference of the time scales of the diffusive processes undergone by u and v. We have seen the crucial rule of this constant in \cite{biler2011} for the existence of self-similar solutions. In fact, the theorem is true for all $c>0$. More precisely, we observe that this quantity doesn't influence on the leading order of the dynamics of the solution. To understand this result, a possible line of reasoning is that, if we fix the constant $c$, we can find initial data satisfying the conditions of the theorem enough concentrated such that the propagation speed of the chemoattractant becomes negligible.
\item Comparing Theorem \ref{thmmain} and the blow up result in \cite{raphael2012b}, we see that the parabolic and elliptic couplings yield {\it in the regime we consider} the same kind of blow up bubble to leading order. Let us stress that this has no reason to hold in general. A celebrated example is the case of the Zakahrov equations of plasma physics: $$\left\{\begin{array}{ll} i\pa_tu+\Delta u-nu=0\\ \frac{1}{c_0^2}\pa_{tt} n-\Delta n=\Delta |u|^2\end{array}\right., \ \ x\in \Bbb R^2$$ which in the limit $c_0\to +\infty$ reduce to the mass critical nonlinear Schr\"odinger equation. Merle proved in \cite{zak} that for all $c_0>+\infty$ (wave coupling), the stable log log regime of the $c_0=+\infty$ case (elliptic coupling) is destroyed, hence showing the importance of the nature of the coupling.
\end{enumerate}
{\bf Aknowledgements.} The author would like to thank A. Blanchet, N. Masmoudi and P. Rapha{\"e}l for their interest and support during the preparation of this work. Part of this work was done while R.S was visiting the UBC Mathematics department  and the Courant Institute which he would like to thank for their kind hospitality. This work is supported by the ERC/ANR grant SWAP, the ERC BLOWDISOL and the junior ERC DISPEQ.\\
\subsection{Notations}
In this problem, we study a couple of solution $(u,v)$. We notice without any difference
\bee
(u,v) = \left| \begin{array}{l}u\\v \end{array}\right. = {\boldsymbol u},
\eee
where the last notation will be used only in the case where there is no possible confusion. In the same way, we notice for an operator $F$:
\bee
F(u,v) = \left| \begin{array}{l}F^{(1)}(u,v)\\F^{(2)}(u,v) \end{array}\right. 
\eee
We use the real $L^2\times\dot H^1$ scalar product :
\bee
\left<\left| \begin{array}{l}u\\v \end{array}\right.,\left| \begin{array}{l}f\\g \end{array}\right.\right> = \int uf + \int \nabla v \nabla g=\int_0^{\infty} u(r)f(r)rdr + \int_0^{\infty} \nabla v(r) \nabla g(r)rdr.
\eee
For a given function $u$, we note its Poisson field $$\phi_u=\frac{1}{2\pi}\log |x|\star u.$$
We let $\chi\in \mathcal C^{\infty}_c(\R)$ be a radially symmetric cut off function with $$\chi(r)=\left\{\begin{array}{ll} 1\ \ \mbox{for}\ \ r\leq 1,\\ 0\ \ \mbox{for}\ \ r\geq 2\end{array}\right., \ \ \chi(r)\geq 0.$$
Moreover, for a given $B>0$, we let 
\bee
\chi_B(r) = \chi\left(\frac{r}{B}\right)
\eee
Given $b>0$, we let 
\be
\label{defB0B1}
B_0=\frac{1}{\sqrt{b}}, \ \ B_1=\frac{|\log b|}{\sqrt{b}}.
\ee
Finally, we use the scaling operator:
\be
\Lambda \left| \begin{array}{l}f\\g \end{array}\right. =\left( \left| \begin{array}{l}\frac{df_\l}{d\l} \\ \frac{d g_\l}{d\l} \end{array}\right.\right)_{\l=1} = \left| \begin{array}{l}2f+y\cdot\nabla f \\ y \cdot \nabla g \end{array}\right.=\left| \begin{array}{l} \nabla\cdot(yf)\\ y \cdot \nabla g \end{array}\right.
\ee \\ \par
The article is organized as follow. In section \ref{jiijij}, we obtain spectral gap estimates for linearized operators close to the ground state. In section \ref{consconscons}, we construt an enough good approximate profiles four our analysis. In section \ref{controlperturbation}, we introduce the energy method based on a bootstrap argument. A keystone of the proof is a monotonicity formula, which we prove in section \ref{lyapounov}. We conclude the bootstrap argument in section \ref{superbootstrap}, and concludes the proof of the Theorem \ref{thmmain} in section \ref{lastsection}.
\section{Spectral gap estimates}
\label{jiijij}
In this section, we obtain spectral gap estimates for linearized operators, which are a keystone of the energy method to control the gap between the solution and the approximate profile, which shall construct in the next section. 
\subsection{On the linearized energy}
In this subsection, we start with the introduction of the operator $\Mm$ coming from the linearization of free energy around the ground state on the suitable space. Then, the structure of this operator is going to be studied, before giving a proposition on its subcoercivity. 
\\ \par
 Let the suitable space $X_Q$ defined by:$$X_Q = \left\{ (u,v) \in L^2_Q \times H^1, \mbox{ s.t. } \int u = 0, \ \mbox{and} \ \int (1+|\log r|)^2|\nabla v|^2<\infty \right\}.$$
 We introduce the suitable norm:
 \be
 \|(u,v)\|_{X_Q} = \| u\|_{L^2_Q} + \| v \|_{\dot H^1}.
 \ee
 From the Lemma \ref{lemmainterpolation}, we know that $\forall (u,v) \in X_Q$, $\nabla \pu \in L^2$. 
Let's prove that the ground state is a local minima of the free energy for deformations in the space $X_Q$.
Let $(u,v) \in X_Q$, and let the function $F$ define for small $\l \in \mathbb R$ such that $ Q + \l u >0$ by:
\bee
F(\l) &=& E(Q+\l u, \phi_Q + \l v) \\
&=& \int (Q+\l u) \log (Q + \l u) + \int (Q + \l u)(\phi_Q + \l v) - \frac 12 \int (\phi_Q + \l v) \Delta (\phi_Q + \l v)
\eee
We compute
\bee
F'(\l) &=& \int u [ \log (Q + \l u) + 1] + \int u(\phi_Q + \l v) +  \int (Q + \l u)v \\&-& \frac 12 \int v \Delta (\phi_Q + \l v) - \frac 12 \int (\phi_Q + \l v) \Delta v \\
&=& \int u [ \log (Q + \l u) - \log 8 + \phi_Q] +  \frac 12 \int (vQ - \Delta v \phi_Q) + \l  \int \left \{ 2uv - v \Delta v \right\}
\eee
Hence, using the explicit expression of $Q$ and $\phi_Q$, we obtain
\bee
F'(0) =  \frac 12 \int (vQ - \Delta v \phi_Q).
\eee
We shall prove that the integration by part is valid for $(u,v) \in X_Q$. In this purpose, let $R_0>0$. We compute
\bee
\int_{R_0}^{R_0+1} \left\{\int_0^R \Delta v\phi_Q \right\}dR = \int_{R_0}^{R_0+1} \left\{\int_0^R v Q \right\}dR+\int_{R_0}^{R_0+1} \left[\phi_Q r v'(r) - \nabla \phi_Q v(r)r \right]^{R}_0 dR.
\eee
Now, as $(u,v) \in X_Q$, $v\in L^2$ and $\int (1+|\log r|)^2|\nabla v|^2<\infty$ yield
\bee
\lim\limits_{R_0 \rightarrow +\infty} \int_{R_0}^{R_0+1} \left[\phi_Q r v'(r) - \nabla \phi_Q v(r)r \right]^{R}_0 dR = 0.
\eee
Thus : $\int vQ = \int \Delta v \phi_Q$, and thus Q is a critical point: $$F'(0)=0.$$
We can compute the Hessian, using that $(\nabla \pu, \nabla v)\in L^2 \times L^2$:
\bee
F''(\l)_{|\l=0} &=& \int \frac {u^2}Q + 2 uv - \int v \Delta v \\
&=&  \int \left(\frac {u^2}Q + u\pu + u(v-\pu) + uv - v \Delta v \right)\\
&=& \int \left(\frac {u^2}Q + u\pu \right) + \int |\nabla v - \nabla \pu|^2 \\
&=& \int \left(\frac {u^2}Q + vu \right) + \int \nabla (v-\phi_u) \nabla v = \left< \mathcal M \   \begin{array} {|c} u \\v \end{array},  \begin{array} {c} u \\v \end{array}  \right> .
\eee 
In \cite{raphael2012b} is proved that for a function $u$ such that $\int u=0$,  $\int \left(\frac {u^2}Q + u\pu \right) \geq 0$. Then, from the third line of the above compute, $F''(\l)_{|\l=0} \geq 0$ and thus $(Q,\phi_Q)$ is a local minima of the free energy for deformations in the space $X_Q$. Furthermore, the operator $\mathcal M$ define by
\be
\label{defoperatorM}
\mathcal M : \left| \begin{array} {c} \frac{u}{Q} + v \\ v - \pu \end{array}\right. = \left| \begin{array} {c} \mathcal M^{(1)}(u,v) \\  \mathcal M^{(2)}(u,v) \end{array}\right.
\ee 
is a positive operator on $X_Q$. The following lemma describes the structure of this operator.
\begin{lemma}[Structure of the linearized energy]
\label{lemmakernel}
Let the operator 
\be
\label{defm}
\mathcal M : \left| \begin{array} {c} \frac{u}{Q} + v \\ v - \pu \end{array}\right.
\ee  Then:
\be
\label{mcontiniuos}
\int Q|\mathcal M^{(1)}(u,v)|^2 + \int |\nabla \mathcal M^{(2)}(u,v) |^2\lesssim \|(u,v)\|_{X_Q}^2 + \int \frac{v^2}{1+y^4},
\ee
Moreover, there holds:\\
{\it (i) Sefl-adjointness}:
\be
\label{efadjointness}
\forall ({\bf u},{\bf v})\in L^2_Q\times H^1, \ \ \left<\mathcal M {\bf u},{\bf v}\right>=\left<{\bf u},\mathcal M{\bf v}\right>.
\ee
{\it (ii) Algebraic identities}:
\be
\label{relationsm}
\mathcal M({\bf \Lambda Q} )= \mathcal M \   \begin{array} {|c} \Lambda Q \\ \phi_{\Lambda Q} \end{array}= \begin{array} {|c} -2 \\0 \end{array}
\ee
{\it (iii) Generalized kernel}: Let $(u,v) \in X_Q$ such that :
\bee \left\{\begin{array} {c}
\nabla \left( \frac uQ + v\right) =0 \\
\Delta v = u
\end{array}
\right.
\eee
Then 
\be
\label{identificationkernel}
{\bf u} \in \mbox{Span}({\bf \Lambda Q})
\ee
\end{lemma}
\begin{proof}
Let $(u,v) \in X_Q$. The continuity of this operator follows from the Lemma \ref{lemmainterpolation}:
\bee
\left|\left< \mathcal M {\bf u}, \mathcal M {\bf u} \right> \right|&=& \left | \int  Q \left( \frac u Q + v \right)^2 + \int (\nabla v - \nabla \phi_u)^2 \right | \\
&\lesssim& \int  \frac {u^2} Q + \int Q v^2 + \int |\nabla v|^2 + \int |\nabla \phi_u|^2 \\
&\lesssim&  \int  \frac {u^2} Q + \int |\nabla v|^2 + \int \frac{v^2}{1+y^4}.
\eee

To prove the self-adjointness of the operator $\mathcal M$, let $\left\{ (u,v) ; (f,g)\right \} \in (L^2_Q\times H^1)^2$, and we compute:
\bee
\left< \mathcal M \   \begin{array} {|c} u \\v \end{array},  \begin{array} {c} f \\g \end{array}  \right> &=& \int \left(\frac uQ + v \right)f + \int |\nabla v - \nabla \phi_u| \nabla g \\
&=& \int \frac{uf}{Q} + \int \nabla v \nabla g + \int vf + \int gu. \\
&=& \left<  \begin{array} {c} u \\v \end{array}, \  \mathcal M \  \begin{array} {|c} f \\g \end{array}  \right>
\eee
The last equality comes from $\int f = \int u = 0$. \fref{relationsm} is a consequence of the explicit expression for $\Lambda Q$ and $\phi_{\Lambda Q}$. 
\par
Let's prove \fref{identificationkernel}. For this purpose, let $(u,v) \in X_Q$ such that:
\bee \left\{\begin{array} {c}
\nabla \left( \frac uQ + v\right) =0 \\
\Delta v = u
\end{array}
\right.
\eee
From the Lemma \ref{phideltav}, $(u,v)\in X_Q$ yields $\phi_{\Delta v} = v$ and thus 
\be
\phi_u = v.
\ee
Prove \fref{identificationkernel} is equivalent to find the functions $u \in L^2_Q$ such that :
\be
\label{dola}
\nabla \left( \frac u Q + \phi_u\right) = 0.
\ee
In the framework of radial functions, let the partial mass be :
\be
m_u(r) = \int_0^r u(\tau) \tau d\tau.
\ee
Remark that $\phi'_u(r) = \frac{m_u(r)}{r}$ and $u(r) = \frac{m'_u(r)}{r}$. Hence,
\bee
 \left( \frac u Q + \phi_u\right)' = \frac{u'(r)}{Q} - \frac{Q'}{Q^2} u(r) + \phi'_u(r) =  \frac{1}{Q}\left( \frac{m_u(r)}{r}\right)' - \frac{Q'}{Q^2}  \frac{m'_u(r)}{r} +  \frac{m_u(r)}{r} = 0
\eee
Thus, the equation \fref{dola} becomes:
\bee
L_0m_u = 0
\eee
where
\bea
\label{defmo}
L_0 m_u & = &  -m_u''+\left(\frac{1}{r}+\frac{Q'}{Q}\right)m_u'-Qm_u\\
\nonumber& = & -m_u''-\frac{3r^2-1}{r(1+r^2)}m_u'-\frac{8}{(1+r^2)^2}.
\eea
The basis of solutions to this homogeneous equation is explicit and given by:
 \be
 \label{basislzero}
 \psi_0(r)=\frac{r^2}{(1+r^2)^2}, \ \ \psi_1(r)=\frac{1}{(1+r^2)^2}\left[r^4+4r^2\log r-1\right].
 \ee
The regularity of $u$ at the origin, and thus the regularity of $m_u$ implies that $m_u \in Span \left( \psi_0 \right)$. To remark that $m_{\Lambda Q} = 8 \psi_0$ conclude the proof of \fref{identificationkernel} and hence the proof of the Lemma \ref{lemmakernel}.
\end{proof}
Before studying the linearized operator $\mathcal L$ close to the ground state of the (PKS) flow, we prove sub-coercivity for the operator $\mathcal M$ which are the key to the proof of coercive estimates for these operators under additional orthogonality conditions.
\begin{proposition}[Sub-coercivity of the operator $\mathcal M$]
\label{propenergy}
There exists a universal constant $\delta_0>0$ such that for all ${\bf u}\in X_Q$,
\be
\label{corec}
\left<\mathcal M {\bf u},{\bf u}\right>\geq \delta_0\left(\int\frac{u^2}{Q} + \int |\nabla v| ^2 \right)-\frac{1}{\delta_0}\left< {\bf u},{\bf \Lambda Q}\right>^2.
\ee
\end{proposition}
\begin{proof}
{\bf Step 1 : }Coercivity \\ \\
To begin,  prove that :
\be
\label{cneonoen}
I=\inf\left\{\left<\mathcal M {\bf u},{\bf u}\right>, \ \  \|{\bf u} \|_{X_Q}=1, \ \ \left < {\bf u},{\bf \Lambda Q} \right > =0\right\}>0.
\ee
We argue by contradiction. Let a sequence $\begin{array}{|c}u_n \\v_n \end{array} =  {\bf u}_n \in X_Q$ such that:
$$0\leq \left<\mathcal M {\bf u}_n,{\bf u}_n\right>\leq \frac 1n, \ \ \|{\bf u}_n\|_{X_Q}=1, \ \ \left< {\bf u}_n,{\bf \Lambda Q}\right>=0.$$ 
We recall that:
\bee
\left<\mathcal M {\bf u}_n,{\bf u}_n\right> = (\mathcal M_0 u_n,u_n) + \int |\nabla v_n - \nabla \phi_{u_n}|^2.
\eee
Up to a subsequence, 
\be
\label{weakconv}
{\bf u}_n\rightharpoonup {\bf u} = \begin{array}{|c}u \\v \end{array} \ \ \mbox{in}\ \ X_Q
\ee
Hence,
\be
\label{wmpo}
\int \frac{u^2}{Q} + \int |\nabla v|^2 \leq \liminf_{n\to +\infty} \int \frac{u_n^2}{Q} + \int |\nabla v_n|^2 = 1 \  \mbox{and} \ \left < {\bf u},{\bf \Lambda Q} \right > =0
\ee
From standard argument (see for example the proof of the Proposition 2.3 of \cite{raphael2012b}) :
\bee
\nabla \phi_{u_n}\to \nabla \phi_{u} \ \mbox{in}\ \ L^2.
\eee
As
\bee
\nabla v_n\rightharpoonup \nabla v  \ \mbox{in}\ \ L^2,
\eee
we obtain :
\be
\int u_nv_n = - \int \nabla \phi_{u_n} \nabla v_n \to  - \int \nabla \phi_{u} \nabla v = \int uv
\ee
The positivity of the operator $\mathcal M$ together with \fref{wmpo} yields
\be
\left < \mathcal M {\bf u},{\bf u}\right> = 0, \ \ \int \frac{u^2}{Q} + \int |\nabla v|^2 = - 2 \int uv \leq 1.
\ee 
From the normalization of the sequence we have:
\be
-2\int u_nv_n = \int \frac{u_n^2}{Q} + \int |\nabla v_n|^2 - \left<\mathcal M {\bf u}_n,{\bf u}_n\right> \geq 1 - \frac 1n.
\ee
Thus, the function ${\bf u}$ verifies:
\be
\left<\mathcal M {\bf u},{\bf u}\right> = 0, \ \ \|{\bf u} \|_{X_Q}=1 \  \mbox{and}\ \left < {\bf u},{\bf \Lambda Q} \right > =0 
\ee
From the Lemma \ref{lemmakernel}, the condition $\left<\mathcal M {\bf u},{\bf u}\right> = 0$ yields that ${\bf u} = c {\bf \Lambda Q}$, and the orthogonality condition $ \left < {\bf u},{\bf \Lambda Q} \right > =0$ impose that $c=0$. Hence, $u=0$ which contradicts $ \|{\bf u} \|_{X_Q}=1$. This concludes the proof of \fref{cneonoen}.
\\ \\
{\bf Step 2 : }Conclusion \\ \\
Consider ${\bf u} \in X_Q$. Let 
\be
{\bf v} = {\bf u} - \frac{\left < {\bf u},{\bf \Lambda Q}\right>}{\left <{\bf \Lambda Q},{\bf \Lambda Q}\right>} {\bf \Lambda Q}
\ee
By construction, $\left < {\bf v},{\bf \Lambda Q}\right> = 0$ and thus from \fref{cneonoen}
$$\left<\mathcal M {\bf v},{\bf v}\right>\geq \delta_0 \| {\bf v}\|^2_{X_Q}\geq \delta_0\| {\bf u}\|^2_{X_Q}-\frac{1}{\delta _0}\left < {\bf u},{\bf \Lambda Q}\right>.$$
This concludes the proof of the Proposition \ref{propenergy}.
\end{proof}
\subsection{On the linearized operator $\mathcal L$}
In this section, we begin to the study of the structure of the linearized operator close to $Q$ of the (PKS) flow for perturbations in the energy space $\mathcal E$. Moreover the operator $\mathcal L$ is given by:
\bee
\mathcal L (\e,\eta) = \begin{array}{|c} \nabla. \left\{ Q \nabla \mathcal M^{(1)}(\e,\eta) \right\}  \\ \Delta \mathcal M^{(2)}(\e,\eta) \end{array} = \begin{array}{|c} \nabla . \left\{Q \nabla \left( \frac \e Q +\eta\right)\right\}  \\ \Delta \eta - \e \end{array} 
\eee
We can formally define its adjoint for the $L^2 \times \dot H^1$ scalar product by:
\be
\label{ladjoint}
\mathcal L^* (\e,\eta) = \mathcal M \  \begin{array}{|c} \nabla. \left\{ Q \nabla \e \right\}  \\ \Delta \eta \end{array} = \begin{array}{|c} \frac{ \nabla. \left\{ Q \nabla \e \right\} }{Q} + \Delta \eta  \\ \Delta \eta - \phi_{ \nabla. \left\{ Q \nabla \e \right\} } \end{array} 
\ee
\begin{lemma}[Structure of the operator $\mathcal L$] 
\label{lemmacontiniuty}
{\it (i)} Continuity of $\mathcal L$ on $\mathcal E$: 
\be
\label{upperbound}
\|\mathcal L\boldsymbol{\e}\|_{X_Q}\lesssim \| \boldsymbol{\e}\|_{\mathcal E}.
\ee
{\it (ii)} Adjunction: $\forall (\boldsymbol{\e}, \boldsymbol{\tilde \e})\in \mathcal E^2$,
\be
\label{biebibvei}
 \left<\mathcal L\boldsymbol{\e}, \boldsymbol{\tilde \e}\right>= \left<\boldsymbol{\e},\mathcal L^* \boldsymbol{\tilde \e}\right>.
\ee
 {\it (iii) Algebraic identities}: 
  \be
 \label{cneneovneo}
 \mathcal L({ \bf \Lambda Q} )=0,
 \ee
 \be
 \label{nveovneneon}
\forall c \in \mathbb R, \ \mathcal L^*(1,cr^2)=0, \ \ \mathcal L^*\left(r^2, -4\int_0^r \frac{\log (1+\tau^2)d\tau}{\tau}\right)=\begin{array} {|c} -4\\ 0 \end{array}.
 \ee
 {\em (iv) Vanishing average}: $\forall \boldsymbol{\e} \in\mathcal E$,
 \be
 \label{estvooee}
 \left<\mathcal L\boldsymbol{\e}, \begin{array} {|c} 1\\ c \end{array} \right>=0, \ \ \forall c \in \mathbb R
 \ee
 \end{lemma}
\begin{proof}
{\bf Step 1 :} Continuity :
\\ \\
First, we rewrite the operator $\mathcal L$  with an explicit formula:
\bee
\mathcal L (\e,\eta) = \begin{array}{|c} \Delta \e + \e Q + \nabla \e \nabla \phi_Q + Q \Delta \eta + \nabla Q \nabla \eta \\ \Delta \eta - \e \end{array}.
\eee
Hence,
\bee
&&\| \mathcal L (\e,\eta)\|_{X_Q}^2 \\
& \lesssim & \int\frac{|\Delta \e|^2}{Q}+\int Q|\e|^2+\int\frac{|\nabla \phi_Q|^2}{Q}|\nabla \e|^2+\int \frac{|\nabla Q|^2}{Q}|\nabla \eta|^2+ \int Q (\Delta \eta)^2 + \int |\nabla (\Delta \eta)|^2\\
& \lesssim & \|\e\|_{H^2_Q}^2+ \int |\nabla(\Delta \eta)|^2 + \int \frac{|\Delta \eta|^2}{1+r^4} + \int \frac{|\nabla \eta|^2}{1+r^6} \lesssim \|(\e,\eta)\|_{\mathcal E}^2,
\eee
and the continuity is proved.\\ \\
{\bf Step 2 :} Adjunction :
\\ \\
To prove rigorously the formal adjoint \fref{ladjoint}, and thus \fref{biebibvei}, we must justify now the integration by parts. For this purpose, let's begin to prove that both integrals are absolutely convergent. First :
\bee
\left|  \left<\mathcal L\boldsymbol{\e}, \boldsymbol{\tilde \e}\right> \right| &\lesssim& \int |\mathcal L^{(1)}(\e,\eta)|| \et| + \int |\pa_r (\Delta \eta - \e)||\pa_r \tilde \eta| \\
&\lesssim& \|\mathcal L^{(1)}(\e,\eta)\|_{L^2_Q}  \|\et\|_{L^2} +  \|\mathcal L^{(2)}(\e,\eta)\|_{\dot H^1}  \|\tilde \eta\|_{\dot H^1} \\
&\lesssim& \|(\e,\eta)\|_{\mathcal E} \|(\tilde \e, \tilde \eta)\|_{L^2\times \dot H^1}
\eee
Now, using that $\nabla \phi_{\nabla . u} = u$, we obtain using Cauchy-Schwarz
\bee
\left|  \left<\boldsymbol{\e},\mathcal L^* \boldsymbol{\tilde \e}\right> \right| &=& \left| \int \e \left( \frac{ \nabla. \left\{ Q \nabla \et \right\} }{Q} + \Delta \tilde \eta\right) + \int \nabla \eta \nabla \left( \Delta \tilde \eta - \phi_{ \nabla. \left\{ Q \nabla \et \right\} }\right) \right| \\
&\lesssim& \|\e\|_{L^2}\|\Delta \tilde \eta \|_{L^2} + \int |\e| \left[ |\Delta \et| + \frac{|\nabla Q|}{Q}|\nabla \et|\right] + \int \nabla \eta (\nabla \Delta \tilde \eta - \tilde \e)\\
&\lesssim&  \|(\e,\eta)\|_{\mathcal E} \|(\et,\tilde\eta)\|_{\mathcal E}
\eee
The integrals being absolutely convergent, we have thus :
\be
\left< \mathcal L \boldsymbol{\e},  \boldsymbol{\tilde \e}\right > = \lim\limits_{R \rightarrow +\infty} \int_0^R\left\{ \mathcal L^{(1)}(\e,\eta) \et +\nabla \mathcal L^{(2)}(\e,\eta) \nabla \tilde \eta  \right\}
\ee
Using the radial coordinates, we can rewrite the last integral by:
\bea 
 \nonumber&&\int_0^R\left\{ \mathcal L^{(1)}(\e,\eta) \et +\nabla \mathcal L^{(2)}(\e,\eta) \nabla \tilde \eta  \right\} \\
 \nonumber&=& \int_0^R \left\{ \pa_r \left(r Q \pa_r \mathcal M^{(1)}(\e,\eta) \right)\et+ r \pa_r \left[\frac 1r \pa_r \left( r \pa_r \mathcal M^{(2)}(\e,\eta)   \right) \right]\pa_r\tilde\eta   \right\}dr \\
 \nonumber&=& - \int_0^R \left\{ r Q \pa_r \mathcal M^{(1)}(\e,\eta) \pa_r \et+ \pa_r \left( r \pa_r \mathcal M^{(2)}(\e,\eta)   \right)\frac 1r \pa_r \left(r \pa_r\tilde\eta\right)   \right\}dr \\ 
 \nonumber&+& \left[ r Q \et\pa_r \mathcal M^{(1)}(\e,\eta) + \pa_r \left( r \pa_r \mathcal M^{(2)}(\e,\eta)   \right) \pa_r\tilde\eta\right]_0^R \\
\label{pppllkkjjuu1}&=&  \int_0^R \left\{   \mathcal M^{(1)}(\e,\eta)\frac 1r\pa_r \left(rQ \pa_r \et \right)+   \pa_r \mathcal M^{(2)}(\e,\eta)  \pa_r \left [ \frac 1r \pa_r \left(r \pa_r\tilde\eta\right)  \right] \right\}rdr \\ 
\label{pppllkkjjuu2}&+& \left[ r Q \et\pa_r \mathcal M^{(1)}(\e,\eta) + \pa_r \left( r \pa_r \mathcal M^{(2)}(\e,\eta)   \right) \pa_r\tilde\eta\right]_0^R\\
\label{pppllkkjjuu3}&-&  \left[ r Q \mathcal M^{(1)}(\e,\eta) \pa_r \et + \pa_r \mathcal M^{(2)}(\e,\eta) \pa_r \left( r \pa_r\tilde\eta  \right) \right]_0^R
\eea
Using the smoothness of $(\e,\eta)$ and of $(\et, \tilde \eta)$, the terms \fref{pppllkkjjuu2} and \fref{pppllkkjjuu3} cancel at the origin. Now, there exists a sequence $R_n \rightarrow + \infty$ such that:
\label{nhytgb}\bea&& \left[ r Q \et\pa_r \mathcal M^{(1)}(\e,\eta) + \pa_r \left( r \pa_r \mathcal M^{(2)}(\e,\eta)   \right) \pa_r\tilde\eta\right]_0^{R_n}\\
\nonumber&-&  \left[ r Q \mathcal M^{(1)}(\e,\eta) \pa_r \et + \pa_r \mathcal M^{(2)}(\e,\eta) \pa_r \left( r \pa_r\tilde\eta  \right) \right]_0^{R_n} \underset{n \rightarrow +\infty}{\rightarrow} 0.
\eea
Indeed, we estimate from Cauchy-Schwarz and the Hardy bound \fref{harfylog} :
\bee
\left|\int Q \pa_r \et \left( \frac \e Q + \eta \right)\right| &\lesssim& \left|\int \e \pa_r \et\right| +  \left|\int Q \eta \pa_r \et\right| \\
&\lesssim& \|\e\|_{L^2}\|\pa_r \et\|_{L^2} +  \left|\int Q \eta^2\right |^{\frac 12} \left |\int Q| \pa_r \et|^2\right|^{\frac 12} \\
&\lesssim&  \|\e\|_{L^2}\|\pa_r \et\|_{L^2} + \left\|\pa_r \eta \right\|_{L^2} \left\|\pa_r \et  \right\|_{L^2} < + \infty
\eee
Moreover :
\bee
\left|\int Q \et \pa_r \left( \frac \e Q + \eta \right)\right| &\lesssim& \left|\int \et \left(\pa_r \e + \e \phi_Q + Q\pa_r \eta \right)\right| \\
&\lesssim&  \|\et\|_{L^2} \left( \left\|\frac{\pa_r \e}{1+r}  \right\|_{L^2_Q} +  \|\e\|_{L^2} +\left \| \frac{\pa_r \eta}{1+y^4} \right\|_{L^2}  \right)
\eee
Now, with Cauchy-Schwarz
\bee
\left|\int \pa_r \tilde \eta \left(\Delta \eta - \e \right)\right| &\lesssim& \| \pa_r \tilde \eta \|_{L^2} \left( \|\e \|_{L^{2}} +  \|\Delta \eta \|_{L^{2}} \right)
\eee
Finally,
\bee
\left|\int \Delta \tilde \eta \left(\pa_r \eta -\pa_r \phi_\e \right)\right| &\lesssim& (\| \pa_r \eta \|_{L^{\infty}} + \| \pa_r \phi_\e \|_{L^{\infty}} ) \|\Delta \tilde \eta \|_{L^{1}} \lesssim (\| \Delta \eta \|_{L^{2}} + \| \e \|_{L^{2}} ) \|(\et, \tilde \eta) \|_{\mathcal E} 
\eee
This concludes the proof of \fref{nhytgb}. We have proved for the moment that:
\bee
 \left<\mathcal L \ \begin{array}{|c} \e\\ \eta  \end{array}, \begin{array}{|c} \et\\ \tilde \eta  \end{array}\right> =  \left<\mathcal M \ \begin{array}{|c} \e\\ \eta  \end{array}, \begin{array}{|c} \nabla. \left\{ Q \nabla \et \right\}  \\ \Delta \tilde \eta \end{array} \right> 
\eee
With the same method, the proof of the self-adjointness of $\mathcal M$ here takes any difficulty, and is left to the reader. This yields \fref{biebibvei}. 
\\ \\
{\bf Step 3 :} Algebraic identities.
\\ \\
The identity \fref{relationsm} yields directly \fref{cneneovneo}. Now, $\forall c \in \mathbb R$
\be
\mathcal L^* (1,c) = \mathcal M \  \begin{array}{|c} \nabla. \left\{ Q \nabla 1 \right\}  \\ \Delta c \end{array} =  \begin{array}{|c} 0\\0 \end{array}.
\ee
For the last algebraic identity, we use:
\bee
\frac1r \pa_r \left( r \pa_r \left[  -2\int_0^r \frac{\log (1+\tau^2)d\tau}{\tau}\right]\right) = \phi_{\Lambda Q}.
\eee
Thus,
\bee
 \mathcal L^*\left(r^2, -4\int_0^r \frac{\log (1+\tau^2)d\tau}{\tau}\right) =  \mathcal M \  \begin{array}{|c} \nabla. \left\{ Q \nabla r^2 \right\}  \\ \Delta \left( -2\int_0^r \frac{\log (1+\tau^2)d\tau}{\tau} \right)  \end{array} = \mathcal M \  \begin{array}{|c} 2 \Lambda Q  \\ 2 \phi_{\Lambda Q}  \end{array} =   \begin{array}{|c} -4\\0  \end{array}.
\eee
\\ \\
{\bf Step 4 :} Vanishing average
\\ \\
Remark that the integral \fref{estvooee} is absolutely convergent. Indeed, let $c \in \mathbb R$.
\bee
\left| \left<\mathcal L\boldsymbol{\e}, \begin{array} {|c} 1\\ c \end{array} \right>\right| \lesssim \left|\int \mathcal L^{(1)}(\e,\eta)\right| \lesssim  \int \frac{\mathcal L^{(1)}(\e,\eta)^2}{Q} \lesssim \|(\e,  \eta) \|^2_{\mathcal E}
 \eee
Hence,
\bee
\int \mathcal L^{(1)}(\e,\eta) =\lim\limits_{R \rightarrow +\infty} \int_0^R \mathcal L^{(1)}(\e,\eta)
\eee
But:
\bee
\int_0^R \mathcal L^{(1)}(\e,\eta) = \left[ r Q \pa_r \left( \frac{\e}{Q} + \eta \right)\right]_0^R
\eee
The last term cancels at the origin. Now, using that:
\bee
\int \left|Q \pa_r \left( \frac{\e}{Q} + \eta \right)\right|^2 \lesssim \int \left[ |\pa_r \e| + |\e|^2\phi_Q^2 + Q |\pa_r \eta|^2 \right] < +\infty,
 \eee
 there exists a sequence $R_n \rightarrow +\infty$ such that:
 \bee
\left\{ Q \pa_r \left( \frac{\e}{Q} + \eta \right) \right\}(R_n) = o \left( \frac{1}{R^2_n}\right)
 \eee
Thus,
 \bee
\int \mathcal L^{(1)}(\e,\eta) =\lim\limits_{R_n \rightarrow +\infty} \int_0^{R_n} \mathcal L^{(1)}(\e,\eta) = R_n \left\{ Q \pa_r \left( \frac{\e}{Q} + \eta \right) \right\}(R_n) \rightarrow 0.
\eee
This concludes the proof of the Lemma \ref{lemmacontiniuty}.
 \end{proof}
In the last Lemma, we have exhibited the kernel of the operator $\mathcal L^*$, and we have seen that the elements of this kernel have irrevelant growth. Therefore, we shall introduce in the following Lemma an enough good approximation of this kernel, defining both directions $\boldsymbol{\Phi}_M$ and $\mathcal L^* \boldsymbol{\Phi}_M$. It is a enough good approximation in the sense that we prove in the Proposition \ref{interpolationhtwo} that if $\boldsymbol \e$ is orthogonal to this directions, then the operator $\mathcal L$ is coercive. \par
Moreover, we must anticipate the construction of the approximate profile, and we define ${\bf T}_1= \left| \begin{array}{l}T_1\\S_1 \end{array}\right.$ such that $\mathcal L {\bf T}_1 = {\bf \Lambda Q}$. Thus, we have the following bounds:
\bee
T_1(r) \lesssim \frac{1}{1+r^2}, \ \ \pa_r S_1(r) \lesssim \frac{r}{1+r^2}.
\eee
 \begin{lemma}[On the direction $\boldsymbol{\Phi}_M$]
 \label{direcitonroth}
 Given $M\geq M_0>1$ large enough, we define the directions: $\Pi$
 \be
 \label{defphimzero}
{\bf \Phi}_{0,M} = \begin{array} {|c} \Phi_{0,M} \\  \Pi_{0,M} \end{array}=\begin{array} {|c}\chi_Mr^2 \\  -4 \int_0^r \frac{\log{(1+\tau^2)}}{\tau} \chi_M d\tau \end{array}, 
 \ee
 \be
\label{defphim}
{\bf \Phi}_M(y)={\bf \Phi}_{0,M} +c_M \mathcal L^*({\bf \Phi}_{0,M}), \ \ c_M=-\frac{\left<{\bf \Phi}_{0,M},{\bf T_1}\right>}{\left< {\bf \Phi}_{0,M},{\bf \Lambda Q}\right>},
\ee
Then:\\
{\it (i) Estimate on ${\bf \Phi}_{M}$}:
\be
\label{estphim}
{\bf \Phi_M(r)}={\bf \Phi}_{0,M} -\chi_M \begin{array}{|c}4c_M\\ 0 \end{array}+ \frac{M^2}{\log M}O\left( \ \begin{array}{|c} {\bf 1}_{M\leq r\leq 2M} \\\frac{1}{1+r^2}{\bf 1}_{r\leq 2M}+  \frac{\log M}{M^2}{\bf 1}_{M\leq r\leq 2M}\end{array}\right) ,
\ee
\be
\label{orthophim}
\left<{\bf \Phi}_{M},{\bf T_1}\right>=0, \ \ \left< {\bf \Phi}_{M},{\bf \Lambda Q}\right>=-(32\pi) \log M+O_{M\to +\infty(1)},
\ee
{\it (ii) Estimate on scalar products}: $\forall(\e,\eta)$ such that $(\e, \nabla \eta) \in L^1\times L^1\left(\frac{dr}{1+r^3} \right)$, 
\be
\label{estimationorthobis}
|\left<\boldsymbol \e,\bf \Phi_M \right>| \lesssim  \int_{r\leq 2M} \left\{(1+r^2)|\e|+ \frac{|\pa_r \eta| \log(1+r^2)}{r} \right\}+\frac{M^2}{\log M}\left[|(\e,1)|+\int_{r\geq M}|\e|\right],
\ee
\be
\label{bjebbeibei}
|\left<\boldsymbol \e,\mathcal L^*\bf\Phi_{0,M}\right>|\lesssim \int_{r\leq 2M}|\e| + \log M \int_{M\leq r\leq 2M}\frac{|\pa_r \eta|}{1+r^3} ,
\ee
\be
\label{newestimate}
\left|(\e,\mathcal L^*\Phi_M)\right|\lesssim  \int_{r\leq 2M}|\e|+ \log M \int_{M\leq r\leq 2M}\frac{|\pa_r \eta|}{1+r^3} + \frac{M^2}{\log M}\int_{r\geq M} \frac{|\e|}{1+r^2},
\ee
\be
\label{estfonamentalebus}
 \left|(\e,\mathcal L^*\Phi_M)+4(\e,1)\right|+\left|(\e,\mathcal L^*\Phi_{0,M})+4 (\e,1)\right|\lesssim \int_{r\geq M}|\e|+ \log M \int_{M\leq r\leq 2M}\frac{|\pa_r \eta|}{1+r^3}
\ee
{\it (ii) Rough bounds}: if moreover ${\boldsymbol \e} \in X_Q$:
\be
\label{firstroughbound}
|\left<\boldsymbol \e,\bf \Phi_M \right>|  + |\left<\boldsymbol \e,\bf \Phi_{0,M} \right>| \lesssim M \|\boldsymbol \e\|_{X_Q}^2 + \frac{M^2}{\log M} \left| \int \e \right| 
\ee
\be
\label{secondroughbound}
\left|(\e,\mathcal L^*\Phi_M)\right| \lesssim M \left( \|\e\|_{L^2} +\left \| \frac{\pa_r \eta}{1+r^3} \right\|_{L^2} \right).
\ee
\end{lemma}
\begin{proof}
{\bf Step 1 :} Proof of estimate on ${\bf \Phi}_{M}$.
\\ \\
We begin to compute $\mathcal L^*({\bf \Phi}_{0,M})$. We let:
\be
\mathcal L^* \  \begin{array} {|c} \Phi_{0,M} \\  \Pi_{0,M} \end{array} = \mathcal M \  \begin{array} {|c} r_{1,M} \\  r_{2,M} \end{array}, \ \mbox{with} Ê\ \begin{array} {|c} r_{1,M} \\  r_{2,M} \end{array} = \begin{array} {|c} \nabla.(Q \nabla (\Phi_{0,M})) \\  \Delta \Pi_{0,M} \end{array}.
\ee
Hence,
\be
\label{calculr1r2}
 \begin{array} {|c} r_{1,M} \\  r_{2,M} \end{array} = 2{\bf  \Lambda Q} \chi_M +  \begin{array} {|c} (5rQ + r^2Q')\chi'_M + r^2Q\chi_M'' \\  -4\frac{\log{(1+r^2)}}{r}\chi'_M \end{array}
\ee
Moreover, as $\int r_{1,M} = \int \nabla.(Q \nabla (\Phi_{0,M})) = 0$, we have the following expression for the Poisson field of $r_{1,M}$:
\be
\label{champr1}
\left|\phi_{r_{1,M}}(r)\right|=\left|\int_{r}^{+\infty}Q\pa_r(\chi_Mr^2)dr\right|\lesssim\frac{1}{1+r^2} {\bf 1}_{r\leq 2M}.
\ee
By definition
\bee
\mathcal L^* \  \begin{array} {|c} \Phi_{0,M} \\  \Pi_{0,M} \end{array} =  \begin{array} {|c} \frac{r_{1,M}}{Q} + r_{2,M} \\  r_{2,M} - \phi_{r_{1,M}} \end{array}, 
\eee
and thus, the decay of $Q(r) \lesssim \frac{1}{1+r^4}$ together \fref{calculr1r2} and \fref{champr1} yield:
\be
\label{lamande}
\mathcal L^* \  \begin{array} {|c} \Phi_{0,M} \\  \Pi_{0,M} \end{array} =  \begin{array} {|c} -4\chi_M + O({\bf 1}_{M\leq r \leq 2M}) \\  \frac{1}{1+r^2}O({\bf 1}_{r \leq 2M}) + \frac{\log M}{M^2}O({\bf 1}_{M\leq r \leq 2M})   \end{array}.
\ee
Moreover we have the following cancellation:
\be
\label{nolwenn}
\left| \pa_r \left(( \mathcal L^*)^{(2)} (\Phi_{0,M} , \Pi_{0,M})\right)\right| = \left| \pa_r (r_{2,M}) - \pa_r(\phi_{r_{1,M}}) \right| \lesssim \frac{\log M}{M^3}{\bf 1}_{M \leq y \leq 2M}
\ee
Indeed:
\bee
\pa_r( r_{2,M}) = 2 \pa_r (\phi_{\Lambda Q}) \chi_M + O \left( \frac{\log M}{M^3}{\bf 1}_{M \leq y \leq 2M}\right)
\eee
Moreover
\bee
\pa_r(\phi_{r_{1,M}}) = Q \pa_r (r^2 \chi_M) = 2 rQ \chi_M + O \left( \frac{1}{M^3}{\bf 1}_{M \leq y \leq 2M}\right)
\eee
The identity $ \pa_r (\phi_{\Lambda Q}) = rQ$ concludes the proof of \fref{nolwenn}.
Now, remark that ${\bf \Lambda Q} \in \mathcal E$, and using the Lemma \fref{lemmacontiniuty}, we obtain
\bee
\left< \mathcal L^*{\bf \Phi_{0,M}}, {\bf \Lambda Q}\right> = \left<{\bf \Phi_{0,M}},  \mathcal L {\bf \Lambda Q}\right> =0
\eee
and so
\bee
\left< \bf \Phi_{0,M}, {\bf \Lambda Q}\right> = \left< \bf \Phi_{M}, {\bf \Lambda Q}\right>.
\eee
Now:
\bee
&&\int \chi_M r^2 \Lambda Q(r) = 2\pi \int \chi_M r^2 \pa_r \left( r^2 Q \right) dr = - 2\pi \int r^2Q \pa_r \left(  \chi_M r^2 \right)dr \\
&=& -2\pi \int Qr^3 \chi_M dr + O\left(\frac{2\pi}{M} \int_M^{2M} Qr^4 dr\right) = -32\pi \log M + O(1).
\eee
Moreover
\bee
\int \pa_r {\Pi_{0,M}} \pa_r {\phi_{\Lambda Q}} = \int \frac{16r \log{(1+r^2)}}{(1+r^2)^2} dr < + \infty.
\eee
Finally, we obtain:
\be
\label{calcul32pi}
\left< \bf \Phi_{0,M}, {\bf \Lambda Q}\right> = \left< \bf \Phi_{M}, {\bf \Lambda Q}\right> = -32\pi \log M + O(1).
\ee

The compact support of $\Phi_{0,M}$ and $\pa_r \Pi_{0,M}$ together the decay of $T_1$ and $\pa_r S_1$ easily justify that:
\bee
\left< \mathcal L^*{\bf \Phi_{0,M}}, {\bf T_1}\right> = \left<{\bf \Phi_{0,M}},  \mathcal L {\bf T_1}\right> =\left<{\bf \Phi_{0,M}}, {\bf \Lambda Q}\right>
\eee
This yields \fref{orthophim}. Using $|T_1(r)| \lesssim \frac{1}{1+r^2}$ and $|\pa_r S_1(r) | \lesssim \frac{r}{1+r^2}$, we can compute
 \bee
|\left< {\bf \Phi_{0,M}}, {\bf T_1}\right>| \lesssim \int \frac{r^2}{1+r^2} \chi_M + \int \frac{\log(1+r^2)}{1+r^2} \chi_M \lesssim M^2.
\eee
With \fref{calcul32pi}, we have the upper bound:
\be
\label{jjyy}
|c_M| \lesssim \frac{M^2}{\log M}
\ee
which concludes the proof of (i).
\\ \\
{\bf Step 2 : } Proof of the estimate on the scalar product.
\\ \\
First, using the definition \fref{defphimzero} of the direction ${\bf \Phi_{0,M}}$ :
\be
\label{rrrr}
|\left< \boldsymbol \e , \bf \Phi_{0,M} \right> | \lesssim \int_{y \leq 2M} r^2|\e| + \int_{y \leq 2M} \frac{|\pa_r\eta|(\log(1+r^2))}{r}. 
\ee
Now, using \fref{lamande} and the cancellation \fref{nolwenn} :
\bea
\nonumber \left< \boldsymbol \e , \mathcal L^*\bf \Phi_{0,M} \right> &=& \int \e \left(  -4\chi_M + O({\bf 1}_{M\leq r \leq 2M})\right) + O \left(\int \nabla \eta \frac{\log M}{M^3}{\bf 1}_{M \leq y \leq 2M} \right)\\
\label{rrrrr}&=& -4 \int \e + O \left(\int_{y \geq M} |\e| + \log M \int_M^{2M}  \frac{|\nabla \eta|}{1+y^3} \right)
\eea
This result together \fref{rrrr}, the definition \fref{defphim} of the direction  ${\bf \Phi_{M}}$, and the bound \fref{jjyy} yields \fref{estimationorthobis} and \fref{bjebbeibei}.
To finish the proof of the Lemma \ref{direcitonroth}, we must estimate $(\mathcal L^*)^2 ({\bf \Phi_M})$. Let
\bee
(\mathcal L^*)^2 ({\bf \Phi_M}) = \mathcal M  \ \begin{array}{|c} r_3 \\ r_4 \end{array} = \mathcal M \  \begin{array}{|c} \nabla.(Q \nabla (\mathcal M^{(1)}(r_1,r_2))) \\ \Delta(\mathcal M^{(2)}(r_1,r_2))  \end{array}
\eee
With \fref{lamande} and \fref{nolwenn}, we obtain the following bound :
\bee
|r_3| &\lesssim&  \frac{1}{1+y^6}{\bf 1}_{M \leq y \leq 2M} \\
|r_4| &\lesssim& \frac{\log M}{M^4} {\bf 1}_{M \leq y \leq 2M}
\eee
and thus, using that $\pa_r \left(\phi_{\nabla.(Q \nabla (\mathcal M^{(1)}(r_1,r_2)))}\right) = Q \pa_r (\mathcal M^{(1)}(r_1,r_2))$
\bee
|\mathcal M^{(1)}(r_3,r_4)| &\lesssim& \frac{1}{1+y^2}{\bf 1}_{M \leq y \leq 2M} \\
\left|\pa_r \left(\mathcal M^{(2)}(r_3,r_4)\right)\right| &\lesssim& \frac{\log M}{M^5} {\bf 1}_{M \leq y \leq 2M}
\eee
Hence,
\bee
\left|\left< \boldsymbol \e, (\mathcal L^*)^2 ({\bf \Phi_M})\right>\right| \lesssim \int_{M \leq y \leq 2M} \frac{|\e|}{1+r^2} + \frac{\log M}{M^2}\int_{M \leq y \leq 2M} \frac{|\pa_r \eta|}{1+y^3}
\eee
With this bound, we obtain any difficulty \fref{newestimate} and \fref{estfonamentalebus}. \fref{firstroughbound} and \fref{secondroughbound} respectively come from \fref{estimationorthobis} and \fref{newestimate} using Cauchy-Schwarz. This concludes the proof of the Lemma \ref{direcitonroth}
\end{proof}
We are now in position to derive the fundamental coercivity property described in the following Proposition at the heart of our analysis. Moreover, we track the M dependence of constants which is crucial for the derivation of the blow-up speed. 
\begin{proposition}[Coercivity of $\mathcal L $]
\label{interpolationhtwo}
There exist universal constants $\delta_0,M_0>0$ such that $\forall M\geq M_0$, there exists $\delta(M)>0$ such that the following holds. Let $\boldsymbol \e\in \mathcal E$ with 
satisfying the following orthogonality conditions:
\be
\label{orthowappendix}
\left<\boldsymbol \e,\boldsymbol \Phi_M \right> = \left<\boldsymbol \e, \mathcal L^*\boldsymbol\Phi_M \right> =0.
\ee
Then there hold the bounds:\\
(i) Control of $\mathcal L(\e, \eta)$: 
\be
\label{coerclwotht}
\left<\mathcal M\mathcal L (\e,\eta),\mathcal L(\e,\eta) \right> \geq \frac{\delta_0(\log M)^2}{M^2} \| \mathcal L(\e,\eta)\|_{X_Q}^2,
\ee
(ii) Coercivity of $\mathcal L$: 
\bea
\label{contorlcoerc}
\nonumber \frac{1}{\delta(M)} \|\mathcal L(\e,\eta)\|_{X_Q}^2& \geq & \int (1+r^4)|\Delta \e|^2+\int(1+r^2)|\nabla \e|^2+\int \e^2\\
 \nonumber& + & \int\frac{|\nabla \phi_\e|^2}{r^2(1+|\log r|)^2} + \int |\nabla(\Delta \eta)|^2 \\
&+& \int\frac{|\Delta \eta|^2}{r^2(1+|\log r|)^2}  +\int\frac{|\nabla \eta|^2}{r^2(1+r^2)(1+|\log r|)^2} 
\eea
\end{proposition}
\begin{proof}
{\bf Step 1 : } Control of $\mathcal L\boldsymbol \e$.
\\ \\
Let $\begin{array} {|c} \e \\ \eta \end{array} \in \mathcal E$. Let
$$\begin{array} {|c} \e_2 \\ \eta_2 \end{array} = \mathcal L \ Ê\begin{array} {|c} \e \\ \eta \end{array} = \begin{array} {|c} \nabla.(\nabla \e + \nabla \phi_Q \e + \nabla \eta) \\ \Delta \eta - \e \end{array}. $$
The definition of the energy space $\mathcal E$, and the vanishing average \fref{estvooee} of the Lemma \ref{lemmacontiniuty} assure that $\boldsymbol \e_2 \in X_Q$. Moreover, from the choice of orthogonality conditions \fref{coerclwotht}, and the adjunction \fref{biebibvei} :
\bee
\left<\boldsymbol \e_2,\Phi_M \right> = \left<\mathcal L\boldsymbol \e,\Phi_M \right> = \left<\boldsymbol \e, \mathcal L^*\Phi_M \right> =0.
\eee
In order to know precisely the M-dependence of the constants, we use the following suitable function:
\be
\begin{array} {|c} \et_2 \\ \tilde \eta_2 \end{array} = \begin{array} {|c} \e_2 \\ \eta_2 \end{array} - a_1 {\bf \Lambda Q} 
\ee
with
\be
 a_1=\frac{\left< \boldsymbol \e_2, \bf \Lambda Q \right>}{\left< \bf \Lambda Q, \bf \Lambda Q \right> }
\ee
which yields :
\be
\left< \boldsymbol \et_2, {\bf \Lambda Q}\right> = 0.
\ee
Furthermore, the cancellation
\bee
\int \Lambda Q = 0
\eee
together the definition of $ \boldsymbol \et_2$ yield that $ \boldsymbol \et_2 \in X_Q$. Thus, we can apply the Lemma \ref{propenergy}, and obtain the bound:
\be
\label{ppppaaaa}
\left< \mathcal M \boldsymbol \et_2, \boldsymbol \et_2 \right> \geq \delta_0  \|\boldsymbol \et_2\|_{X_Q}^2
\ee
Next 
\be
\label{gtgtgtgtgt}
\left< \mathcal M \boldsymbol \e_2, {\bf \Lambda Q} \right> = \left<\boldsymbol \e_2, \mathcal M {\bf \Lambda Q} \right> = 0
\ee
and
\be
\label{gtgtgtgtgtgt}
\left< \mathcal M {\bf \Lambda Q}, {\bf \Lambda Q} \right> = 0.
\ee
Thus
\bea
\nonumber\left< \mathcal M \boldsymbol \et_2, \boldsymbol \et_2 \right> &=& \left< \mathcal M \boldsymbol \e_2 -a_1 \mathcal M {\bf \Lambda Q}, \boldsymbol \e_2 -a_1{\bf \Lambda Q} \right> \\
\nonumber &=& \left< \mathcal M \boldsymbol \e_2, \boldsymbol \e_2 \right> - 2a_1 \left< \mathcal M \boldsymbol \e_2, {\bf \Lambda Q} \right> + a_1^2 \left< \mathcal M {\bf \Lambda Q},{\bf \Lambda Q} \right>\\
\label{pppppaaaaa}  &=& \left< \mathcal M \boldsymbol \e_2, \boldsymbol \e_2 \right>
\eea
Finally, we use the orthogonality condition on $\boldsymbol \e_2$, and the bound \fref{firstroughbound} to estimate:
\bee
|a_1|&=& \left|\frac{\left<\boldsymbol \et_2,\Phi_M \right> }{\left<{\bf \Lambda Q},\Phi_M \right> }\right| \lesssim \frac{1}{|\log M|}\left[M \| \boldsymbol \et_2\|_{X_Q}+\frac{M^2}{\log M}\left|\int \et_2\right| \right] \\
& \lesssim & \frac{M}{\log M}\| \boldsymbol \et_2\|_{X_Q},
\eee
and thus 
\be
\label{cenneoneoen}
\| \boldsymbol \e_2\|_{X_Q}\lesssim \frac{M}{\log M}\| \boldsymbol \et_2\|_{X_Q}
\ee
which together \fref{ppppaaaa} and \fref{pppppaaaaa} concludes the proof of \fref{coerclwotht}.
\\ \\
{\bf Step 2 :} Subcoercivity of $\mathcal L$
\\ \\
In order to prove \fref{contorlcoerc}, we begin to prove the subcoercivity estimate, which is describe in the following inequality, based on two dimensional Hardy inequalities and a very good knowledge of the structure of the operator. We will be then in position to prove the coercivity of $\mathcal L$ under orthogonality conditions \fref{orthowappendix} from a compactness argument. Let $\boldsymbol \e \in X_Q$. Then:
\bea
\label{subcerovitybis}
\nonumber&&  \int\frac{\left(\mathcal L^{(1)}(\e,\eta)\right)^2}{Q} + \int \left| \nabla \mathcal L^{(2)}(\e,\eta)\right|^2  \\
 \nonumber & \gtrsim & \int (1+r^4)(\Delta \e)^2+\int\left[\frac{1}{r^2(1+|\log r|)^2}+r^2\right]|\nabla \e|^2 + \int \e^2 + \int |\nabla (\Delta \eta)|^2 \\
\nonumber &+& \int\frac{|\Delta \eta|^2}{r^2(1+|\log r|)^2} 
   +\int\frac{|\nabla \eta|^2}{r^2(1+r^2)(1+|\log r|)^2} + \int\frac{ \eta^2}{(1+r^6)(1+|\log r|)^2} \\
&  - &\int\frac{\e^2}{1+r^2}-\int |\nabla\e|^2 -  \int\frac{|\Delta \eta|^2}{1+r^4} - \int\frac{|\nabla \eta|^2}{1+r^6}.
\eea
{\it Proof of \fref{subcerovitybis} :} \\
Using the explicit expression of $\mathcal L^{(1)}$ :
\bea
\label{cnoneoneoc}
\nonumber \int \frac{\left|\mathcal L^{(1)}(\e,\eta)\right|^2}{Q}&\gtrsim& \int\frac{(\Delta \e+\nabla \phi_Q\cdot\nabla \e)^2 }{Q}-\int\frac{|\nabla Q|^2}{Q}|\nabla \eta|^2-\int Q(\e^2+|\Delta \eta|^2) \\
& \gtrsim & \int\frac{(\Delta \e+\nabla \phi_Q\cdot\nabla \e)^2 }{Q}-\int \frac{\e^2 + |\Delta\eta|^2}{1+r^4} - \int\frac{|\nabla \eta|^2}{1+r^6}.
\eea
The first term of RHS requires a carefully compute. Let's develop it:
\bee
\int\frac{(\Delta \e+\nabla \phi_Q\cdot\nabla \e)^2 }{Q}=\int\frac{(\Delta \e)^2}{Q}+\frac{(\pa_r\phi_Q\pa_r\e)^2 }{Q}+\int\frac{2}{Q}\Delta \e\nabla\phi_Q\cdot \nabla \e.
\eee
Now we observe that : $$\frac{2\nabla\phi_Q}{Q}=-2\frac{\nabla Q}{Q^2}=2\nabla \psi, \ \ \psi=\frac{1}{Q}.$$ With the classical Pohozaev integration by parts formula:
\bee
\nonumber &&2\int \Delta \e\pa_r\psi\pa_r\e=    2 \int \pa_r(r\pa_r\e) \pa_r\psi\pa_r\e dr\\
\nonumber & = & -\int (r\pa_r\e)^2\pa_r\left(\frac{\pa_r\psi}{r}\right)dr =  -\int(\pa_r\e)^2 \left[\pa^2_r\psi-\frac{\pa_r\psi}{r}\right].
\eee
Moreover, we have the following Taylor series for $r\gg 1$: $$\psi(r)=\frac{1}{Q}=\frac{r^4}{8}+O(r^2), \ \ \pa^2_r\psi-\frac{\pa_r\psi}{r}=r^2+O(1),$$ $$\phi_Q'(r)=\frac{1}{r}\int_0^rQ(\tau)\tau d\tau=\frac{4}{r}+O\left(\frac{1}{r^3}\right),\ \ \frac{(\pa_r\phi_Q)^2}{Q}=2r^2+O(1).$$ Thus:
\bee
&&\int\frac{(\Delta \e+\nabla \phi_Q\cdot\nabla \e)^2 }{Q}  \gtrsim \int(1+r^4)(\Delta \e)^2+\int (2r^2-r^2)(\pa_r\e)^2-\int|\nabla\e|^2.
\eee
The above inequality together \fref{cnoneoneoc} yields the lower bound:
\be
\label{HRP}
\int \frac{\left|\mathcal L^{(1)}(\e,\eta)\right|^2}{Q} \gtrsim  \int(1+r^4)(\Delta \e)^2+\int r^2|\nabla \e|^2-\int \frac{\e^2+ |\Delta\eta|^2}{1+r^4}-\int|\nabla\e|^2- \int\frac{|\nabla \eta|^2}{1+r^6}
\ee
Now,
\be
\label{HRP2}
\int \left| \nabla \mathcal L^{(2)}(\e,\eta)\right|^2 \gtrsim  \int |\nabla (\Delta \eta)|^2 -\int|\nabla\e|^2.
\ee
Injecting the Hardy bounds of the Lemma \fref{weightedhardy} in \fref{HRP} and \fref{HRP2} conclude the proof of \fref{subcerovitybis}.
\\ \\
{\bf Step 3 :} Coercivity of $\mathcal L$:
\\ \\
Let's prove now \fref{contorlcoerc}. In fact, we prove it for $(\e,\eta) \in \mathcal C^{\infty}_c(\mathbb R)^2$, such that $\e(0) = \Delta \eta(0)$. Using the continuity of $\mathcal L$  and a standard argument of density, it is enough to obtain \fref{contorlcoerc}. \par
In the same way as the proof of the coercivity of $\mathcal M$, we argue by contradiction. Let a sequence  $\boldsymbol \e_p = (\e_p,\eta_p) \in \mathcal C^{\infty}_c(\mathbb R)^2$ such that :
\begin{itemize}
\item Condition at the origin : $\e_p(0) = \Delta \eta_p(0).$
\item Orthogonality condition : $\left<\boldsymbol \e_p,\Phi_M \right> = \left<\boldsymbol \e_p,L^*\Phi_M \right> =0.$
\item Normalization :
\bee
&&\int (1+r^4)(\Delta \e_p)^2+\int\left[\frac{1}{r^2(1+|\log r|)^2}+r^2\right]|\nabla \e_p|^2+\int \e_p^2 + \int |\nabla (\Delta \eta_p)|^2 \\
&+& \int\frac{|\Delta \eta_p|^2}{r^2(1+|\log r|)^2} + \int\frac{|\nabla \eta_p|^2}{r^2(1+r^2)(1+|\log r|)^2} + \int\frac{ \eta_p^2}{(1+r^6)(1+|\log r|)^2}  =1
\eee
\item Control of the $X_Q$-norm : $\|\boldsymbol \e_p \|^2 \leq \frac 1p.$
\end{itemize}
The normalization condition implies that the sequence $\boldsymbol \e_p$ is uniformly bound in $H^2_{loc} \times H^3_{loc}$. Hence, we may extract up to a subsequence such that it weakly converges in $H^2_{loc} \times H^3_{loc}$ to $(\e_{\infty},\eta_{\infty})$ with the orthogonality conditions :
\be
\label{ttrreezzaaa}
\left<\boldsymbol \e_{\infty},\Phi_M \right> = \left<\boldsymbol \e_{\infty},L^*\Phi_M \right> =0.
\ee
Moreover, the control of the $X_Q$-norm implies that:
\be
\label{hrphrp}
\mathcal L (\e_{\infty},\eta_{\infty}) = \begin{array} {|c} \frac{1}{r} \pa_r \left( r Q \pa_r \left( \frac{\e_{\infty}}{Q} + \eta_{\infty} \right) \right) \\ \Delta \eta_{\infty} - \e_{\infty} \end{array} = \begin{array} {|c} 0 \\ 0 \end{array}.
\ee
As we have for all $ p$, $\phi_{\Delta \eta_p} = \eta_p$, it's clear that  $\phi_{\Delta \eta_{\infty}} = \eta_{\infty}$. Now, \fref{hrphrp} implies that
\be
\label{goretex}
r Q \pa_r \left( \frac{\e_{\infty}}{Q} + \eta_{\infty} \right) = c, \ \ c \in \mathbb R.
\ee
But
\bee
\left| r Q \pa_r \left( \frac{\e_{\infty}}{Q} + \eta_{\infty} \right)\right| \lesssim r |\pa_r \e| + |\e| + \frac{|\pa_r \eta|}{1+r^3}.
\eee
With the normalization condition, and the weakly convergence, we obtain the upper bound :
\bee
&&\int (1+r^4)(\Delta \e_{\infty})^2+\int\left[\frac{1}{r^2(1+|\log r|)^2}+r^2\right]|\nabla \e_{\infty}|^2+\int \e_{\infty}^2 + \int |\nabla (\Delta \eta_{\infty})|^2 \\
&+& \int\frac{|\Delta \eta_{\infty}|^2}{r^2(1+|\log r|)^2} + \int\frac{|\nabla \eta_{\infty}|^2}{r^2(1+r^2)(1+|\log r|)^2} + \int\frac{ \eta_{\infty}^2}{(1+r^6)(1+|\log r|)^2}  \leq 1
\eee
Hence, the constant $c$ of \fref{goretex} equals 0. The generalized kernel \fref{identificationkernel} of the operator $\mathcal M$ of the Lemma \ref{lemmakernel} ensures that $\boldsymbol \e_{\infty} \in Span ({\bf \Lambda Q})$, and the orthogonality conditions \fref{ttrreezzaaa} implies that $\boldsymbol \e_{\infty} =0$.
\par
Now, the control of the $X_Q$-norm together the subcoercivity lower bound \fref{subcerovitybis} yield the non degeneracy :
\bee
\int\frac{\e_{\infty}^2}{1+r^2}+\int |\nabla\e_{\infty}|^2 +  \int\frac{|\Delta \eta_{\infty}|^2}{1+r^4} + \int\frac{|\nabla \eta_{\infty}|^2}{1+r^6} >0.
\eee
which contradicts $\boldsymbol \e_{\infty} =0$. This concludes the proof of the Proposition \ref{interpolationhtwo}.
\end{proof}
\section{Construction of approximate solution}
\label{consconscons}
The purpose of this section is to obtain an approximate blow-up solution of \fref{kps}. More precisely, the constructed solution will contain the main qualitative informations on the dynamics of the singularity formation. We will measure the made error by suitable quantities which will allow us to prove thereafter the smallness of the gap between this approximate solution and the exact solution, in the sense that this gap doesn't perturb the blow-up dynamic found in this section. 
\subsection{On the rescaled variables}
First, we can remark that if $(u,v)$ is solution of \fref{kps}, then $(u,v+c)$ with $c \in \mathbb R$ too. So, before introducing the rescaled variables, the second equation of the system \fref{kps} are being gone through the operator gradient. Hence, \fref{kps} becomes:
\be
\label{PKSP}
(PKS') \left\{ \begin{array} {c}
\pa_t u = \nabla.(\nabla u + u\nabla v), \\
\pa_t \nabla v = \nabla \Delta v - \nabla u, \\
u_{|t=0} = u_0 >0, \\
\nabla v_{|t=0} = \nabla v_0
\end{array} \right. \ \ (t,x )\in (\Bbb R \times \Bbb R^2) 
\ee
Moreover, we have the following scaling invariance: if $(u(t,r), \nabla v (t,r))$ solves \fref{PKSP}, then so does $(u_\l,\nabla v_\l)=\left(\frac 1 {\l^2}u(\frac t {\l^2},\frac r {\l}),\frac 1 {\l} \nabla v (\frac t {\l^2},\frac r {\l})\right)$ for $\l >0$. \par
Now, let $\l(t)$ a regular non negative function. Let the rescaled variables:
\be
s = \int_0^t \frac{d \tau}{\l^2(\tau)}, \ \ y = \frac{r}{\l(t)}
\ee
As we look for a slower blow up than the self similar regime, ie $\l(t) < \sqrt{T-t}$, noticing $T$ the blow-up time. We can remark that $s(t)$ is a bijection between $[0,T[$ and $\mathbb R^+$. With this new variables \fref{PKSP} becomes
 \be \pa_s \ \ \begin{array}{|c}
u\\
\nabla v
\end{array} - \frac{\l_s}{\l} \Lambda \ \ \begin{array}{|c}
u\\
\nabla v
\end{array}
 = \begin{array}{|c}
\nabla.(\nabla u + u\nabla v) \\
\nabla (\Delta v - u). \\
\end{array}
\ee
where we recall that $\Lambda$ is the scaling operator define by:
\be
\label{defLambda}
\Lambda \ \ \begin{array}{|c}
f\\ \nabla g
\end{array} =  \begin{array}{|c}
\Lambda ^{(1)}f\\ \Lambda ^{(2)} \nabla g
\end{array} = \begin{array}{|c}
2f + y.\nabla f\\ \nabla g + y. \nabla^2 g
\end{array}= \begin{array}{|c}
\nabla.(y f)\\ y.\Delta g
\end{array}
\ee
\subsection{Introduction of the main tools}
First, we let 
\be
\label{introb}
b(s) = - \frac{\l_s}{\l}.
\ee
In the blow-up regime described in the Theorem \ref{thmmain}, we can see that $b$ is a very small non negative function. In fact, in the next sections, we relax this constraint to obtain two independent parameters of modulation $b$ and $\l$, which we allow us to fix two orthogonality conditions for the gap between the constructed approximate solution and the exact solution. We saw in the last section the importance of this to have coercivity properties for the operators $\mathcal L$ and $\mathcal M$.\par
Now, we look for an approximate solution ${\bf Q}_b$ of \fref{PKSP} close to $\bf Q$ in the form :
\be
\label{formeQb}
{\bf Q}_b =
  \begin{array}{| c}
     Q_b(r) \\
     P_b(r)
  \end{array} =   {\bf Q} + b \ \ 
    \begin{array}{| c}
     T_1(r) \\
     S_1(r)
  \end{array} + b^2 \ \  
   \begin{array}{ |c}
     T_2(b,r) \\
     S_2(b,r)
  \end{array}
  ={ \bf Q} + {\bf \Upsilon}_b = {\bf Q} +  
    \begin{array}{| c}
     \alpha_b(r) \\
     \gamma_b(r)
  \end{array}
\ee
where ${\bf T}_1$ and ${\bf T}_2$ are profiles independent on rescaled time $s$, which we will determine. The error ${\bf \Psi}_b$ associated to ${\bf Q}_b$ is defined according to the
formula :
 \be
 \label{eqerreurpsib}
 \begin{array}{| c}
    \Psi_b^{(1)} \\
     \nabla  \Psi_b^{(2)}
  \end{array}  =  \begin{array}{ |c}
   \nabla . \left( \nabla Q_b + Q_b\nabla P_b\right) \\
   \nabla (\Delta P_b - Q_b) 
  \end{array} 
- b \Lambda 
  \ \ \begin{array}{ |c}
     Q_b(r) \\
      \nabla P_b(r)
  \end{array} 
    +c_bb^2
\chi_{\frac{B_0}{4}}\begin{array}{|c}T_1 \\  \nabla S_1 \end{array}   \ee
  with $c_b$ given by \fref{defcb}. Remark that $\Psi_b^{(2)}$ is defined to within a function of s, which we don't have to determine. The last term of the RHS comes from the radiation which we will use in the construction to improve the size of ${\bf \Psi}_b$. This term has to be extracted from the error. Indeed, whereas we will can use the inequality of Cauchy-Schwartz for the terms depending on ${\bf \Psi}_b$ in the control of the gap between ${\bf Q}_b$ and the exact solution, we will have to be careful with this last term, and use its particular structure. \par
  Moreover, indicate that the constructed profiles will have pathological growth outside the parabolic zone. Hence, after determining this profiles, we must localize them. We will see in particular that the choice $B_1 = \frac{|\log b|}{\sqrt b}$ comes from terms of error due to the localization, which mustn't be bigger than terms of error due to the construction.
\par
  Let's introduce the partial mass associated to $Q_b$
  \be
  m_b(r) = \int_0^r Q_b(\tau) \tau d\tau.
  \ee
  and the mass partial associated to $\Delta P_b$
  \be
  n_b(r) = \int_0^r \Delta P_b(\tau) \tau d\tau.
  \ee 
  Hence, remarking that
  \be
  \nabla \phi_{Q_b}(r) = \frac{m_b}{r}, \ \ Q_b(r) = \frac{m_b'}{r}, \ \  \nabla P_b(r) = \frac{n_b}{r}, \  \mbox{and} \ \ \Delta P_b(r) = \frac{n_b'}{r},
  \ee
 we can rewrite \fref{eqerreurpsib} by
   \be
 \label{eqerreurpsib}
 \begin{array}{| c}
    \Psi_b^{(1)} \\
     \nabla  \Psi_b^{(2)}
  \end{array}  =  
 \frac 1r\begin{array}{| c}
    \Phi_b' \\
  \Omega_b
  \end{array}  
      +c_bb^2
\chi_{\frac{B_0}{4}}\begin{array}{|c}T_1 \\  \nabla S_1 \end{array}   \ee
where
\bea
\label{gepx}
 \begin{array}{| c}
  \Phi_b \\
   \Omega_b
  \end{array} &=& 
  \begin{array}{ |c}
   m''_b - \frac{m'_b}r + \frac{m'_bn_b}{r^2} - brm'_b \\
   (n_b-m_b)'' - \frac{(n_b - m_b)'}{r} - brn'_b
  \end{array} \\
  \nonumber
   &=&  \begin{array}{ |c}
   m''_b - \frac{m'_b}r + \frac{m'_bm_b}{r} + \frac{m'_bd_b}{r} - brm'_b \\
d''_b - \frac{d'_b}r - brn'_b
  \end{array} 
\eea
with
\be
d_b = n_b - m_b.
\ee
We proceed to an expansion on the form :
\be
\label{expansion}
m_b = m_0 + bm_1 + b^2m_2, \ \ n_b = n_0 + bn_1 + b^2n_2, \ \  d_b = d_0 + bd_1 + b^2d_2
\ee
with
\be
 m_0(r) = n_0(r) =  \int_0^r Q(y) y dy = \frac{4r^2}{1+r^2}, \ \ d_0(r) = 0,
\ee
$m_1$ and $m_2$ the partial mass of respectively $T_1$ and $T_2$, $n_1$ and $n_2$ the partial mass of respectively $\Delta S_1$ and $\Delta S_2$. The subject of this section is to determine $m_1$, $m_2$, $n_1$ and $n_2$, in order to minimize in a suitable sense the size of ${\bf \Psi}_b$. So, it will need to find a function $u$ solution of 
\be
L_0 u = -f,
\ee
where $f$ is a given function, and $L_0$ the linearized operator close to $m_0$ of  \fref{gepx} given by:
\bee
L_0 m & = &  -m''+\left(\frac{1}{r}+\frac{Q'}{Q}\right)m'-Qm.\\
\eee
We recall that the basis of solutions to this homogeneous equation is explicit and given by:
 \bee
 \psi_0(r)=\frac{r^2}{(1+r^2)^2}, \ \ \psi_1(r)=\frac{1}{(1+r^2)^2}\left[r^4+4r^2\log r-1\right].
 \eee
 with the Wronskian :
 \be
 \label{defW}
 W=\psi_1'\psi_0-\psi_1\psi_0'=\frac{rQ}{4}=\frac{2r}{(1+r^2)^2}.
 \ee
Hence a solution to $$L_0m=-f$$ can be found by the method of variation of constants:
 \be
 \label{fromulau}
 m=A\psi_0+B\psi_1\ \ \mbox{with}\ \ \left\{\begin{array}{ll}A'\psi_0+B'\psi_1=0,\\ A'\psi'_0+B'\psi_1'=f, \end{array}\right ..
 \ee 
Hence, we obtain
 $$B'=\frac{f\psi_0}{W}=\frac{r}{2}f, \ \ A'=-\frac{f\psi_1}{W}=-\frac{r^4+4r^2\log r-1}{2r}f$$ and a solution is given by:
  \be
 \label{inversionknot}
 m(r)=-\frac12\psi_0(r)\int_0^r\frac{\tau^4+4\tau^2\log \tau-1}{\tau}f(\tau)d\tau+\frac12\psi_1(r)\int_0^r\tau f(\tau)d\tau.
 \ee
 Moreover, we compute
 \be
 \label{estderivatibves}
 \frac{\psi_0'}{r}=\frac{2(1-r^2)}{(1+r^2)^3} = \frac{\Lambda Q}{8}, \ \  \frac{\psi_1'}{r}=\frac{8(1+r^2-(r^2-1)\log r)}{(1+r^2)^3}.
 \ee
and then \fref{fromulau} yields 
\bea
\label{formulamprime}
\frac{m'}{r} & = & A\frac{\psi'_0}{r}+B\frac{\psi'_1}{r}\\
\nonumber & = & -\frac{1-r^2}{(1+r^2)^3}\int_0^r\frac{\tau^4+4\tau^2\log \tau-1}{\tau}f(\tau)d\tau+\frac{4(1+r^2-(r^2-1)\log r)}{(1+r^2)^3}\int_0^r\tau f(\tau)d\tau.
\eea
To conclude this part, we will need inverse an other operator, coming from the second equation :
\be
\label{defL1}
L_1 d = d'' - \frac{d'}{r} = f
\ee
where $f$ is given. $d_0(r) = r^2$  and $d_1(r) = 1$ are a basis of the homogeneous problem. Hence, a solution of \fref{defL1} is given by
\be
\label{solutiond}
d = \frac{1}{2} \left[  - \int_0^r f(\tau) \tau d\tau + r^2 \int_0^r \frac{f(\tau)}{\tau} d\tau \right] + cr^2, \ \ c \in \mathbb R.
\ee
According to the definition of d, we have the constraint $d(0)=0$. Hence, we consider only the solutions of the above form. In fact, $d_1(r)$ corresponds to the singular solution at the origin of $\Delta u=0$. \par
We are now in position to determine the approximate solution :
\begin{proposition}[Construction of the approximate profile]
\label{construction}
Let $M >0$ enough large. Then, there exists a small enough universal constant $b^*(M)$ such that the following holds. Let $b \in ]0,b^*(M)[$, and $B_0$ et $B_1$ given by  \fref{defB0B1}. There exists radial profiles $T_1$, $T_2$, $S_1$ et $S_2$, such that
\be
{\bf Q}_b =
  \begin{array}{| c}
     Q_b(r) \\
     P_b(r)
  \end{array} =   {\bf Q} + b \ \ 
    \begin{array}{| c}
     T_1(r) \\
     S_1(r)
  \end{array} + b^2 \ \  
   \begin{array}{ |c}
     T_2(b,r) \\
     S_2(b,r)
  \end{array}
  ={ \bf Q} + {\bf \Upsilon}_b = {\bf Q} +  
    \begin{array}{| c}
     \alpha_b(r) \\
     \gamma_b(r)
  \end{array}
\ee
is an approximate solution of \fref{kps} in the following sense. Let the error:
 \be
 \label{eqerreurpsib}
{\bf \Psi}_b =  \begin{array}{| c}
    \Psi_b^{(1)} \\
      \Psi_b^{(2)}
  \end{array}  =  \begin{array}{ |c}
   \nabla . \left( \nabla Q_b + Q_b\nabla P_b\right) \\
  \Delta P_b - Q_b 
  \end{array} 
- b \Lambda 
  \ \ \begin{array}{ |c}
     Q_b(r) \\
     P_b(r)
  \end{array} 
    +c_bb^2 \breve {\bf T}
 \ee
  with $c_b$ given by \fref{defcb}, and $\breve {\bf T}$ such that
 \be
 \label{breveT1}
\begin{array}{|c}\breve T^{(1)} \\ \nabla \breve T^{(2)}  \end{array} = \chi_{\frac{B_0}{4}}\begin{array}{|c}T_1 \\ \nabla S_1 \end{array}
 \ee
  Then there holds:\\
{\em (i) Control of the tails}: $ \forall r\geq 0$, $ \forall i\geq 0$:
\bea
\label{esttoneprop}
|r^i\pa^i_rT_1|&\lesssim& \frac{r^2}{1+r^4}, \\
\label{nablaS1taille}
|r^i\pa^i_r\nabla S_1|&\lesssim& \frac{r}{1+r^2},
\eea
and $ \forall r\leq 2B_1$, $ \forall i\geq 0$:
\bea
 \label{esttwo}
 |r^i\pa_r^iT_2|&\lesssim& r^2{\bf 1}_{r\le1} +\frac{1 + |\log (r\sqrt b)|}{|\log b|}{\bf 1}_{1\leq r\leq 6B_0}+ \frac{1}{b^2r^4|\log b|}{\bf 1}_{r\geq 6B_0},\\
\label{estttwodtdb}
|b\pa_br^i\pa_r^iT_2|&\lesssim&\frac{1}{|\log b|}\left[ r^2{\bf 1}_{r\le1} +\frac{ 1+|\log r|}{|\log b|}{\bf 1}_{1\leq r\leq 6B_0}+ \frac{1}{b^2r^4}{\bf 1}_{r\geq 6B_0}\right],\\
 \label{nablaS2par}
 |r^i\pa_r^i\nabla S_2|&\lesssim&  r{\bf 1}_{r\leq1} + r (1+ |\log r|) {\bf 1}_{r\geq1},\\
 \label{pabnablaS2bis}
 |b\pa_b \nabla S_2|&\lesssim& \frac {1}{|\log b|^2} \left\{  r{\bf 1}_{r\leq1} + r (1+ |\log r|^2) {\bf 1}_{r\geq1} \right\}.
 \eea
{\em (ii) Control of the error in weighted norms: for} $i\geq 0$,
\be
\label{roughboundltaow}
\int_{r\leq 2B_1} |r^i\pa_r^i\Psi^{(1)}_b|^2+ \int_{r\leq 2B_1}\frac{|\mathcal L^{(1)}(\Psi^{(1)}_b,\Psi^{(2)}_b)|^2}{Q}\lesssim \frac{b^5}{|\log b|^2},
\ee
\be
\label{roughboundltaow2}
\nonumber  \int_{r\leq 2B_1} \frac{|\nabla {\Psi^{(2)}_b}|^2}{1+\tau^2}+ \int_{r\leq 2B_1} |\mathcal L^{(2)}(\Psi^{(1)}_b,\Psi^{(2)}_b)|^2\lesssim b^5|\log b|^2,
\ee
\be
\label{cneneoneonoenoe}
\int_{r\leq 2B_1} Q|\nabla \mathcal M^{(1)}(\Psi^{(1)}_b,\Psi^{(2)}_b)|^2\lesssim \frac{b^4}{|\log b|^2}.
\ee
\be
\label{roughboundltaow3}
\int_{r\leq 2B_1} |\nabla \Psi_b^{(2)}|^2 \lesssim b^4 |\log b|^6.
\ee
\end{proposition}
\begin{proof}
{\bf Step 1 : } Computation of $(\Phi_b,\Omega_b)$
\\
We recall that :
\bee
 \begin{array}{| c}
  \Phi_b \\
   \Omega_b
  \end{array}
   &=&  \begin{array}{ |c}
   m''_b - \frac{m'_b}r + \frac{m'_bm_b}{r} + \frac{m'_bd_b}{r} - brm'_b \\
d''_b - \frac{d'_b}r - brn'_b
  \end{array} 
\eee
Injecting the expansions of $m_b$ and $d_b$ \fref{expansion}, and the definition of the operators $L_0$ and $L_1$, we obtain
\bea
\nonumber \begin{array}{| c}
  \Phi_b \\
   \Omega_b
  \end{array} &=& b \left\{{\bf L} \ \begin{array} {|c} m_1 \\ d_1 \end{array}- \begin{array}{| c}
   rm'_0 \\
  rn'_0
  \end{array} \right\}+ b^2 \left\{ {\bf L} \ \begin{array} {|c}m_2 \\d_2 \end{array}-\begin{array}{| c}
rm'_1- \frac{m'_1m_1}{r} - \frac{m'_1d_1}{r}\\
rn'_1
  \end{array} \right\} \\
\label{defphib}   &+& b^3 \ \  \begin{array}{| c}
-rm'_2+ \frac{(m_1m_2)'}{r} + \frac{m'_1d_2 + m'_2d_1}{r}\\
-rn'_2
  \end{array} 
+
b^4 \ \  \begin{array}{| c}
\frac{m_2m_2'}{r} + \frac{m'_2d_2}{r}\\
0
  \end{array} 
\eea
where
\be
\label{L}
 {\bf L} \  \begin{array} {|c} m \\ d \end{array} = \begin{array} {|c}- L_0 m + \frac{m'_0}{r} d \\ L_1 d \end{array}
\ee
{\bf Step 1 : } Level $b$
\\ \\
We look for $m_1$ and $d_1$ such that
\bee
{\bf L} \ \begin{array} {|c} m_1 \\ d_1 \end{array} = \begin{array}{| c}
   rm'_0 \\
  rn'_0
  \end{array}
\eee
First with \fref{defL1}, a solution of $L_1d_1 = rn'_0$ is given for $ c \in \mathbb R$ by:
\bee
d_{1,c}(r) &=& \frac 12 \left[ - \int_0^r \frac{8\tau^3d\tau}{(1+\tau^2)^2} + r^2 \int_0^r \frac{8\tau d\tau}{(1+\tau^2)^2}\right] + cr^2 \\
&=& \frac 12 \left[ - \int_0^r \frac{8(\tau^3+ \tau) d\tau}{(1+\tau^2)^2} + (r^2+1) \int_0^r \frac{8\tau d\tau}{(1+\tau^2)^2}\right] + cr^2 \\
&=& - 2 \log (1+r^2) + (2+c)r^2.
\eee
Hence, we select as solution :
\be
\label{defd1}
d_1(r) = d_{1,-2} (r) =  - 2 \log (1+r^2).
\ee
We are in position to determine $m_1$ be the solution to :
\be
L_0m_1 =  - rm'_0 \left(1 - \frac{d_1}{r^2}\right)
\ee
given by
\bee
m_1(r)&=&-4\psi_0(r)\int_0^r\frac{\tau(\tau^4+4\tau^2\log \tau-1)}{(1+\tau^2)^2} \left(1 - \frac{d_1}{\tau^2}\right)d\tau\\
&+&4\psi_1(r)\int_0^r\frac{\tau^3}{(1+\tau^2)^2} \left(1 - \frac{d_1}{\tau^2}\right)d\tau.
\eee
Using the explicit formula :
\be
\int_0^r \frac{\tau^3 d\tau}{(1+\tau^2)^2} = \frac{\log (1+r^2)}{2} + \frac 12 \left( \frac{1}{1+r^2} - 1\right),
\ee
we obtain :
\be
\label{comportementm1}
m_1 = \left \{ \begin{array} {c}
O(r^4) \ \ \mbox{at the origin} \\
2 \log (1+r^2) - 4 + O \left( \frac{|\log r|^2}{r^2}\right) = 4(\log r-1) + O \left( \frac{|\log r|^2}{r^2}\right) \ \ \mbox{as} \ \ r \rightarrow + \infty
\end{array} \right.
\ee
 Hence, using that  $T_1 = \frac{m_1'}{r}$, there holds the behavior at the origin
 \be
 T_1 = O(r^2)
 \ee
and, for r large :
\be
T_1(r) = \frac{4}{r^2} + O \left( \frac{|\log r|^2}{r^4} \right)
\ee
In particular, this yields the bound for $i \geq 0$:
\be
\label{boundT_1}
|r^i\pa_r^i T_1| \lesssim \frac{r^2}{1+r^4}.
\ee
Now, by definition
\be
n_1 = d_1 + m_1
\ee
Hence, with \fref{defd1} and \fref{comportementm1}:
\be
\label{comportementn1}
n_1 = \left \{ \begin{array} {c}
O(r^4) \ \ \mbox{at the origin} \\
-4 + O \left( \frac{|\log r|^2}{r^2}\right)  \ \ \mbox{as} \ \ r \rightarrow + \infty
\end{array} \right.
\ee
and
\be
\label{comportementn1'}
|rn'_1| \lesssim r^4 {\bf 1}_{r \leq 1} + \frac{1+|\log r|^2}{r^2} {\bf 1}_{r \geq 1}.
\ee
Thus, we obtain :
\be
\label{comportementS1}
\nabla S_1 = \frac{n_1}{r}= \left \{ \begin{array} {c}
O(r) \ \ \mbox{at the origin} \\
-\frac{4}{r} + O \left( \frac{|\log r|^2}{r^3}\right) \ \ \mbox{as} \ \ r \rightarrow + \infty.
\end{array} \right.
\ee
{\bf Step 2 :} Construction of the radiation
\\
\\
In this step, we construct the term of radiation which allow us to reduce the growth of the profile $T_2$, outside the parabolic zone ($y \geq B_0$). We will see that the choice of the radiation impose the law of $b$, and thus, the dynamics of the blow-up regime.
\par
We let  the radiation ${\bf \Sigma_b} = \begin{array} {|c} \Sigma_{1,b} \\ \nabla \Sigma_{2,b}  \end{array}$ defined by $(m_{{\bf \Sigma_b}},d_{{\bf \Sigma_b}})$ be the solution of
\be
\label{defradiation}
{\bf L} \ \begin{array} {|c} m_{{\bf \Sigma_b}} \\ d_{{\bf \Sigma_b}} \end{array} = c_b\chi_{\frac{B_0}{4}} \begin{array}{| c}
   rm'_0 \\
  rn'_0
  \end{array} + (1 - \chi_{3B_0}) {\bf L} \ \begin{array} {|c} \beta_{1,b} rm'_0 \\ \beta_{2,b} r^2 + \beta_{3,b} \end{array} 
\ee
with $c_b$ must be determined. We will find $c_b$ such that:
\be
 \label{defcb}
 c_b=\frac{2}{|\log b|}\left[ 1 + O \left( \frac{1}{|\log b|}\right)\right].
 \ee
Let
 \bea
 \label{defdb}
 \beta_{2,b} &=&   \int_0^{+\infty} \frac{\psi_0(\tau)}{\tau} \left (1-\chi_{\frac{B_0}{4}}\right)d\tau = O (b),\\
 \beta_{3,b} &=&  \int_0^{+\infty} \tau \psi_0(\tau) \chi_{\frac{B_0}{4}}d\tau = O(|\log b|).
 \eea
A solution of \fref{defradiation} is given by
\be
d_{{\bf \Sigma_b}} = 4 c_b \left[  - \int_0^r \tau \psi_0(\tau) \chi_{\frac{B_0}{4}}d\tau + r^2 \int_0^r \frac{\psi_0(\tau)}{\tau} \chi_{\frac{B_0}{4}}d\tau - \frac12r^2 + (1-\chi_{3B_0}) \left\{ \beta_{2,b} r^2 + \beta_{3,b} \right\} \right]
\ee
and
\bea
m_{{\bf \Sigma_b}}   & = &-4c_b\psi_0(r)\int_0^r\frac{\tau(\tau^4+4\tau^2\log \tau-1)}{(1+\tau^2)^2}\left(\chi_{\frac{B_0}{4}} - \frac{d_{{\bf \Sigma_b}}}{\tau^2}\right)d\tau \\
\nonumber &+&4c_b\psi_1(r)\int_0^r\frac{\tau^3}{(1+\tau^2)^2}\left(\chi_{\frac{B_0}{4}} - \frac{d_{{\bf \Sigma_b}}}{\tau^2}\right)d\tau + \beta_{1,b}(1-\chi_{3B_0})\psi_0(r).
\eea
with
\bea
 \beta_{1,b} &=& 4c_b\int_0^{+\infty}\frac{\tau(\tau^4+4\tau^2\log \tau-1)}{(1+\tau^2)^2}\left(\chi_{\frac{B_0}{4}} - \frac{d_{{\bf \Sigma_b}}}{\tau^2}\right)d\tau =  O \left( \frac{1}{b|\log b|}\right).
\eea
Look for $c_b$ such that:
\be
\label{contraintecb}
c_b\int_0^{+\infty} \frac{\tau^3}{(1+\tau^2)^2}\left(\chi_{\frac{B_0}{4}} - \frac{d_{{\bf \Sigma_b}}}{\tau^2}\right)d\tau = 1.
\ee
Let
\bee
c_{1,b} &=& \int_0^{+\infty} \frac{\tau^3}{(1+\tau^2)^2}\chi_{\frac{B_0}{4}} = \frac{|\log b|}{2} + O(1), \\
c_{2,b} &=& \int_0^{+\infty} \frac{\tau^3}{(1+\tau^2)^2} \frac{d_{{\bf \Sigma_b}}}{c_b \tau^2}d\tau = O(1).
\eee
Then, the following $c_b$ satisfy the constraint \fref{contraintecb}:
\be
c_b = \frac{c_{1,b} - \sqrt {c_{1,b}^2 - 4 c_{2,b}}}{2c_{2,b}} = \frac{2}{|\log b|} + O \left( \frac{1}{|\log b|^2} \right).
\ee
Observe that by definition:
\be
\label{propradiaitonorigine}
\begin{array} {|c} m_{{\bf \Sigma_b}} \\ d_{{\bf \Sigma_b}} \end{array} = c_b \ \begin{array} {|c} m_{1} \\ d_{1} \end{array} \ \ \mbox{for} \ \ r \leq \frac{B_0}{4}
\ee
and
\be
\label{propradiaitoninfty}
\begin{array} {|c} m_{{\bf \Sigma_b}} \\ d_{{\bf \Sigma_b}} \end{array} = \begin{array} {|c} 4 \psi_1 \\ 0 \end{array} \ \ \mbox{for} \ \ r \geq 6B_0.
\ee
Now, we can estimate for $\frac{B_0}{4} \leq r \leq 6 B_0$
\be
\label{propradiaitonmilieu}
\begin{array} {|c} m_{{\bf \Sigma_b}} \\ d_{{\bf \Sigma_b}} \end{array} = \begin{array} {|c} 4 + O\left( \frac1 {|\log b|^2}\right) \\ O(1) \end{array} 
\ee
Now, we are in position to estimate ${\bf \Sigma_b}$ and its derivates. First, by construction, we have
\be
\label{busb}
\begin{array} {|c} \Sigma_{1,b} \\ \nabla \Sigma_{2,b}  \end{array} = c_b \  \begin{array} {|c} T_1 \\ \nabla S_1  \end{array} \ \ \mbox{for} \ \ r \leq \frac{B_0}{4}.
\ee
Moreover, we recall that :
\bee
\Sigma_{1,b}(r) = \frac{m_{{\bf \Sigma_b}}'(r)}{r} \ \ \mbox{and} \ \ \nabla \Sigma_{2,b} = \frac{d_{{\bf \Sigma_b}}(r) + m_{{\bf \Sigma_b}}(r)}{r} 
\eee
Hence, with \fref{propradiaitoninfty} and the above formula, we obtain for $r \geq 6B_0$ :
\be
\label{busbu}
\begin{array} {|c} \Sigma_{1,b} \\ \nabla \Sigma_{2,b}  \end{array} = \begin{array} {|c} 4 \frac{\psi_1'}{r} \\ \frac{4 \phi_1}{r}  \end{array} = \begin{array} {|c} O \left( \frac{\log r}{r^4} \right) \\  \frac 4r +O\left( \frac{\log r} {r^3}\right)  \end{array}
\ee
To obtain a precise bound in the transition zone, we use the improved formula \fref{formulamprime}:
\bea
\nonumber \frac{m'_{\bf \Sigma_b}}{r}&=& -4c_b\frac{\psi'_0(r)}{r}\int_0^r\frac{\tau(\tau^4+4\tau^2\log \tau-1)}{(1+\tau^2)^2}\left(\chi_{\frac{B_0}{4}} - \frac{d_{{\bf \Sigma_b}}}{\tau^2}\right)d\tau-\beta_{1,b}\chi'_{3B_0}\frac{\psi_0}{r} \\
\label{ubc}&+& 4c_b\frac{\psi'_1(r)}{r}\int_0^r\frac{\tau^3}{(1+\tau^2)^2}\left(\chi_{\frac{B_0}{4}} - \frac{d_{{\bf \Sigma_b}}}{\tau^2}\right)d\tau +  \beta_{1,b}(1-\chi_{3B_0})\frac{\psi'_0(r)}{r}
\eea
Hence, for $\frac{B_0}{4} \leq y \leq 6B_0$, \fref{ubc} together with \fref{propradiaitonmilieu} give :
\be
\label{busbus}
|\Sigma_{1,b}(r) | \lesssim \frac{1}{|\log b| r^2} \ \ \mbox{and} \ \ | \nabla \Sigma_{2,b}(r) | \lesssim \frac{1}{r}
\ee
The above estimation with \fref{busb} and \fref{busbu} imply the following rough bounds, for $i \geq 0$ and $r \leq 2B_1$ :
\bea
\label{aaqqzz} |r^i\pa_r^i\Sigma_{1,b}(r) | &\lesssim& \frac{1}{|\log b|}\left[r^2{\bf 1}_{r\leq 1}+\frac{1}{r^2}{\bf 1}_{1\leq r\leq 6B_0}+\frac{1}{br^4}{\bf 1}_{r\geq 6B_0}\right] \\
\label{aazzqq} |r^i\pa_r^i\nabla \Sigma_{2,b}(r) | &\lesssim& \frac{1}{|\log b|}\left[r{\bf 1}_{r\leq 1}+\frac{1}{r}{\bf 1}_{1\leq r\leq \frac{B_0}4}+\frac{|\log b|}{r}{\bf 1}_{r\geq \frac{B_0}4}\right] 
\eea
Before determining $m_2$ and $d_2$, compute the $b$ dependence of ${\bf \Sigma_b}$. First we have :
 $$\frac{\pa c_b}{\pa b}=O\left(\frac{1}{b|\log b|^2}\right),\ \ \frac{\pa \beta_{1,b}}{\pa b}=O\left(\frac{1}{b^2|\log b|}\right),\ \ \frac{\pa \beta_{2,b}}{\pa b} = O(1), \ \ \frac{\pa \beta_{3,b}}{\pa b} = O\left(\frac1b\right). $$
By definition,
\be
\label{busb1}
\frac{\pa}{\pa b} \ \begin{array} {|c} m_{\bf \Sigma_b} \\ d_{\bf \Sigma_b}  \end{array} = \pa_b c_b \  \begin{array} {|c} m_1 \\ d_1  \end{array} =O\left(\frac{1}{b|\log b|^2}\right) \  \begin{array} {|c} m_1 \\ d_1  \end{array} \ \ \mbox{for} \ \ r \leq \frac{B_0}{4}.
\ee
Now
\be
\label{busb2}
\frac{\pa}{\pa b} \ \begin{array} {|c}m_{\bf \Sigma_b} \\ d_{\bf \Sigma_b}  \end{array} = \begin{array} {|c} 0 \\ 0  \end{array} \ \ \mbox{for} \ \ r \geq 6B_0.
\ee
To conclude, in the transition zone $\frac {B_0}4 \leq r \leq 6B_0$
 \bee
\left|\frac{ \pa d_{{\bf \Sigma_b}}}{\pa b} \right| &\lesssim& \frac{1}{b|\log b|} +  |c_b|  \int_0^r \frac 1 \tau \left| \pa_b\chi_{\frac{B_0}{4}} \right| d\tau + |c_b| r^2 \int_0^r \frac 1 {\tau^3} \left| \pa_b\chi_{\frac{B_0}{4}} \right| d\tau \\
&+& |c_b| \left| \frac{\pa \left\{\chi_{3B_0}(\beta_{2,b} r^2 + \beta_{1,b})\right\}}{\pa b} \right|\\
&\lesssim&  \frac{1}{b}.
\eee
Using the same way, we prove that in the transition zone,
 \be
\left|\frac{\pa m_{{\bf \Sigma_b}}}{\pa b} \right| \lesssim  \frac{1}{b|\log b|}.
\ee
This yields the bound for $r \leq 2B_1$ : 
\bea
\label{pabmsigmab} \left|\frac{\pa m_{\bf \Sigma_b}}{\pa b}(r) \right| &\lesssim& \frac{1}{b|\log b|^2}\left[r^4{\bf 1}_{r\leq 1}+(1+\log r) {\bf 1}_{1\leq r\leq 6B_0}  \right] \\
\label{pabdsigmab} \left|\frac{\pa d_{\bf \Sigma_b}}{\pa b} (r)\right | &\lesssim& \frac{1}{b|\log b|^2}\left[r^2{\bf 1}_{r\leq 1}+(1+\log r)^2{\bf 1}_{1\leq r\leq 6B_0} \right] 
\eea
and the following bound for $i \geq 1$ and $r \leq 2B_1$ using the explicit formula of $\psi_1'$, $m_1'$ and $d_1'$: 
\bea
\label{pabmsigmabb} \left|  r^i\pa^i_r \frac{\pa m_{\bf \Sigma_b}}{\pa b}(r) \right| &\lesssim& \frac{1}{b|\log b|^2}\left[r^4{\bf 1}_{r\leq 1}+ {\bf 1}_{1\leq r\leq 6B_0} \right] \\
\label{pabdsigmabb} \left| r^i\pa^i_r \frac{\pa d_{\bf \Sigma_b}}{\pa b} (r)\right | &\lesssim& \frac{1}{b|\log b|^2}\left[r^2{\bf 1}_{r\leq 1}+ (1+\log r)^3{\bf 1}_{1\leq r\leq 6B_0} \right] 
\eea
{\bf Step 3 : } Level $b^2$ \\ \\
We are now in position to determine $(m_2,d_2)$ be solution to
\be
{\bf L} \ \begin{array} {|c}m_2 \\d_2 \end{array} = \begin{array}{| c}
rm'_1- \frac{m'_1m_1}{r} - \frac{m'_1d_1}{r}\\
rn'_1
  \end{array} - \begin{array} {|c} m_{{\bf \Sigma_b}} \\ d_{{\bf \Sigma_b}} \end{array}=  \begin{array} {|c} m_{{\bf \Sigma_2}} \\ d_{{\bf \Sigma_2}} \end{array}
\ee
Hence, from \fref{solutiond} , $d_2$ is given by:
\be
d_{2} = \frac{1}{2} \left[  - \int_0^r d_{{\bf \Sigma_2}}(\tau) \tau d\tau + r^2 \int_0^r \frac{d_{{\bf \Sigma_2}}(\tau)}{\tau} d\tau \right]
\ee
With the estimations \fref{comportementn1'}, \fref{propradiaitonorigine} and \fref{propradiaitonmilieu}, we obtain the rough bound:
\be
\label{dtwo}
|d_2(r)| \lesssim r^2{\bf 1}_{r \leq 1} + r^2(1+\log r) {\bf 1}_{r \geq 1}.
\ee
Moreover, we have the following bound using \fref{pabdsigmab} and \fref{defdeltab}:
\bea
\label{d2pab}
\nonumber | b \pa_b d_{2}(r)| &\lesssim& \frac{1}{2} \left[  \int_0^r \left|b \pa_b d_{{\bf \Sigma_b}}(\tau)\right| \tau d\tau + r^2 \int_0^r \left| \frac{b \pa_b d_{{\bf \Sigma_b}}(\tau)}{\tau} \right| d\tau \right] + |b \pa_b\delta_b| r^2 \\
&\lesssim& \frac{r^2}{|\log b|^2} {\bf 1}_{r \leq 1} + r^2\frac{1+|\log r|^2}{|\log b|^2} {\bf 1}_{1 \leq r \leq 6B_0} + r^2 {\bf 1}_{r \geq 6B_0}  .
\eea
Now $m_2$ is given by:
\bee
m_2(r)&=&-\frac12\psi_0(r)\int_0^r\frac{\tau^4+4\tau^2\log \tau-1}{\tau}\left[Q(\tau)d_2(\tau) - m_{{\bf \Sigma_2}}(\tau)\right]d\tau \\
&+&\frac12\psi_1(r)\int_0^r\tau \left[Q(\tau)d_2(\tau) - m_{{\bf \Sigma_2}}(\tau)\right]d\tau.
\eee
Let
\be
g_b(\tau) = Q(\tau)d_2(\tau) - m_{{\bf \Sigma_2}}(\tau).
\ee
Using the definition of $m_{{\bf \Sigma_2}}$ and the above estimations :
\bee
\left|g_b(r) \right| \lesssim  r^2{\bf 1}_{ r \leq 1} + \frac{1 + |\log (r\sqrt b)|}{|\log b|} {\bf 1}_{1 \leq r \leq 6B_0} + \frac{|\log r|^2}{1+r^2} {\bf 1}_{r \geq 6B_0}.
\eee 
and
\bea
\label{dsigmatwodb}
\left|b\pa_bg_b(r) \right| &\lesssim& \frac{1}{|\log b|^2}\left[r^4{\bf 1}_{r\leq 1}+ (1+ \log r) {\bf 1}_{1\leq r\leq 6B_0} \right] + \frac{1}{r^2} {\bf 1}_{r \geq 6B_0}.
\eea
 Hence, near the origin, we have $$m_2=O(r^4).$$ For $1\leq r\leq 6B_0$, 
 \bee
 |m_2(r)|& \lesssim &\frac{1}{1+r^2}\int_0^r \tau^3\frac{1 + |\log (\tau\sqrt b)|}{|\log b|} d\tau+\int_0^r\tau\frac{1 + |\log (\tau\sqrt b)|}{|\log b|}d\tau\\
 & \lesssim & r^2\frac{1 + |\log (r\sqrt b)|}{|\log b|}.
 \eee
 For $r\geq 6B_0$, 
 \bee
 |m_2(r)|& \lesssim&  \frac{1}{b|\log b|}+\frac{1}{r^2}\int_{6B_0}^r\tau^3\frac{|\log \tau|^2}{1+\tau^2}d\tau+\int_{6B_0}^r\tau\frac{|\log \tau|^2}{1+\tau^2}d\tau\\
 & \lesssim &  \frac{1}{b|\log b|}+|\log r|^3.
 \eee
 Now, outside the parabolic zone, for $ 6B_0\leq r\leq 2B_1$,
 \bea
 \label{imprvedmoneout}
 \nonumber 
\left| \frac{\pa_r m_2}{r} \right|& = &\left|\frac12\frac{\psi'_0(r)}r\int_0^r\frac{\tau^4+4\tau^2\log \tau-1}{\tau}g_b(\tau)d\tau-\frac12\frac{\psi'_1(r)}r\int_0^r\tau g_b(\tau)d\tau\right|\\
\nonumber  & \lesssim & \frac{1}{r^4}\int_0^{r}\frac{\tau^4+4\tau^2\log \tau-1}{\tau}\left[ r^2{\bf 1}_{ r \leq 1} +\frac{1 + |\log (\tau\sqrt b)|}{|\log b|} {\bf 1}_{1 \leq \tau\leq 6B_0}+ \frac{|\log \tau|^2}{1+\tau^2} {\bf 1}_{\tau\geq 6B_0}\right]d\tau\\
\nonumber & + & \frac{|\log r|}{r^4}\int_0^{r}\tau\left[ r^2{\bf 1}_{ r \leq 1} +\frac{1 + |\log (\tau\sqrt b)|}{|\log b|} {\bf 1}_{1 \leq \tau \leq 6B_0}+ \frac{|\log \tau|^2}{1+\tau^2} {\bf 1}_{\tau \geq 6B_0}\right]d\tau\\
\nonumber& \lesssim & \frac{1}{r^4}\left[\frac{1}{b^2|\log b|}+r^2(\log r)^2\right]+\frac{\log r}{r^4}\left[\frac{1}{b|\log b|}+|\log r|^3\right]\\
& \lesssim & \frac{1}{r^4b^2|\log b|}.
 \eea
 The collection of above bounds yields the control: $\forall r\leq 2B_1$
 \be
 \label{estmtwo}
 |m_2|\lesssim r^4{\bf 1}_{r\le1} +r^2\frac{1 + |\log (r\sqrt b)|}{|\log b|}{\bf 1}_{1\leq r\leq 6B_0}+\frac{1}{b|\log b|}{\bf 1}_{r\geq 6B_0},
 \ee
  \be
 \label{estmtwoderivaitve}
 |r^i\pa_r^im_2|\lesssim r^4{\bf 1}_{r\le1} +r^2\frac{1 + |\log (r\sqrt b)|}{|\log b|}{\bf 1}_{1\leq r\leq 6B_0}+ \frac{1}{r^2b^2|\log b|}{\bf 1}_{r\geq 6B_0}, \ \ i\geq 1.
 \ee
 The $b$ dependance is estimated using \fref{dsigmatwodb}:
  \bea
 \label{estmtwoddb}
 \nonumber  |b\pa_bm_2|&\lesssim & \frac{r^2}{1+r^4}\int_0^r\frac{1}{|\log b|^2}\left[\tau^7{\bf 1}_{\tau\leq 1}+ \tau^3(1+ \log \tau) {\bf 1}_{1\leq \tau \leq 6B_0} \right] + \tau {\bf 1}_{\tau \geq 6B_0}d\tau\\
 \nonumber & + & \int_0^r\frac{1}{|\log b|^2}\left[\tau^5{\bf 1}_{\tau\leq 1}+ \tau(1+ \log \tau) {\bf 1}_{1\leq \tau \leq 6B_0} \right] + \frac{1}{\tau} {\bf 1}_{\tau \geq 6B_0} d\tau\\
  & \lesssim & \frac{1}{|\log b|^2}\left[r^6{\bf 1}_{r\leq 1}+r^2(1+ \log r){\bf 1}_{1\leq r\leq 6B_0}+\frac{\log b}b {\bf 1}_{r\geq 6B_0}\right]
  \eea
    and for higher derivatives:
  \bea
 \label{estmtwoddbhd}
  \nonumber \left| \frac{b\pa_b\pa_rm_2}{r} \right|& \lesssim  & \frac{1}{1+r^4}\int_0^r\frac{1}{|\log b|^2}\left[\tau^7{\bf 1}_{\tau\leq 1}+ \tau^3(1+ \log \tau) {\bf 1}_{1\leq \tau \leq 6B_0} \right] + \tau {\bf 1}_{\tau \geq 6B_0}d\tau\\
 \nonumber & + & \frac{1+ |\log r|}{1+r^4} \int_0^r\frac{1}{|\log b|^2}\left[\tau^5{\bf 1}_{\tau\leq 1}+ \tau(1+ \log \tau) {\bf 1}_{1\leq \tau \leq 6B_0} \right] + \frac{1}{\tau} {\bf 1}_{\tau \geq 6B_0}\\
  & \lesssim & \frac{1}{|\log b|^2}\left[r^4{\bf 1}_{r\leq 1}+(1+ \log r){\bf 1}_{1\leq r\leq 6B_0}\right]+\frac{1}{b^2r^4|\log b|} {\bf 1}_{r\geq 6B_0}
\eea
and hence the bounds for $r\leq 2B_1$:
\be
 \label{estmtwodtdb}
 |b\pa_bm_2|\lesssim \frac{1}{|\log b|}\left[r^4{\bf 1}_{r\le1} +r^2\frac{1+|\log r|}{|\log b|}{\bf 1}_{1\leq r\leq 6B_0}+ \frac{1}{b}{\bf 1}_{r\geq 6B_0}\right],
 \ee
 and for $i\geq 1$:
  \be
 \label{estmtwoderivaitvedtdb}
 |b\pa_br^i\pa_r^im_2|\lesssim \frac{1}{|\log b|}\left[r^4{\bf 1}_{r\le1} +r^2\frac{1+|\log r|}{|\log b|}{\bf 1}_{1\leq r\leq 6B_0}+ \frac{1}{r^2b^2}{\bf 1}_{r\geq 6B_0}\right].
 \ee
This yields the estimate on $T_2=\frac{m_2'}{r}$: for $i\geq 0$, $r\leq 2B_1$:
$$
 |r^i\pa_r^iT_2|\lesssim r^2{\bf 1}_{r\le1} +\frac{1 + |\log (r\sqrt b)|}{|\log b|}{\bf 1}_{1\leq r\leq 6B_0}+ \frac{1}{b^2r^4|\log b|}{\bf 1}_{r\geq 6B_0},
$$
$$|b\pa_br^i\pa_r^iT_2|\lesssim\frac{1}{|\log b|}\left[ r^2{\bf 1}_{r\le1} +\frac{ 1+|\log r|}{|\log b|}{\bf 1}_{1\leq r\leq 6B_0}+ \frac{1}{b^2r^4}{\bf 1}_{r\geq 6B_0}\right].
$$
Using the fact that $\nabla S_2 = \frac{d_2(r) + m_2(r)}{r}$, we have the following rough bounds:
 \be
 \label{estnablaS2}
 |r^i \pa_r^i\nabla S_2|\lesssim   r{\bf 1}_{r\leq1} + r (1+ |\log r|) {\bf 1}_{r\geq1}
 \ee
 and 
  \be
 \label{estnablaS2b}
 |r^i \pa_r^i b\pa_b\nabla S_2|\lesssim \frac {1}{|\log b|^2} \left\{  r{\bf 1}_{r\leq1} + r (1+ |\log r|^2) {\bf 1}_{r\geq1} \right\}
 \ee
 {\bf Step 4 : } Estimate on the error: \\ \\
According to the construction of the profiles $T_i$ and $S_i$, $\left |\begin{array}{l} \Phi_b \\ \Omega_b \end{array} \right.$ satisfies the following equation :
\be
\left | \begin{array}{l} \Phi_b \\ \Omega_b \end{array} \right. = - b^2 \left | \begin{array}{l} m_{{\bf \Sigma}_b} \\ d_{{\bf \Sigma}_b}  \end{array} \right. +\left|\begin{array}{l} R^{(1)}(r) \\ R^{(2)}(r)  \end{array} \right.
\ee
where
\be
\left|\begin{array}{l} R^{(1)}(r) \\ R^{(2)}(r)  \end{array} \right. =  b^3 \ \  \begin{array}{| c}
-rm'_2+ \frac{(m_1m_2)'}{r} + \frac{m'_1d_2 + m'_2d_1}{r}\\
-rn'_2
  \end{array} 
+
b^4 \ \  \begin{array}{| c}
\frac{m_2m_2'}{r} + \frac{m'_2d_2}{r}\\
0
  \end{array} .
\ee
Hence, from the definition \fref{eqerreurpsib} of the error, we obtain :
\be
 \begin{array}{| c}
    \Psi_b^{(1)} \\
     \nabla  \Psi_b^{(2)}
  \end{array}  =  
 \frac 1r\begin{array}{| c}
    \Phi_b' \\
  \Omega_b
  \end{array}  
      +c_bb^2
\chi_{\frac{B_0}{4}}\begin{array}{|c}T_1 \\  \nabla S_1 \end{array} = \frac 1r \left|\begin{array}{l} \left(R^{(1)}(r)\right)' \\ R^{(2)}(r)  \end{array} \right. + b^2 \begin{array}{|c}c_bT_1\chi_{\frac{B_0}{4}} - \Sigma_{1,b}  \\  c_b\nabla S_1\chi_{\frac{B_0}{4}} - \nabla \Sigma_{2,b} \end{array}.
\ee
Using \fref{defd1}, \fref{comportementm1}, \fref{dtwo}, \fref{estmtwo} and \fref{estmtwoderivaitve}, we prove the following bound for $i \geq 0$, remarking that the worst term is $-rm'_2$:
\be
\label{estr}
|r^i\pa_r^iR^{(1)}(r)| \lesssim b^3 \left[ r^4{\bf 1}_{r\le1} +r^2\frac{1 + |\log (r\sqrt b)|}{|\log b|}{\bf 1}_{1\leq r\leq 6B_0}+ \frac{1}{r^2b^2|\log b|}{\bf 1}_{r\geq 6B_0}\right].
\ee
Now, the bound \fref{estnablaS2} gives :
\be
|r^i\pa_r^iR^{(2)}(r)| \lesssim   b^3 \left\{r^2{\bf 1}_{r\leq1} + r^2 (1+ |\log r|) {\bf 1}_{r\geq1} \right\}.
\ee
According to the conception of the radiation ${\bf \Sigma}_b$, and the estimations \fref{aaqqzz} and \fref{aazzqq}, this yields the bounds for $r \leq 2B_1$ and $i \geq 0$:
\bea
\nonumber
 |r^i\pa_r^i\Psi^{(1)}_b|& \lesssim & b^3 \left[ r^2{\bf 1}_{r\le1} +\frac{1 + |\log (r\sqrt b)|}{|\log b|}{\bf 1}_{1\leq r\leq 6B_0}+ \frac{1}{r^4b^2|\log b|}{\bf 1}_{r\geq 6B_0}\right]\\
\label{esibgeogogeo} & + &  \frac{b^2}{|\log b|}\left[\frac{1}{r^2}{\bf 1}_{\frac{B_0}{4}\leq r\leq 6B_0}+\frac{1}{br^4}{\bf 1}_{r\geq 6B_0}\right]\\
\nonumber
 |r^i\pa_r^i\nabla \Psi^{(2)}_b|& \lesssim &  b^3 \left[ r{\bf 1}_{r\leq1} + r (1+ |\log r|) {\bf 1}_{r\geq1} \right] +\frac{b^2}r {\bf 1}_{r\geq \frac{B_0}{4}} \\
\label{esibgeogogeo2}
 & \lesssim &  b^3 \left[ r{\bf 1}_{r\leq1} + r (1+ |\log r|) {\bf 1}_{r\geq1} \right]
\eea
We therefore estimate:
\bee
\int_{r\leq 2B_1} |\Psi_b^{(1)}|^2& \lesssim & b^6\int_{r\leq 2B_1}\left[r^4{\bf 1}_{r\leq 1} +\left(\frac{ 1+|\log (r\sqrt b)|}{|\log b|}{\bf 1}_{1\leq r\leq 6B_0}\right)^2+ \frac{1}{r^8b^4|\log b|^2}{\bf 1}_{r\geq 6B_0}\right]\\
& + & \frac{b^4}{|\log b|^2}\int_{r\leq 2B_1}\left[\frac{1}{r^4}{\bf 1}_{\frac{B_0}{4}\leq r\leq 6B_0}+\frac{1}{b^2r^8}{\bf 1}_{r\geq 6B_0}\right]\\
& \lesssim & \frac{b^5}{|\log b|^2}
\eee
and similarily for higher derivatives:
$$\int _{r\leq 2B_1} |r^i\pa^i_r\Psi^{(1)}_b|^2\lesssim  \frac{b^5}{|\log b|^2}.$$
Moreover, using the explicit formula of $\mathcal L^{(1)}$, and the bounds \fref{esibgeogogeo} and \fref{esibgeogogeo2}, we obtain :
\bee
\int_{r \leq 2 B_1} \frac{\left| \mathcal L^{(1)}(\Psi^{(1)}_b, \Psi^{(2)}_b)\right|^2}{Q} &\lesssim& \sum_{i=0}^2 \int_{r \leq 2 B_1} (1+r^{2i}) |\pa_r^i\Psi^{(1)}_b|^2 + \sum_{i=0}^1 \int_{r \leq 2 B_1} \frac{|\pa_r^i\nabla \Psi^{(2)}_b|^2}{1+r^{6-2i}} \\&\lesssim& \frac{b^5}{|\log b|^2}.
\eee
Now, with \fref{esibgeogogeo2} :
\bea
\label{presqueenfin}
\nonumber \int_{r\leq 2B_1} \frac{|\nabla \Psi_b^{(2)}|^2}{1+r^2} + \int_{r\leq 2B_1} |\Delta \Psi_b^{(2)}|^2& \lesssim & b^6\int_{r\leq 2B_1}\left[r^2{\bf 1}_{r\leq 1} + (1+ \log r)^2{\bf 1}_{r\geq 1} \right] \\
 & \lesssim & b^5 |\log b|^2.
\eea
Moreover,
\be
\label{presqueenfin2}
\int_{r\leq 2B_1} |\nabla \Psi_b^{(2)}|^2 \lesssim  b^6\int_{r\leq 2B_1}\left[r^2{\bf 1}_{r\leq 1} + r^2(1+ \log r)^2{\bf 1}_{r\geq 1} \right] \lesssim b^4 |\log b|^6.
\ee
Finally, we estimate the following norm:
\bee
&&\int_{r\leq 2B_1} Q|\nabla \mathcal M^{(1)}(\Psi^{(1)}_b, \Psi^{(2)}_b)|^2\\
&\lesssim & \int_{r\leq 2B_1} (1+r^2)\left[|r\pa_r\Psi^{(1)}_b|^2+|\Psi^{(1)}_b|^2\right]+\frac{|\nabla \Psi^{(2)}_b|^2}{1+r^4}\\
& \lesssim & b^6\int_{r\le 2B_1}  \left[ 1+\frac{(1 + |\log (r\sqrt b)|^2)(1+r^2)}{|\log b|^2}{\bf 1}_{1\leq r\leq 6B_0}+ \frac{1}{r^6b^4|\log b|}{\bf 1}_{r\geq 6B_0}\right]\\
\nonumber & + &  \frac{b^4}{|\log b|^2}\int_{r\le 2B_1}\left[\frac{1}{r^2}{\bf 1}_{\frac{B_0}{4}\leq r\leq 6B_0}+\frac{1}{b^2r^6}{\bf 1}_{r\geq 6B_0}\right]+\frac{b^5}{|\log b|^2}\\
& \lesssim & \frac{b^4}{|\log b|^2}.
\eee
This concludes the proof of the Proposition \ref{construction}.
 \end{proof}
 \subsection{Localization}
 We can see that, outside the parabolic zone $r \geq 2B_1$, the profiles $T_i$ and $S_i$ have a pathological growth. Therefore, we must localize this profiles, as described in the following Proposition.
 \begin{proposition}[Localization]
\label{localization}
Given a small parameter
\be
0 < b \ll 1.
\ee
Let the localized profiles
\be
{\bf \tilde Q}_b =
  \begin{array}{| c}
     \tilde Q_b(r) \\
    \tilde P_b(r)
  \end{array} =   {\bf Q} + b \ \ 
    \begin{array}{| c}
     \tilde T_1(r) \\
     \tilde S_1(r)
  \end{array} + b^2 \ \  
   \begin{array}{ |c}
     \tilde T_2(b,r) \\
    \tilde S_2(b,r)
  \end{array}
  ={ \bf Q} + {\bf\tilde \Upsilon}_b = {\bf Q} +  
    \begin{array}{| c}
    \tilde \alpha_b(r) \\
   \tilde  \gamma_b(r)
  \end{array}
\ee
with
\be
\label{deftildeT}
\tilde T_i = \chi_{B_1} T_i \ \ \mbox{for} \ \ i = 1,2.
\ee
and
\be
\label{deftildeS}
\tilde S_i(r) = \int_0^r \chi_{B_1} \nabla S_i dr \ \ \mbox{for} \ \ i = 1,2.
\ee
Let the error:
 \be
 \label{eqerreurpsibtilde}
{\bf \tilde\Psi}_b =  \begin{array}{| c}
    \tilde\Psi_b^{(1)} \\
     \tilde \Psi_b^{(2)}
  \end{array}  =  \begin{array}{ |c}
   \nabla . \left( \nabla \tilde Q_b + \tilde Q_b\nabla \tilde P_b\right) \\
  \Delta \tilde P_b - \tilde Q_b 
  \end{array} 
- b \Lambda 
  \ \ \begin{array}{ |c}
     \tilde Q_b(r) \\
    \tilde P_b(r)
  \end{array} 
    +c_bb^2\breve{\bf T}  
  \ee
  with $c_b$ given by \fref{defcb}, and $\breve{\bf T}$ satisfied the condition \fref{breveT1} then there holds:\\
{\em (i) Control of the tails}: $ \forall r\geq 0$, $ \forall i\geq 0$:
\bea
\label{esttonepropt}
|r^i\pa^i_r\tilde T_1|&\lesssim& \frac{r^2}{1+r^4}, \\
\label{esttpemfmp}
|b\pa_b r^i\pa^i_r \tilde T_1|&\lesssim& \frac{1}{r^2}{\bf 1}_{B_1\leq r\le 2B_1},\\
\label{nablaS1taillet}
|r^i\pa^i_r\nabla \tilde S_1|&\lesssim& \frac{r}{1+r^2},\\
\label{esttpemfmptt}
|b\pa_b r^i\pa^i_r \tilde S_1|&\lesssim& \frac{1}{r}{\bf 1}_{B_1\leq r\le 2B_1},
\eea
and $ \forall r\leq 2B_1$, $ \forall i\geq 0$:
\bea
 \label{esttwot}
 |r^i\pa_r^i\tilde T_2|&\lesssim& r^4{\bf 1}_{r\le1} +\frac{1 + |\log (r\sqrt b)|}{|\log b|}{\bf 1}_{1\leq r\leq 6B_0}+ \frac{1}{b^2r^4|\log b|}{\bf 1}_{r\geq 6B_0},
\\
\label{estttwodtdbt}
|b\pa_br^i\pa_r^i\tilde T_2|&\lesssim&\frac{1}{|\log b|}\left[ r^4{\bf 1}_{r\le1} +\frac{ |\log r|}{|\log b|}{\bf 1}_{1\leq r\leq 6B_0}+ \frac{1}{b^2r^4}{\bf 1}_{r\geq 6B_0}\right].
\\
 \label{nablaS2part}
 |r^i\pa_r^i\nabla \tilde S_2|&\lesssim& r{\bf 1}_{r\le1} +r(1+ \log r){\bf 1}_{r \geq 1},
\\
\label{nablaS2pabt}
|b\pa_br^i\pa_r^i \nabla\tilde S_2|&\lesssim&\frac{1}{|\log b|^2}\left[  r{\bf 1}_{r\le1} +r(1+ \log r)^2{\bf 1}_{r \geq 1}\right].
\eea
{\em (ii) Control of the error in weighted norms: for} $i\geq 0$,
\bea
\label{roughboundltaowt}
\int |r^i\pa_r^i\tilde \Psi^{(1)}_b|^2+ \int \frac{|\mathcal L^{(1)}(\tilde \Psi^{(1)}_b,\tilde \Psi^{(2)}_b)|^2}{Q}&\lesssim& \frac{b^5}{|\log b|^2}, \\
\label{roughboundltaowt2}
 \int \frac{|\nabla {\tilde \Psi^{(2)}_b}|^2}{1+\tau^2}+ \int |\mathcal L^{(2)}(\tilde \Psi^{(1)}_b,\tilde \Psi^{(2)}_b)|^2&\lesssim& b^5|\log b|^4, \\
\label{cneneoneonoenoet}
\int Q|\nabla \mathcal M^{(1)}(\tilde \Psi^{(1)}_b,\tilde \Psi^{(2)}_b)|^2&\lesssim& \frac{b^4}{|\log b|^2},\\
\label{roughboundltaowt3}
\int  |\nabla \tilde\Psi_b^{(2)}|^2 &\lesssim& b^4 |\log b|^6.
\eea
{ \em{(iii)} Degenarate flux:} Let $\frac{B_0}{20} \leq B \leq 20B_0$ and ${\bf \Phi}_{0,B} = \begin{array}{|l}r^2 \chi_B \\ -4 \int_0^r\frac{\log (1+\tau^2)}{\tau} \chi_B d\tau \end{array}$, then
\be
\left|\left<\mathcal L \bf {\tilde \Psi}_b, {\bf \Phi}_{0,B} \right>\right| \lesssim \frac{b^2}{|\log b|}.
\ee

\end{proposition}
\begin{remark}
We have fixed $\tilde S_i(0) = 0$. In fact, it isn't a necessity for our analysis, but there is an advantage. We have the equality $\phi_{\Delta \tilde P_b} = \tilde P_b$.
\end{remark}
\begin{proof}
{\bf Step 1} Terms induced by the localization : \\
From \fref{deftildeT} and \fref{deftildeS}, we have :
$$|\pa_r \tt_i|=\left|\chi_{B_1}\pa_r T_i + \frac{1}{B_1} \chi' \left(\frac{y}{B_1}\right) T_i\right|\lesssim | \pa_rT_i|{\bf 1}_{r\leq 2B_1}+\left|\frac{T_i}{y}\right|{\bf 1}_{B_1\leq r\leq 2B_1},$$
$$|\pa_r\nabla \tilde S_i|=\left|\chi_{B_1}\pa_r \nabla S_i + \frac{1}{B_1} \chi' \left(\frac{y}{B_1}\right) \nabla S_i\right|\lesssim | \pa_r \nabla S_i|{\bf 1}_{r\leq 2B_1}+\left|\frac{\nabla S_i}{y}\right|{\bf 1}_{B_1\leq r\leq 2B_1},$$
and
$$|b\pa_b\tt_i|=\left|b\chi_{B_1}\pa_bT_i-b\frac{\partial_b B_1}{B_1}(y\chi')\left(\frac{y}{B_1}\right)T_i\right|\lesssim |b\pa_bT_i|{\bf 1}_{r\leq 2B_1}+|T_i|{\bf 1}_{B_1\leq r\leq 2B_1},$$
$$|b\pa_b\nabla \tilde S_i|=\left|b\chi_{B_1}\pa_b\nabla S_i-b\frac{\partial_b B_1}{B_1}(y\chi')\left(\frac{y}{B_1}\right)\nabla S_i\right|\lesssim |b\pa_b\nabla S_i|{\bf 1}_{r\leq 2B_1}+|\nabla S_i|{\bf 1}_{B_1\leq r\leq 2B_1}.$$

Using with the bounds of the Proposition \ref{construction} yields \fref{esttonepropt}, \fref{esttpemfmp}, \fref{nablaS1taillet}, \fref{esttpemfmptt}, \fref{esttwot}, \fref{estttwodtdbt}, \fref{nablaS2part} and \fref{nablaS2pabt}. \par
Now, let's focus on the control of the error $\bf {\tilde \Psi}_b$. In this purpose, we recall the definition of the partial mass of respectively $Q_b$ and $\Delta P_b$ :
\bee
m_b(r) = \int_0^r Q_b(\tau) \tau d\tau, \ \ n_b(r) = \int_0^r \Delta P_b(\tau) \tau d\tau.
\eee
Similarly, we define now the partial mass of respectively $\tilde Q_b$ and $\Delta \tilde P_b$ by:
\bee
\tilde m_b(r) &=& m_0 + m_{\tilde \alpha_b} =  \int_0^r \tilde Q_b(\tau) \tau d\tau, \\ 
\tilde n_b(r) &=&  n_0 + n_{\tilde \gamma_b} = \int_0^r \Delta \tilde P_b(\tau) \tau d\tau.
\eee
Hence, we have :
\bee
m'_{\tilde \alpha_b} = \chi_{B_1} m'_{\alpha_b} \ \ \mbox{and} \ \ n_{\tilde \gamma_b} = \chi_{B_1} n_{\gamma_b}.
\eee
In the same way as follows, we can rewrite \fref{eqerreurpsibtilde} by 
  \be
 \label{eqerreurpsibtilde}
{\bf \tilde \Psi}_b   =\begin{array}{| c}
     \tilde\Psi_b^{(1)} \\
     \nabla   \tilde\Psi_b^{(2)}
  \end{array}  =  
 \frac 1r\begin{array}{| c}
     \tilde\Phi_b' \\
   \tilde\Omega_b
  \end{array}  
      +c_bb^2 \breve{\bf T}
         \ee
where
\bea
\nonumber
 \begin{array}{| c}
   \tilde\Phi_b \\
    \tilde\Omega_b
  \end{array} &=& 
  \begin{array}{ |c}
    \tilde m''_b - \frac{ \tilde m'_b}r + \frac{ \tilde m'_b \tilde n_b}{r^2} - br \tilde m'_b \\
   ( \tilde n_b- \tilde m_b)'' - \frac{( \tilde n_b -  \tilde m_b)'}{r} - br \tilde n'_b
  \end{array}\\
  \nonumber 
  &=& 
  \begin{array}{ |c}
      m''_{\tilde \alpha_b} - \frac{   m'_{\tilde \alpha_b}}r + \frac{m'_0   n_{\tilde \gamma_b} +    m'_{\tilde \alpha_b} n_0 +    m'_{\tilde \alpha_b}   n_{\tilde \gamma_b}}{r^2} - br (m'_0 +   m'_{\tilde \alpha_b}) \\
      n''_{\tilde \gamma_b}-   m''_{\tilde \alpha_b} - \frac{  n'_{\tilde \gamma_b} -    m'_{\tilde \alpha_b}}{r} - br (n'_0 +   n'_{\tilde \gamma_b})
  \end{array}
 \\
 \label{gepxxx}
&=&
\chi_{B_1} 
\begin{array}{| c}
   \Phi_b \\
    \Omega_b
  \end{array}
+
\begin{array}{| c}
   J_{1,b} \\
   J_{2,b}
  \end{array}
\eea
where
\bee
\begin{array}{| c}
   J_{1,b} \\
   J_{2,b}
  \end{array}
  = -br (1-\chi_{B_1})
  \begin{array}{| c}
   m'_0 \\
   n'_0
  \end{array}
+ \chi'_{B_1}
\begin{array}{| c}
   m'_{\alpha_b} \\
   -m'_{\alpha_b} + 2 n'_{\gamma_b} - \left( br + \frac{1}{r}\right)n_{\gamma_b} 
  \end{array}
  + 
  \begin{array}{| c}
  (\chi^2_{B_1} -  \chi_{B_1} ) \frac{m'_{\alpha_b}n_{\gamma_b}}{r^2} \\
   \chi''_{B_1} n_{\gamma_b}
  \end{array}
  \eee
We recall that :
\bee
rm'_0 = rn'_0 = \frac{8r^2}{(1+r^2)^2}.
\eee
Now, with the bounds \fref{boundT_1}, and \fref{estmtwoderivaitve} :
\bee
\left| m'_{\alpha_b} \right| \lesssim \frac br {\bf 1}_{B_1 \leq r \leq 2B_1} + b^2 \frac{1}{r^3b^2|\log b|} {\bf 1}_{B_1 \leq r \leq 2B_1} \lesssim \frac br {\bf 1}_{B_1 \leq r \leq 2B_1},
\eee
and with the estimate \fref{comportementn1}, \fref{comportementn1'} and \fref{estnablaS2}, for $i \geq 0$ :
\bee
\left|r^i \pa_r^i n_{\gamma_b} \right| \lesssim  b {\bf 1}_{B_1 \leq r \leq 2B_1} + b^2 r^2 \log r {\bf 1}_{B_1 \leq r \leq 2B_1} \lesssim b^2 r^2 \log r {\bf 1}_{B_1 \leq r \leq 2B_1}.
\eee
Using the following bounds, we obtain, for $i \geq 0$ :
\bea
\label{wapiti} \left | r^i \pa_r^i J_{1,b} \right | &\lesssim& \frac b{r^2} {\bf 1}_{r \geq B_1}\\
\label{wapitii} \left | r^i \pa_r^i J_{2,b} \right | &\lesssim& \frac b{r^2}\left\{ |\log b|^5 {\bf 1}_{B_1 \leq r \leq 2B_1}  + {\bf 1}_{r \geq 2B_1} \right\} 
\eea
Finally, from the definition  \fref{eqerreurpsibtilde} and \fref{gepxxx} :
  \bee
{\bf \tilde \Psi}_b   &=&  
 \frac 1r\begin{array}{| c}
     \chi_{B_1} \Phi_b' +  \chi'_{B_1} \Phi_b + J'_{1,b} \\
   \chi_{B_1} \Omega_b + J_{2,b}
  \end{array}  
      +c_bb^2
\chi_{\frac{B_0}{4}}\begin{array}{|c}T_1 \\  \nabla S_1 \end{array} \\
&=& 
\chi_{B_1}{\bf \Psi}_b +   \frac 1r\begin{array}{| c}
 \chi'_{B_1} \Phi_b + J'_{1,b} \\
   J_{2,b}
  \end{array}  
  \eee
  We now estimate from \fref{esibgeogogeo}, \fref{esibgeogogeo2}, \fref{wapiti} and \fref{wapitii} : 
\bee
\left | \tilde \Psi^{(1)}_b - \chi_{B_1} \Psi^{(1)}_b\right| &\lesssim&  \frac {b^2}{r^2}{\bf 1}_{B_1 \leq r \leq 2B_1} + \frac b{r^4}{\bf 1}_{r \geq B_1} \\
\left | \nabla \tilde \Psi^{(2)}_b - \chi_{B_1} \nabla \Psi^{(2)}_b\right| &\lesssim& \frac b{r^3}\left\{ |\log b|^5 {\bf 1}_{B_1 \leq r \leq 2B_1}  + {\bf 1}_{r \geq 2B_1} \right\} 
\eee
 and hence using \fref{roughboundltaow}:
\bee
\int |r^i\pa_r^i\tilde \Psi_{b}^{(1)}|^2\lesssim \frac{b^5}{|\log b|^2}+\int_{r\geq B_1}\left[\frac{b^4}{r^4}+\frac{b^2}{r^8}\right]\lesssim \frac{b^5}{|\log b|^2}.
\eee
Moreover :
\bee
\int \frac{\left| \mathcal L^{(1)}(\tilde \Psi^{(1)}_b, \tilde \Psi^{(2)}_b)\right|^2}{Q} &\lesssim& \sum_{i=0}^2 \int (1+r^{2i}) |\pa_r^i \tilde \Psi^{(1)}_b|^2 + \sum_{i=0}^1 \int \frac{|\pa_r^i\nabla \tilde \Psi^{(2)}_b|^2}{1+r^{6-2i}} \\&\lesssim& \frac{b^5}{|\log b|^2} + \int_{r\geq B_1}\left[\frac{b^4}{r^4}+\frac{b^2}{r^8}\right] +\int_{r\geq B_1}\frac{b^2|\log b|^{10}}{r^{12}} \lesssim \frac{b^5}{|\log b|^2}.
\eee
We estimate from \fref{presqueenfin} 
\bee
\int \frac{|\nabla\tilde \Psi_b^{(2)}|^2}{1+r^2} + \int |\Delta \tilde \Psi_b^{(2)}|^2& \lesssim & b^5 |\log b|^2 + \int_{r \geq B_1} \frac{b^2 |\log b|^{10}}{r^8}\\
& \lesssim & b^5 |\log b|^4
\eee
and from \fref{presqueenfin2} :
\bee
\int |\nabla\tilde \Psi_b^{(2)}|^2& \lesssim & b^4 |\log b|^6 + \int_{r \geq B_1} \frac{b^2 |\log b|^{10}}{r^6}\\
& \lesssim & b^4 |\log b|^6.
\eee
To conclude this step of the proof :
\bee
\int Q|\nabla \mathcal M^{(1)}(\tilde \Psi^{(1)}_b, \tilde \Psi^{(2)}_b)|^2
&\lesssim & \int (1+r^2)\left[|r\pa_r \tilde \Psi^{(1)}_b|^2+|\tilde \Psi^{(1)}_b|^2\right]+\frac{|\nabla \tilde \Psi^{(2)}_b|^2}{1+r^4}\\
& \lesssim & \frac{b^4}{|\log b|^2} + \int_{r\geq B_1}\left[\frac{b^4}{r^2}+\frac{b^2}{r^6}\right] +b^5 |\log b|^4\\
&\lesssim& \frac{b^4}{|\log b|^2}.
\eee
{\bf Step 2 :} Degenerate flux.
\\
\par
To prove the estimation of the degenarate flux, we use the rough bound \fref{firstroughbound} together the cancellation \fref{estvooee} to obtain :
\bee
\left|\left<\mathcal L \bf {\tilde \Psi}_b, {\bf \Phi}_{0,B} \right>\right| &\lesssim& B \|\mathcal L \tilde {\bf \Psi}_b\|_{X_Q}^2\\
&\lesssim& \frac{1}{\sqrt b} \left\{ \int \frac{\left| \mathcal L^{(1)}(\tilde \Psi^{(1)}_b, \tilde \Psi^{(2)}_b)\right|^2}{Q} +\int \frac{|\Delta \tilde \Psi_b^{(2)}|^2}{1+r^2} + \int |\tilde \Psi_{b}^{(1)}|^2 \right\} \\
&\lesssim& \frac{b^{\frac 52}}{\sqrt b|\log b|} \lesssim \frac{b^2}{|\log b|}.
\eee
This concludes the proof of the Proposition \ref{localization}.
\end{proof}
\section{Control of the perturbation of the approximate profile}
\label{controlperturbation}
After the construction of approximate profile in the above section, the goal of the rest of this paper is to prove the existence of an open set $\mathcal O$ of initial data, such that we can split the corresponding solution to the problem \fref{kps}  in two parts. The first is the approximate profile. The second is an error term, which doesn't perturb the blow-up dynamics whose we have formally predicted by the construction. In this purpose, we use a bootstrap argument implementing an energy method. 
\subsection{On the open set of initial data}
In this subsection, we describe the open set $\mathcal O$ of initial data, whose corresponding strong solution to \fref{kps} satisfy the Theorem \ref{thmmain}. \par
To begin, the following lemma gives the uniqueness of the decomposition of the solution, under orthogonality conditions.
\begin{lemma}{$L^1\times \dot H^1$ modulation}
\label{uniqueness}
Let $M>M^*$ large enough, then there exists a universal constant $\delta^*(M)>0$ such that $\forall (v,w) \in L^1 \times \dot H^1$ with
\bee
\|v-Q\|_{L^1} + \| \nabla w-\nabla \phi_Q\|_{L^2} < \delta^*(M),
\eee
there exists a unique decomposition
\bee
\left| \begin{array}{l} v \\\nabla w \end{array}\right. = \left| \begin{array}{l} \frac{1}{\l_1^2}\left(\tilde Q_b + \e_1\right) \left( \frac{r}{\l_1}\right) \\ \frac 1{\l_1}\left(\nabla \tilde P_b + \nabla \eta_1\right) \left( \frac{r}{\l_1}\right)  \end{array}\right.
\eee
such that
\bee
\left<\boldsymbol \e_1,\boldsymbol \Phi_M \right> = \left<\boldsymbol \e_1,\mathcal L^* \boldsymbol \Phi_M \right> =0.
\eee
Moreover,
\bee
\| \e_1\|_{L^1} + \| \nabla \eta_1\|_{L^2} + |\l_1 - 1| + |b| \lesssim c(M)\delta^*.
\eee
\end{lemma}
\begin{proof}
To prove this lemma, we use the implicit function theorem. Indeed, consider the $\mathcal C^1$ functional
\bee
F(v,\l_1,b) = \left[ \left<\boldsymbol v_{Mod},\boldsymbol \Phi_M \right>,\left<\boldsymbol v_{Mod}, \mathcal L^* \boldsymbol \Phi_M \right>\right],
\eee
where
\bee
\boldsymbol v_{Mod} = \left| \begin{array}{l} v_{Mod}  \\ \nabla w_{Mod} \end{array}\right.= \left| \begin{array}{l} \l_1^2v(\l_1x) - \tilde Q_b \\ \l_1\nabla w(\l_1x) - \nabla \tilde P_b \end{array}\right.
\eee
and
\bee
Mod = (\l_1,b).
\eee
By definition, we have $F({\bf Q},1,0) = 0$. Now, we compute the Jacobian at this point, using \fref{orthophim}:
\bea
\nonumber
&&\left|\begin{array}{ll}\left<\frac{\pa}{\pa\l_1}(\boldsymbol v_{Mod}),\Phi_M\right>& \left<\frac{\pa}{\pa b}(\boldsymbol v_{Mod}),\Phi_M\right> \\
\left<\frac{\pa}{\pa\l_1}(\boldsymbol v_{Mod}),\mathcal L^*\Phi_M\right>& \left<\frac{\pa}{\pa b}(\boldsymbol v_{Mod}),\mathcal L^*\Phi_M\right> \end{array}\right |_{(\boldsymbol v,Mod) = (Q,1,0)}\\
& = & \left|\begin{array}{ll} \left<-{\bf \Lambda Q},\boldsymbol \Phi_M \right>& 0 \\ 0 & \left<\boldsymbol T_1,H\boldsymbol \Phi_M \right>\end{array}\right |\\
\label{resjacobian}& = & -\left< {\bf \Lambda Q} ,\boldsymbol \Phi_M \right>^2 =  \left( -32 \pi \log M + O(1)\right)^2 > 0,
\eea 
for M large enough. We can apply the implicit function theorem and this concludes the proof of the lemma \ref{uniqueness}.
\end{proof}
Now, we are in position to describe a open set  $\mathcal O$ of initial data $u_0$  whose the corresponding strong solution to \fref{kps} satisfy the dynamics describes in the Theorem \ref{thmmain}.  
\begin{definition}[Description of the open set $\mathcal O$ of initial data]
Let $M>M^*$ large enough. Let $\alpha^*(M)$ small enough. Pick $\alpha^*$ such that $0< \alpha^*<\alpha^*(M)$. Then, we let $\mathcal O$ the set of initial data of the form
\be
\label{forminitialdata}
\left|\begin{array}{l}u_0 \\v_0\end{array}\right. = \left|\begin{array}{l}\frac{1}{\l_0^2} \{\tilde Q_{b_0} + \e_0 \}\left(\frac r{\l_0}\right) \\  \{\tilde P_{b_0} + \eta_0 \}\left(\frac r{\l_0}\right) \end{array}\right.
\ee
where the perturbation $(\e_0,\eta_0)$ satisfies :
\begin{itemize}
 \item Orthogonality conditions
\be
 \label{initicondit}
\left< \boldsymbol \e_0,\boldsymbol \Phi_M \right> = \left< \boldsymbol \e_0, \mathcal L^* \boldsymbol \Phi_M \right> = 0.
\ee
 \item Positivity: $$u_0>0, \ \ v_0>0.$$
  \item Small super critical mass: 
 \be
 \label{intiialmass}
 \int Q<\int u_0<\int Q+\alpha^*.
 \ee
 \item Positivity and smallness of $b_0$: 
 \be
 \label{positivitybzero} 0<b_0<\delta(\alpha^*).
 \ee
 \item Initial smallness:
 \bea
 \label{initialsmalneesb}
\|(1+y)\nabla \Delta \eta_0\|_{L^2}+ \|\Delta \eta_0\|_{L^2}+ \|\eta_0\|_{\dot H^1} &<&b_0^{10}\\
 \|\boldsymbol \e_0\|_{X_Q}&<&b_0^{10}.
 \eea
\end{itemize}
\end{definition}
\begin{remark}
We choose $\alpha^*(M)$ small enough in order to have the uniqueness of the decomposition \fref{forminitialdata}.
\end{remark}
\subsection{The bootstrap argument}
\label{apzoeirutymm}
Let $\boldsymbol u_0\in \mathcal O$, and $\boldsymbol u\in \mathcal C([0,T[,\mathcal E)$ be the corresponding strong solution to \fref{kps}. As $\boldsymbol u_0$ is a very small perturbation in $L^1\times \dot H^1$ of the soliton, we can apply the lemma \fref{uniqueness}, and thus, the solution $\boldsymbol u$ admits a unique decomposition on some small time interval $[0,T^*)$
\be
 \label{decompe}
\left|\begin{array}{l}u(t,r) \\v(t,r)\end{array}\right. = \left|\begin{array}{l}\frac{1}{\l(t)^2} \{\tilde Q_{b(t)} + \e(t)\}\left(\frac r{\l(t)}\right) \\  \{\tilde P_{b(t)} + \eta(t) \}\left(\frac r{\l(t)}\right) \end{array}\right.
\ee
where the error term $\boldsymbol \e(t)$ satisfies the orthogonality conditions
\be
 \label{orthoe}
\left< \boldsymbol \e(t),\boldsymbol \Phi_M \right> = \left< \boldsymbol \e(t), \mathcal L^* \boldsymbol \Phi_M \right> = 0.
\ee
Moreover, using a similar argument that Martel and Merle in \cite{martel2001}, we prove that the geometrical parameters $(\l(t),b(t))$ are nonnegative continuous function.
\par
The goal of the rest of the paper is to prove that the error term $\boldsymbol \e(t)$ doesn't perturb the blow-up dynamics formally predicted by the construction of the approximate profile ${\bf \tilde Q}_b$. In this purpose, we shall use a bootstrap argument to prove that the error term remains very small in suitable norms with respect to $b(t)$. Thus, we shall obtain bounds on the error made on the equation of modulation parameters $(\l,b)$. We shall be in position to conclude the proof of the Theorem \ref{thmmain}.
Using the continuity of geometrical parameters together with the initial smallness assumption, we may assume on $[0,T^*)$ the bootstrap bounds : 

 \begin{itemize}
 \item Positivity and smallness of $b$: 
 \be
 \label{positiv}
0<b(t)<10b_0.
 \ee
 \item $L^1$ bound:
\be
\label{bootsmalllone}
\int |\e(t)|<(\delta^*)^{\frac 14}.
\ee
\item Control of $\boldsymbol \e$ in smoother norms: 
\bea
\label{bootsmallh1}
\|\eta(t)\|^2_{\dot H^1} &\leq& K^* b^2(t) |\log b(t)|^6,\\
\label{bootsmallh2}
 \|\Delta \eta(t)\|^2_{L^2}&\leq& K^*\frac{b^2(t)}{|\log b(t)|},\\
\label{bootsmallh3}
\|(1+y)\nabla \Delta \eta(t)\|^2_{L^2}&\leq&K^* \frac{b(t)}{\sqrt{|\log b(t)|}},\\
\label{bootsmallh2q}
\| {\bf E}_2(t)\|^2_{X_Q}&\leq& K^*\frac{b^3(t)}{|\log b(t)|^2},
\eea
\end{itemize}
where ${\bf E}_2 = \mathcal L (\e,\eta)$, $K^*=K^*(M)$ is a large enough constant and $\delta^* = \delta(\alpha^*)$ is a small enough constant such that:
\bee
\delta(\alpha^*) \rightarrow 0 \ \ \mbox{as} \ \ \alpha^* \rightarrow 0.
\eee
The following proposition prove that the regime is trapped, and thus $T^*=T$. 
\begin{proposition}[Trapped regime]
\label{propboot}
Assume that $K^*$ in \fref{bootsmallh2q}  has been chosen large enough, then $\forall t\in [0,T^*)$:
\bea
 \label{positivboot}
0<b(t)&<&2b_0, \\
\label{bootsmallloneboot}
\int |\e(t)|&<&\frac12(\delta^*)^{\frac 14}, \\
\label{bootsmallh1boot}
\|\eta(t)\|^2_{\dot H^1} &\leq& \frac{K^*}2 b^2(t) |\log b(t)|^6,\\
\label{bootsmallH2}
\|\Delta \eta(t)\|^2_{L^2}&\leq& \frac{K^*}{2}\frac{b^2(t)}{|\log b(t)|},\\
\label{bootsmallH3}
\|(1+y)\nabla \Delta \eta(t)\|^2_{L^2}&\leq&\frac{K^*}{2} \frac{b(t)}{\sqrt{|\log b(t)|}},\\
\label{bootsmallh2qboot}
\|{\bf E}_2(t)\|^2_{X_Q}&\leq& \frac{K^*}{2}\frac{b^3(t)}{|\log b(t)|^2}.
\eea
\end{proposition}
The rest of the paper is organize as follows : In the rest of this section, we shall set up the equation verify by the error term, and we compute the rough modulation equation, coming from our choice of orthogonality conditions \fref{orthoe}. We shall introduce a second decomposition for technical reason,  and we shall are in position to obtain sharp modulation equation. The section \ref{lyapounov} is devoted to the proof of the suitable Lyapounov functionnal at the $X_Q$ level, which is the heart of the energy method to prove \fref{bootsmallh2qboot}. Finally,  we shall prove the Proposition \ref{propboot} in the section \ref{proofpropboot}. We shall have therefore all tools to prove the Theorem \ref{thmmain} in the last section \ref{lastsection}.
\subsection {Equation of the error term}
We recall the space time renormalization :
\be
 \label{defst}
 s(t)=\int_0^t\frac{d\tau}{\l^2(\tau)}, \ \ y=\frac{r}{\lambda(t)}.
 \ee
and the notation
 \be
 \label{resacling}
{\bf f}_{\lambda}(t,r)= \left| \begin{array}{l} f_\l (t,r)\\ g_\l(t,r) \end{array} \right. =  \left| \begin{array}{l} \frac{1}{\l^2}f(s,y) \\ g(s,y)\end{array} \right.
 \ee
 which leads to: 
\be\label{neneoneonoeo}
\pa_t  \left| \begin{array}{l} f_\l \\ g_\l \end{array} \right.=  \frac{1}{\l^2} \left| \begin{array}{l}\left[\pa_sf-\lsl \Lambda f\right]_\l\\ \left[\pa_sg-\lsl \Lambda g\right]_\l. \end{array} \right.
\ee
Let the renormalized flow :
 \bee
 \left| \begin{array}{l} u_\l (t,r)\\ v_\l(t,r) \end{array} \right. =  \left| \begin{array}{l} f_\l (s,y)\\ g_\l(s,y) \end{array} \right..
\eee
Then if $(u,v)$ is a solution of the problem \fref{kps}, $(f,g)$ satisfy the following system :
\bee
\left\{ \begin{array}{l}\pa_s f - \frac{\l_s}{\l} \Lambda f = \nabla.(\nabla f + f \nabla g) \\ \pa_s g - \frac{\l_s}{\l} \Lambda g = \Delta g - f \end{array} \right.
\eee
Using the unique decomposition \fref{decompe} for the solutions whose initial data are in $\mathcal O$, we decompose $(f,g)$ as following
\be
\label{defE}
\begin{array}{| c} f(s,y) \\ g(s,y) \end{array} = \tilde {\bf Q}_{b(s)}  + {\bf E}(s,y) =  \begin{array}{| l} \tilde{Q}_{b(s)} + \e(s,y) \\ \tilde{P}_{b(s)} + \eta(s,y) \end{array}.
\ee
Hence ${\bf E}$ is solution of the following equation:
\be
\label{equationE}
\pa_s {\bf E} - \frac{\l_s}{\l} \Lambda {\bf E} = \mathcal L {\bf E} + {\bf F}  +  \widetilde{ \bf Mod} + {\bf G} =  \mathcal L {\bf E} +  \boldsymbol{\mathcal F},
\ee
where
\be
\label{defF}
{\bf F} = \tilde{\bf \Psi}_b +  {\bf \Theta}_b(\e,\eta) + {\bf N}(\e,\eta)
\ee
with $\tilde{\bf \Psi}_b$  is defined by \fref{eqerreurpsibtilde} and
\bea
\label{defTheta}
 {\bf \Theta}_b(\e,\eta) &=&  \begin{array}{| c} \nabla. \left( \e \nabla \tilde{\gamma}_b + \tilde{\alpha}_b \nabla \eta \right) \\0 \end{array}, \\
\label{defN}
 {\bf N}(\e,\eta) &=&  \begin{array}{| c} \nabla. \left( \e  \nabla \eta \right) \\0 \end{array}, \\
\label{defMod}
\widetilde{\bf Mod} (y,s) &=&  \left( b + \frac{\l_s}{\l}\right) \Lambda \tilde {\bf Q}_b - \pa_s \tilde {\bf Q}_b, \\
\label{defG}
{\bf G} &=& -c_bb^2 \breve{\bf T}.
\eea
In the same way, using the original variables, we can rewrite the decomposition by
\be
\label{defW}
{\bf U} = \left( \tilde {\bf Q}_b \right)_{\l} + {\bf W} = \left(  \ \  \begin{array}{| c} Q + \tilde{\alpha}_b \\ \phi_Q + \tilde{\gamma}_b \end{array} \right)_{\l} + \begin{array}{| c} w \\z \end{array} .
\ee
Hence ${\bf W}$ verifies :
\be
\label{equationW}
\pa_t {\bf W} = \mathcal L_{\l} {\bf W} + \frac{1}{\l^2}  \boldsymbol{\mathcal F}_{\l}
\ee
with
\be
 \mathcal L_{\l} {\bf W} = \begin{array}{| c} \nabla .(Q_\l \nabla \mathcal M^{(1)}_\l (w,z)) \\ \Delta \mathcal M^{(2)}_\l (w,z)  \end{array}, \ \ {\mathcal M}_\l ({\bf W}) = \begin{array} {| c} \frac{w}{Q_{\l}} + z \\ w - \phi_z \end{array}
\ee
\subsection{First modulation equations}
We are in position to compute the two modulation equation of $(b,\l)$, projecting the equation \fref{equationE} respectively onto $\Phi_M$ and $\mathcal L \Phi_M$. 
\begin{lemma}[Rough control of the modulation parameters]
\label{lemmeparam}
There exist an universal constant $C(M)$ large enough, independent of the bootstrap constant $K^*(M)$ such that: 
\be
\label{estlambda}
\left|\lsl+b\right|\lesssim C(M)\frac{b^2}{|\log b|},
\ee
\be
\label{estb}
|b_s|\lesssim b^{\frac 32}.
\ee
\end{lemma} 
\begin{remark} Note that \fref{estlambda}, \fref{estb} imply the bootstrap bound: 
\be
\label{cnkonenoeoeoi3poi}
\left|\lsl+b\right|+|b_s|\lesssim b^{\frac32}.
\ee
\end{remark}
\begin{proof}
Let 
\be
V = \left| \frac{\l_s}{\l} + b \right| + \left|b_s\right|
\ee
{\bf Step 1 :} Projection onto $\Phi_M$ \\
We project the equation \fref{equationE} onto $\Phi_M$. Using the orthogonality conditions \fref{orthoe}, we obtain:
\be
\label{projectontophim}
-\lsl \left< \Lambda {\bf E},\boldsymbol\Phi_M\right> = \left< {\bf F},\boldsymbol\Phi_M\right> + \left< \widetilde{\bf Mod},\boldsymbol\Phi_M\right> + \left< {\bf G},\boldsymbol\Phi_M\right>.
\ee
From \fref{defphim}, the function $\Phi^{(1)}_M$ and $\nabla \phi^{(2)}_M$ are compactly supported in $r \leq 2M$. Using the interpolation bound \fref{coercbase} and the bootstrap bound \fref{bootsmallh2q} , we obtain:
\bee
\left| -\lsl \left< \Lambda {\bf E},\boldsymbol\Phi_M\right>  \right| &\lesssim& \left( b + V\right) \left\{ \int \nabla.(y\e)\Phi_M^{(1)} + \int \nabla(y\nabla\eta)\nabla \Phi_M^{(2)} \right\} \\
&\lesssim& \left( b + V\right) C(M) \left\{ \left(\int \e^2\right)^{\frac12} + \left(\int \frac{|\nabla\eta|^2}{1+y^4}\right)^{\frac12} \right\} \\
&\lesssim&  \left( b + V\right) C(M) \frac{b^{\frac32}}{|\log b|}.
\eee
Now from the bound \fref{roughboundltaowt} and \fref{roughboundltaowt2},
\bee
\left| \left< \Psi_b,\boldsymbol \Phi_M \right> \right| &=& \left| \int \tilde \Psi^{(1)}_b \Phi_M^{(1)} + \int \nabla \tilde \Psi^{(2)}_b \nabla\Phi_M^{(2)} \right| \\
&\lesssim& C(M) \left[ \left(\int \left|\tilde \Psi^{(1)}_b\right|^2\right)^\frac{1}{2} + \left(\int \frac{|\nabla\tilde \Psi^{(2)}_b|^2}{1+y^2}\right)^\frac{1}{2}\right] \\
&\lesssim& C(M) b^{\frac52} |\log b|^2
\eee
The bootstrap bound \fref{bootsmallh2q} yields :
\bee
\left| \left< {\bf N},\boldsymbol \Phi_M \right> \right| &=& \left| \int \nabla.(\e\nabla \eta) \Phi_M^{(1)}\right| \lesssim C(M) \left(\int \e^2\int \frac{|\nabla\eta|^2}{1+y^4}\right)^{\frac12} \lesssim C(M)\frac{b^{\frac32}}{|\log b|}.
\eee
In the same way, we prove that : $\left|\left<\boldsymbol \Theta_b, \boldsymbol \Phi_M \right>\right| \lesssim C(M)\frac{b^{\frac32}}{|\log b|}$.
Now, using the definition of $\boldsymbol\Phi_M$ \fref{defphim}, and the decay of $T_1$ \fref{esttonepropt}, we obtain :
\bee
\left| \left< {\bf G},\boldsymbol \Phi_M \right> \right| &=& \left|c_bb^2\left\{ \int T_1 \chi_{\frac{B_0}{4}} \Phi_M^{(1)} +\int \nabla S_1 \chi_{\frac{B_0}{4}} \nabla\Phi_M^{(2)}  \right\} \right| \lesssim C(M)\frac{b^2}{|\log b|}. 
\eee
Finally, using the estimate on $\boldsymbol \Phi_M$ \fref{orthophim}, we obtain :
\bee
\left( b + \lsl\right) \left< \Lambda \tilde {\bf Q}_b,\boldsymbol \Phi_M\right> &=& \left( b + \lsl\right) \left< \Lambda {\bf Q}, \boldsymbol \Phi_M\right > + O(C(M)bV) \\
&=& \left( b + \lsl\right)\left( -32 \pi \log M + O(1)\right) + O(C(M)bV),
\eee
and
\bee
\left| \left< \pa_s \tilde {\bf Q}_b,\boldsymbol \Phi_M\right>\right| = \left| b_s\left< \tilde{\bf T}_1,\boldsymbol \Phi_M\right> + 2bb_s\left< \tilde{\bf T}_2,\boldsymbol \Phi_M\right>\right| \lesssim C(M)bV.
\eee
Hence, injecting all above bounds in \fref{projectontophim}, we obtain the bound
\be
\label{estiboundone}
\left|\lsl+b\right|\lesssim b^{\frac 34}|V|+C(M)\frac{b^2}{|\log b|}.
\ee
{\bf Step 2 :} Projection onto $ \mathcal L^* \Phi_M$ \\
We project the equation \fref{equationE} onto $ \mathcal L^* \Phi_M$. Using the orthogonality conditions \fref{orthoe}, we obtain:
\be
\label{projectontoLstarphim}
-\lsl \left< \Lambda {\bf E}, \mathcal L^*\boldsymbol\Phi_M\right> = \left< {\bf E}_2, \mathcal L^*\boldsymbol\Phi_M\right> + \left< {\bf F}, \mathcal L^*\boldsymbol\Phi_M\right> + \left< \widetilde{\bf Mod}, \mathcal L^*\boldsymbol\Phi_M\right> + \left< {\bf G}, \mathcal L^*\boldsymbol\Phi_M\right>.
\ee
In the same way as the last step, we prove without difficulties that :
\bee
\left|  \left< - \lsl \Lambda {\bf E} - {\bf F} - {\bf G}, \mathcal L^*\boldsymbol\Phi_M\right>\right| \lesssim C(M) \left(bV + \frac{b^2}{|\log b|} \right).
\eee
The bound \fref{secondroughbound} and the bootstrap bound \fref{bootsmallh2q} imply:
\bee
\left| \left< {\bf E}_2, \mathcal L^*\boldsymbol\Phi_M\right> \right| \lesssim \frac{\| {\bf E}_2\|_{X_Q}}{M} \lesssim b^{\frac32}.
\eee
We recall that $\mathcal L \Lambda Q = 0$. Hence :
\bee
\left| \left( b + \lsl\right) \left< \Lambda \tilde {\bf Q}_b, \mathcal L^*\boldsymbol \Phi_M\right> \right| \lesssim Vb \left| \left< \mathcal L\Lambda \tilde {\bf T}_1 + b\mathcal L\Lambda \tilde {\bf T}_2 , \boldsymbol \Phi_M\right> \right| \lesssim C(M)bV.
\eee
To conclude, 
\bee
 \left< \pa_s \tilde{\bf Q}_b, \mathcal L^*\boldsymbol \Phi_M\right> &=& b_s \left< \mathcal L \tilde{\bf T}_1,\Phi_M \right> + O(C(M)bV) \\
 &=& b_s(-32\pi \log M + O(1)) + O(C(M)bV).
\eee
Injecting the collection of above estimates in \fref{projectontoLstarphim} yields :
\be
\label{estiboundtwo}
\left|b_s\right|\lesssim b^{\frac 32}+C(M)b|V|.
\ee
{\bf Step 3 :} Conclusion \\
Summing the bound \fref{estiboundone} and \fref{estiboundtwo} yields for b enough small 
\be
V \lesssim b^{\frac32}.
\ee
Injecting this bound in \fref{estiboundone} and \fref{estiboundtwo} concludes the proof of the bounds \fref{estlambda} and \fref{estb} and thus the proof of the Lemma \ref{lemmeparam}.
\end{proof}
\subsection{Second decomposition of the flow}
By the construction, we have formally predicted that the law of $b$ verify $b_s = - 2\frac{b^2}{|\log b|} \left( 1+O(1)\right)$. The bound \fref{estb} isn't enough good to know precisely the gap between $b_s$ and $-2\frac{b^2}{|\log b|}$. This is a consequance of the too slow decay of the elements of the kernel of $\mathcal L^*$. Thus, as in \cite{raphael2012b}, we introduce a second decomposition of the flow, which will lift the parameter $b$.
\begin{lemma}[Second decomposition]
\label{introductionbchapeau}
There exists a unique decomposition 
\be
\label{defehat}
\tilde {\bf Q}_b + {\bf E} = \tilde {\bf Q}_{\hat b} + \hat {\bf E} = \tilde {\bf Q}_{\hat b} + \begin{array} {| c} \hat \e \\ \hat \eta\end {array}
\ee
such that $\hat b \lesssim b \lesssim \hat b$ and $\hat \e$ satisfying the orthogonality condition
\be
\label{superortho}
\left < \hat {\bf E}, \mathcal L^* {\bf \Phi}_{0,\hat B_0} \right> = 0, \ \ \hat B_0 = \frac{1}{\sqrt{\hat b}}.
\ee
Moreover, there holds the bound:
\be
\label{deformation}
|b - \hat b| \lesssim \frac{b}{|\log b|}.
\ee
\end{lemma}
\begin{proof}
Let the map $F({\bf V}, \hat b) = \left< {\bf V} - \tilde{\bf Q}_{\hat b}, \mathcal L^* {\bf \Phi}_{0,\hat B_0} \right>$, for $|b-\hat b| \lesssim b$. First, remark that:
\be
F( \tilde{\bf Q}_{ b}, b)=0.
\ee
Moreover, 
\be
\left. \frac{\pa F}{\pa \hat b} \right |_{(\hat b=b, {\bf V} =  \tilde {\bf Q}_{\hat b})} = - \left< \frac{ \pa  \tilde {\bf Q}_{\hat b}}{\pa b}, \mathcal L^* {\bf \Phi}_{0,B_0} \right>.
\ee
Using the bounds of the Proposition \ref{localization}, we obtain:
\bee
\frac{\partial \qbt}{\partial b}&=&\chi_{B_1}T_1+O\left(\frac{1}{1+r^2}{\bf 1}_{B_1\leq r\leq 2B_1}+b\left[\frac{1+|\log( r\sqrt{b})|}{|\log b|}{\bf 1}_{r\leq 6B_0}+\frac{1}{b^2r^4|\log b|}{\bf 1}_{6B_0\leq r\leq 2B_1}\right]\right), \\
\frac{\partial \nabla \tilde P_b}{\partial b}&=& \chi_{B_1} \nabla S_1 + O\left(\frac{1}{1+r}{\bf 1}_{B_1\leq r\leq 2B_1} + b r(1+ \log r)^2{\bf 1}_{1\leq r \leq 2B_1} \right).
\eee
Hence, these bounds together \fref{bjebbeibei} yield
\bee&&\left(\frac{\partial \qbt}{\partial b}, \left(\mathcal L^*\right)^{(1)}\left(\Phi^{(1)}_{0,B_0},\Phi^{(2)}_{0,B_0}\right)\right) \\
&= & \left(\chi_{B_1}T_1,\left(\mathcal L^*\right)^{(1)}\left(\Phi^{(1)}_{0,B_0},\Phi^{(2)}_{0,B_0}\right)\right)+bO\left(\int \left[\frac{1 + |\log (r\sqrt b)|}{|\log b|}{\bf 1}_{r\leq 6B_0}+\frac{1}{b^2r^4|\log b|}{\bf 1}_{ 6B_0\leq r\leq 2B_1}\right]\right)
\eee
and
\bee
&&\left(\frac{\partial \nabla \tilde P_b}{\partial b}, \nabla \left(\mathcal L^*\right)^{(2)}\left(\Phi^{(1)}_{0,B_0},\Phi^{(2)}_{0,B_0}\right)\right) \\
&=& O \left( \int  \frac{|\log b|}{1+r^4} {\bf 1}_{r \leq 2B_1}\right) + b|\log b|O\left(\int_{r \leq 2B_1} \frac{r(1+ \log r)^2}{1+r^2} \right)= O(1).
\eee
Thus with \fref{orthophim} :
\bee
&&\left< \frac{ \pa  \tilde {\bf Q}_{\hat b}}{\pa b}, \mathcal L^* {\bf \Phi}_{0,B_0} \right> = \left(\frac{\partial \qbt}{\partial b}, \left(\mathcal L^*\right)^{(1)}\left(\Phi^{(1)}_{0,B_0},\Phi^{(2)}_{0,B_0}\right)\right) + \left(\frac{\partial \nabla \tilde P_b}{\partial b}, \nabla \left(\mathcal L^*\right)^{(2)}\left(\Phi^{(1)}_{0,B_0},\Phi^{(2)}_{0,B_0}\right)\right) \\
& = & -32\pi \log B_0+O(1) <0.
\eee
From the implicit function theorem, this concludes the proof of the existence and the uniqueness of the decomposition \fref{defehat} with $b\sim \hat{b}$.\\
Now, we let prove \fref{deformation}. We claim the following bound:
\be
\label{fgfgfg}
\left< {\bf \tilde Q}_b - {\bf Q} , {\bf \mathcal L}^* {\bf \Phi}_{0,B}\right> = -32 \pi b \log B + O(b), \ \ \mbox{for} \ \ \frac 1 {\sqrt b} \lesssim B \lesssim \frac 1 {\sqrt b}.
\ee
Let take the scalar product of \fref{defehat} with ${\bf \mathcal L}^* {\bf \Phi}_{0,\hat B_0}$.
\bee
\left< \tilde {\bf Q}_b - \tilde {\bf Q}_{\hat b}, {\bf \mathcal L}^* {\bf \Phi}_{0,\hat B_0}\right> = \left< \hat{\bf E} - {\bf E}, {\bf \mathcal L}^* {\bf \Phi}_{0,\hat B_0}\right> 
\eee
With the bound \fref{fgfgfg}, we have:
\bee
\left< \tilde {\bf Q}_b - \tilde {\bf Q}_{\hat b}, {\bf \mathcal L}^* {\bf \Phi}_{0,\hat B_0}\right> = 32\pi\log \hat B_0 (b-\hat b) + O(b+\hat b).
\eee
The orthogonality condition \fref{superortho} yields:
 $$|\log b||b-\hat{b}|\lesssim |b|+|\hat{b}|+
|\left<{\bf E}_2,{\bf\Phi}_{0,\hat{B}_0}\right>|\lesssim |b|+\frac{\|{\bf E}_2\|_{L^2_Q}}{\sqrt{b}}$$
The last inequality together with the bootstrap bound \fref{bootsmallh2q}, which we recall here:
\bee
\| {\bf E}_2\|_{X_Q} \lesssim \frac{b^3}{|\log b|^2}.
\eee conclude the proof of the lemma \ref{introductionbchapeau}. \\
{\it Proof of \fref{fgfgfg}:} \\
To begin :
\bee
\left< {\bf \tilde Q}_b - {\bf Q} , {\bf \mathcal L}^* {\bf \Phi}_{0,B}\right> &=& \left< {\bf \tilde \Upsilon}_b, {\bf \mathcal L}^* {\bf \Phi}_{0,B}\right> \\
&=&b\left< \mathcal L {\bf T}_1, {\bf \Phi}_{0,B}\right> + \left< b  (\tilde {\bf T}_1 - {\bf T}_1) + b^2  \tilde {\bf T}_2 , {\bf \mathcal L}^* {\bf \Phi}_{0,B}\right> 
\eee
By contruction, and using \fref{orthophim}, we have:
$$b\left< \mathcal L {\bf T}_1, {\bf \Phi}_{0,B}\right> = b\left< \Lambda {\bf Q}, {\bf \Phi}_{0,B}\right > = -32\pi b \log B + O(b).$$
Now,
Now, from \fref{bjebbeibei}:
\bee
\left|\left< b^2  \tilde {\bf T}_2 , {\bf \mathcal L}^* {\bf \Phi}_{0,B}\right>\right| &\lesssim& b^2\int _{ r\leq 2B_1}\left[r^4{\bf 1}_{r\leq 1}+\frac{1 + |\log (r\sqrt b)|}{|\log b|}{\bf 1}_{1\leq r\leq 6B_0}+\frac{1}{b^2r^4|\log b|}{\bf 1}_{ 6B_0\leq r\leq 2B_1}\right]\\
&+& b^2 |\log b|\int_{1 \leq r \leq 2B_1} \frac{r(1+ \log r)}{1+r^3}\\
&\lesssim & b 
\eee
This concludes the proof of \fref{fgfgfg}.
\end{proof}
Let 
\be
\label{defXi}
{\bf \Xi} =\begin{array} {| c} \zeta \\ \xi \end{array} =  \widehat {\bf E} - {\bf E} = \tilde{\bf Q}_b - \tilde{\bf Q}_{\hat b} = \begin{array} {| c} \tilde\alpha_b - \tilde\alpha_{\hat b} \\\tilde \gamma_b - \tilde\gamma_{\hat b} \end{array} 
\ee
the gap between the two decompositions which we split in two parts :
\be
{\bf \Xi} = {\bf \Xi}_{big} + {\bf \Xi}_{sm},
\ee
with
\bea
\label{xibig}{\bf \Xi}_{big} &=& \begin{array} {|c} \zeta_{big} \\ \xi_{big} \end{array}=(b-\hat b) \tilde{\bf T}_1, \\
\label{xism}{\bf \Xi}_{sm} &=& \begin{array} {|c} \zeta_{sm} \\ \xi_{sm} \end{array}= \hat b \  \begin{array} {|c}(\chi_{B_1} - \chi_{\hat B_1})T_1 \\\int_0^r (\chi_{B_1} - \chi_{\hat B_1}) \nabla S_1 \end{array}+\int_{\hat{b}}^b\left[2b  \tilde{\bf T}_2+b^2\pa_b \tilde{\bf T}_2\right]db
\eea
It's important to remark that ${\bf \Xi}_{big}$ is supported along $T_1$. Hence, we will use this particular structure to improve some degenerate norms, in order to close the bootstrap. Let:
\be
{\bf \Xi}_2 = \mathcal L {\bf \Xi}.
\ee
The following lemma describes adapted bounds on the gap betwenn the two decomposition :
\begin{lemma}[Control of the gap]
\label{lemmaradiation}
There holds the pointwise bounds:
\begin{itemize}
\item Estimates on $\boldsymbol \Xi_{big}:$
\bea
\label{estimatonebis} 
\nonumber  \int |\tau^i\pa_{\tau}^i \zeta_{big}|^2+\int \frac{|\nabla \phi_{\zeta_{big}}|^2}{1+\tau^2} +\int \frac{|\tau^i\pa_{\tau}^i\nabla \xi_{big}|^2}{1+\tau^2}&&\\
+ \left\| \mathcal L (\boldsymbol \Xi_{big})\right\|_{X_Q}^2 + \left\| \mathcal L ( \Lambda \boldsymbol \Xi_{big})\right\|_{X_Q}^2 &\lesssim& \frac{b^2}{|\log b|^2}.
\eea
\item Estimates on $\boldsymbol \Xi_{sm}:$
\bea
\label{estimatone} 
\int |\tau^i\nabla^i\zeta_{sm}|^2+\int \frac{|\nabla \phi_{\zeta_{sm}}|^2}{1+\tau^2} +\int \frac{|\tau^i\pa_{\tau}^i \xi_{sm}|^2}{1+\tau^2}&&\\
\nonumber + \left\| \mathcal L (\boldsymbol \Xi_{sm})\right\|_{X_Q}^2 + \left\| \mathcal L ( \Lambda \boldsymbol \Xi_{sm})\right\|_{X_Q}^2 &\lesssim& \frac{b^3}{|\log b|^2},\\
\label{estimatonebisbis} 
\int\left| \mathcal L^{(1)}\left(\zeta_{sm},\xi_{sm}\right)\right|^2+\int\frac{\left| \mathcal L^{(2)}\left(\zeta_{sm},\xi_{sm}\right)\right|^2}{1+\tau^2}&\lesssim& \frac{b^4}{|\log b|^2}.
\eea
\item Estimates on $\boldsymbol \Xi:$
\bea
\label{vivementlafin}
\int \left|\nabla \xi \right|^2 &\lesssim& b^2 |\log b|^2\\
\label{degenrateoneter}
\int Q|\nabla \mathcal M ^{(1)} \Lambda \boldsymbol \Xi |^2 + \int  |\Delta \mathcal M ^{(2)} \Lambda \boldsymbol \Xi|^2&\lesssim& \frac{b^2}{|\log b|^2}.
\eea
\item Estimates on $\boldsymbol \Xi_{2}:$
Let ${\bf v}=(v_1, v_2) \in L^2_Q \times \dot H^1$ such that $\int v_1=0$, then:
\bea
\label{degenrateoneone2}
\left|\int \mathcal M^{(1)}(\zeta_2,\xi_2) v_1\right|&\lesssim& \frac{b^{\frac 32}}{|\log b|}\| v_1\|_{L^2_Q},\\
\label{degenrateoneone}
|\left<\mathcal M {\bf \Xi}_2,{\bf v}\right>|&\lesssim& \frac{b^{\frac 32}}{|\log b|}\|{\bf v}\|_{X_Q},
\eea
Moreover,
\bea
\label{noenoeno2}
\left|\int \mathcal M^{(1)}(\zeta_2,\xi_2) \zeta_2\right|&\lesssim& \frac{b^3}{|\log b|^2},\\
\label{noenoeno}
\left|\left<\mathcal M {\bf \Xi}_2,{\bf \Xi}_2\right>\right|&\lesssim& \frac{b^3}{|\log b|^2},\\
\label{degenrateonebis}
\int Q|\nabla \mathcal M ^{(1)} \boldsymbol \Xi_{2}|^2 + \int  |\Delta \mathcal M ^{(2)} \boldsymbol \Xi_{2}|^2&\lesssim& \frac{b^4}{|\log b|^2}.
\eea
\end{itemize}
\end{lemma}
\begin{proof}
First, using the Lemma \ref{introductionbchapeau} and the construction, we obtain the following bounds:
 \be
 \label{aaaaaaa}
 |r^i \pa_r^i \zeta_{big}| \lesssim |\hat b - b| \frac{{\bf 1}_{r \leq 2B_1}}{1+r^2}, \ \  |r^i\pa_r^i \nabla\xi_{big}| \lesssim |\hat b - b| \frac{r}{1+r^2}{\bf 1}_{r \leq 2B_1}.
\ee
In the same way, using the Poisson field of the radial profile $T_1$, we have:
 \be
 \label{aaaaaaaa}
 |r^i\pa_r^i \nabla \phi_{\zeta_{big}}| \lesssim |\hat b - b| \frac{r (1+\log r)}{1+r^2}{\bf 1}_{r \leq 2B_1}.
\ee
The bounds \fref{aaaaaaa} and \fref{aaaaaaaa} yield the first line of \fref{estimatonebis}. Now, we use the formula:
\bee
\mathcal L^{(1)} \begin{array} {|c} \e\\ \eta\end{array} = \nabla.(Q \nabla \mathcal M^{(1)}(\e,\eta)) = \Delta \e + Q \e + Q \Delta \eta + \nabla Q \nabla \eta + \nabla \e \nabla \phi_Q
\eee																																		to estimate
\bee 
 \mathcal L^{(1)}\left(\tilde T_1,\tilde S_1\right) &=& \Delta \tilde T_1 + Q \tilde T_1 + Q \Delta  \tilde S_1 + \nabla Q \nabla  \tilde S_1 + \nabla \tilde  T_1 \nabla \phi_Q \\
 &=& \chi_{B_1} \Lambda Q + O\left(\frac{{\bf 1}_{B_1 \leq r \leq 2B_1}}{1+r^4}\right) = \Lambda Q + O\left(\frac{{\bf 1}_{r \geq B_1}}{1+r^4}\right)
\eee	
Moreover,
\bee
 \mathcal L^{(2)}\left(\tilde T_1,\tilde S_1\right) = \Delta \tilde S_1 - \tilde T_1 = \chi_{B_1} \Lambda \phi_Q +  O\left(\frac{{\bf 1}_{B_1 \leq r \leq 2B_1}}{1+r^2}\right) = \Lambda \phi_Q + O\left(\frac{{\bf 1}_{r \geq B_1}}{1+r^2}\right)
\eee												
Thus,
\be
\label{pppp}
 \mathcal L {\bf \Xi}_{big} = (b-\hat b)\Lambda {\bf Q} + \boldsymbol{\mathcal R},
\ee
where
\be
\label{mathcalR}
\boldsymbol{\mathcal R} = (b-\hat b) \begin{array} {|c} O\left(\frac{{\bf 1}_{r \geq B_1}}{1+r^4}\right) \\
O\left(\frac{{\bf 1}_{r \geq B_1}}{1+r^2}\right) 
\end{array}
\ee
So,
\bee
 \int \frac{\left| \mathcal L^{(1)}\left(\zeta_{big},\xi_{big}\right)\right|^2}{Q} + \int \left| \nabla\mathcal L^{(2)}\left(\zeta_{big},\xi_{big}\right) \right|^2 \lesssim \frac{b^2}{|\log b|^2}.
\eee
To conclude the proof of the bound \fref{estimatonebis}, we use the following equality, for function ${\bf f} = (f,g)$ well localized:
\be
\label{mmm}
\mathcal L \Lambda {\bf f} = 2 \mathcal L {\bf f} + \Lambda\left( \mathcal L {\bf f} \right) - \begin{array} {|c} \Lambda^{(1)} {\bf Q} (f+\Delta g) + \nabla g \nabla (\Lambda^{(1)} {\bf Q}) + \nabla f \nabla \phi_{\Lambda {\bf Q}}\\0 \end{array}.
\ee
We compute $\mathcal L {\bf f}_\l$ and we differentiate this relation at $\l =1$ to obtain \fref{mmm}, which use together with \fref{pppp} and \fref{mathcalR} yield the following bound:
\be
\label{mathcalR1}
 \mathcal L \Lambda {\bf \Xi}_{big} =(b-\hat b) \begin{array} {|c} O\left(\frac{{\bf 1}_{r \geq B_1}}{1+r^4}\right) \\
O\left(\frac{{\bf 1}_{r \geq B_1}}{1+r^2}\right) 
\end{array}
\ee
The bound \fref{deformation} of the Lemma \ref{introductionbchapeau} concludes the proof of \fref{estimatonebis}.\par
To prove \fref{estimatone} and \fref{estimatonebisbis}, we use the same strategy, using the following bounds, coming from the Proposition \ref{localization}:
\bea
\label{estammleon}
|r^i\nabla^i\zeta_{sm}| & \lesssim & \hat{b}\frac{{\bf 1}_{\frac{B_1}{2}\leq r\leq 3B_1}}{r^2}\\
\nonumber & + & b|b-\hb|\left[r^2{\bf 1}_{r\le1} +\frac{1 + |\log (r\sqrt b)|}{|\log b|}{\bf 1}_{1\leq r\leq 6B_0}+ \frac{1}{b^2r^4|\log b|}{\bf 1}_{ 6B_0\leq r\leq 2B_1}\right],
\eea
\be
\label{estammleonbis}
|\nabla \phi_{\zeta_{sm}}|\lesssim \frac{\hat b}{r}{\bf 1}_{r\geq \frac{B_1}{2}}+b|b-\hb|\left[ r^5{\bf 1}_{r\le1} +\frac{1 + r|\log (r\sqrt b)|}{|\log b|}{\bf 1}_{1\leq r\leq 6B_0}+\frac{1}{rb|\log b|}{\bf 1}_{r\geq 2B_0}\right],
\ee
and
\bea
\label{estammleonter}
 |r^i\pa_r^i \nabla\xi_{sm}| &\lesssim&  \hat{b}\frac{r}{1+r^2}{\bf 1}_{\frac{B_1}{2}\leq r\leq 3B_1} \\
 \nonumber & + & b|b-\hb|\left[r^5{\bf 1}_{r\le1} +r\frac{1 + |\log (r\sqrt b)|}{|\log b|}{\bf 1}_{1\leq r\leq 6B_0}+ \frac{1}{rb|\log b|}{\bf 1}_{ 6B_0\leq r\leq 2B_1}\right].
\eea
The whole proof is available in \cite{raphael2012b}. 
\par
We prove now \fref{degenrateoneter}. The second part is a simple consequence of the above estimates. The first part is more technical, and we must use the structure of $\zeta$. Indeed, 
\bee
\zeta = (b-\hb)\tilde T_1 + \hat b (\chi_{B_1} - \chi_{\hat B_1}){T}_1 +\int_{\hat{b}}^b\left[2b  \tilde{T}_2+b^2\pa_b \tilde{T}_2\right]db
\eee
The crucial point here is the degeneracy of $\Lambda T_1$ :
\bee
\Lambda T_1 = O \left( \frac{\log r}{1+r^4} \right).
\eee
Hence, using the bound \fref{estammleon}:
\bee
 &&\int (1+\tau^2) |\Lambda \zeta|^2 \\ 
 &\lesssim& |b-\hat b|^2 \int _{\tau \leq 2B_1} \frac{|\log r|^2}{(1+\tau^6)} +  |b-\hat b|^2 \int_{B_1 \leq \tau \leq 2B_1} \frac{1}{1+\tau^2} + |\hat b|^2 \int \frac{|(\chi_{B_1} - \chi_{\hat B_1})|^2}{1+\tau^2}\\
&+&|\hat b|^2 |b-\hat b|^2 \left\{ \int_{\tau \leq B_0} \tau^2 + \frac{1}{b^4|\log b|^2}\int_{B_0 \leq \tau \leq 2B_1}\frac{1}{1+ \tau^4} \right\}\\
&\lesssim& \frac{b^2}{|\log b|^2}+ |\hat b|^2 \left| \log (\hat B_1) - \log (B_1)\right| \lesssim \frac{b^2}{|\log b|^2}.
\eee
Using this bound together the following estimation
\bee
\left| \int Q|\nabla \mathcal M ^{(1)} \Lambda \boldsymbol \Xi |^2 \right| \lesssim \int (1+\tau^2) |\Lambda \zeta|^2 + \int \frac {|\nabla \xi|^2}{1+\tau^4},
\eee
and \fref{estimatonebis} and \fref{estimatone} concludes the proof of  \fref{degenrateoneter}.
\par Now, let $ {\bf v}=(v, v_1) \in L^2_Q \times L^2\ \ \mbox{with}\ \  \int v=0$. 
\bee
|\left<\mathcal M {\bf \Xi}_2,{\bf v}\right>| &\lesssim& |\left<\mathcal M \mathcal L {\bf \Xi}_{big},{\bf v}\right>| + |\left<\mathcal M \mathcal L {\bf \Xi}_{sm},{\bf v}\right>| \\
&\lesssim& |(b-\hb)\left<\mathcal M \Lambda {\bf Q},{\bf v}\right>| +|\left<\mathcal M \mathcal R,{\bf v}\right>| +  |\left<\mathcal M \mathcal L {\bf \Xi}_{sm},{\bf v}\right>| 
\eee
Using \fref{relationsm} and $\int v = 0$, we have :
\bee
\left<\mathcal M \Lambda {\bf Q},{\bf v}\right> = \int2v = 0.
\eee
Now, from the estimate \fref{mathcalR} :
\bee
\left| \nabla \phi_{\mathcal R^{(1)}} \right|  = \left|\frac1r \int_0^r \mathcal R^{(1)}(\tau) \tau d\tau \right| \lesssim \frac{|b-\hat b|}{r} \int_0^r \frac{{\bf 1}_{\tau \geq B_1} d\tau}{1+\tau^3} \lesssim \frac{b^2}{r|\log b|^3}.
\eee
We estimate using the definition of the operator $\mathcal M$ \fref{defoperatorM} :
\bee
&&|\left<\mathcal M \mathcal R,{\bf v}\right>| \lesssim \left| \int v_1\left( \frac{\mathcal R^{(1)}}{Q} + \mathcal R^{(2)} \right) \right| + \int \nabla v_2 \nabla \left( \mathcal R^{(2)} - \phi_{\mathcal R^{(1)}} \right) \\
&\lesssim& \left[ \int \frac{v_1^2}{Q}\int \frac{|b-\hb|^2{\bf 1}_{\tau \geq B_1}}{Q(\tau) (1+\tau^8)} \right]^{\frac 12} + \left[ \int|\nabla v_2|^2\left\{ \int |\nabla \phi_{\mathcal R^{(1)}}|^2 + \int | \nabla \mathcal R^{(2)} |^2 \right\}\right|^{\frac12}\\
&\lesssim& \frac{b^{\frac32}}{|\log b|} \|{\bf v}\|_{X_Q}
\eee
To conclude, we estimate the term depending on $\boldsymbol \Xi_{sm}$.
\bee
 |\left<\mathcal M \mathcal L {\bf \Xi}_{sm},{\bf v}\right>| &=& \left| \int \mathcal M^{(1)} \mathcal L \boldsymbol \Xi_{sm}v_1 +\int \nabla \mathcal M^{(2)} \mathcal L \boldsymbol \Xi_{sm} \nabla v_2  \right| \\
 &\lesssim& \left[ \int Q \left| \mathcal M^{(1)} \mathcal L \boldsymbol \Xi_{sm}\right|^2 \int \frac{v_1^2}{Q} \right]^{\frac12} +  \left[ \int |\nabla v_2|^2 \left|  \nabla \mathcal M^{(2)} \mathcal L \boldsymbol \Xi_{sm}\right|^2\right]^{\frac12}\\
&\lesssim& \|\mathcal L \boldsymbol \Xi_{sm}\|_{X_Q}\|\boldsymbol v\|_{X_Q}.
\eee
The last inequality comes from the continuity of the operator $\mathcal M$ \fref{mcontiniuos}. The bound \fref{estimatone} and the collection of above estimates yield \fref{degenrateoneone}. Using the same strategy, we prove \fref{degenrateoneone2}. The proof is left to the reader.\par
To prove \fref{noenoeno}, we use the decomposition \fref{pppp} and the knowledge of the kernel of $\mathcal L$. 
\bee
\left< \mathcal M \boldsymbol \Xi_2,\boldsymbol \Xi_2\right> = \left< \mathcal M \boldsymbol \Xi_2,(b-\hb)\Lambda {\bf Q} + \boldsymbol{\mathcal R} + \mathcal L \boldsymbol \Xi_{sm}\right>.
\eee
From the effect of the operator $\mathcal M$ on the direction $\Lambda Q$, and that $\int \zeta_2 = 0$,
\bee
\left< \mathcal M \boldsymbol \Xi_2,(b-\hb)\Lambda {\bf Q} \right> = 0.
\eee
Now from \fref{degenrateoneone}, \fref{estimatone} and \fref{mathcalR} 
\bee
\left|\left< \mathcal M \boldsymbol \Xi_2, \boldsymbol{\mathcal R} + \mathcal L \boldsymbol \Xi_{sm}\right> \right| &\lesssim& \frac{b^{\frac32}}{|\log b|}\left[ \|\mathcal R\|_{X_Q}+\|\mathcal L\boldsymbol \Xi_{sm}\|_{X_Q}\right] \\
& \lesssim &  \frac{b^{\frac 32}}{|\log b|}\left[\left(\frac{b^2}{|\log b|^2}\int \frac{{\bf 1}_{\tau\geq B_1}}{1+\tau^4}\right)^{\frac 12}+\frac{b^{\frac 32}}{|\log b|}\right]\lesssim \frac{b^3}{|\log b|^2}.
\eee
This concludes the proof of \fref{noenoeno}. We prove now the last inequality \fref{degenrateonebis}. In this purpose, we can remark that :
\bee
\left| \begin{array}{l} \nabla \mathcal M^{(1)} \\ \nabla \mathcal M^{(2)} \end{array} \right. \left( {\boldsymbol \Xi}_2 \right)&=& \left| \begin{array}{l} \nabla \mathcal M^{(1)} \\ \nabla \mathcal M^{(2)} \end{array} \right. \left( (b-\hb)\Lambda {\bf Q} + \boldsymbol{\mathcal R} + \mathcal L {\boldsymbol \Xi}_{sm} \right) \\
&=&\left| \begin{array}{l} \nabla \mathcal M^{(1)} \\ \nabla \mathcal M^{(2)} \end{array} \right. \left(\boldsymbol{\mathcal R} + \mathcal L {\boldsymbol \Xi}_{sm} \right) \\
\eee
Using the bounds \fref{mathcalR}, \fref{estammleon}, \fref{estammleonbis} and \fref{estammleonter}, we obtain:
\bee
\int Q \left| \nabla \mathcal M^{(1)} \mathcal R \right|^2 &\lesssim& \frac{b^4}{|\log b|^6}\\
\int \left| \mathcal L^{(2)} \mathcal R \right|^2 &\lesssim& \frac{b^5}{|\log b|^8}\\
\int Q \left| \nabla \mathcal M^{(1)} \mathcal L \boldsymbol {\Xi}_{sm} \right|^2 + \int \left| \mathcal L^{(2)}  \mathcal L \boldsymbol {\Xi}_{sm}\right|^2  &\lesssim& \frac{b^4}{|\log b|^2}
\eee
The last inequality comes from:
\bee
\int Q \left| \nabla \mathcal M^{(1)} \mathcal L \tilde{\bf T}_2 \right|^2 + \int \left| \mathcal L^{(2)}  \mathcal L \tilde{\bf T}_2\right|^2  &\lesssim& 1
\eee
and, $\forall i \geq 0$
\bee
\left|r^i\pa_r^i \mathcal L \boldsymbol {\Xi}_{sm} \right| \lesssim b|b-\hb| |\tilde{\bf T}_2| \lesssim \frac{b^2}{|\log b|} |\tilde{\bf T}_2|
\eee
\fref{degenrateonebis} is proved and the Lemma \ref{lemmaradiation} too.
\end{proof}
\subsection{Sharp modulation equation}
We are in position to compute the sharp modulation equation. The lifted parameter $\hb$ plays here a crucial rule.

\begin{lemma}[Sharp modulation equations for $b$]
\label{lemmasharpmod}
 There exist $C(M)$ an universal enough large constant, independant of $K^*(M)$, such that:
\be
\label{poitnzeiboud}
\left|\hat{b}_s+\frac{2b^2}{|\log b|}\right|\lesssim C(M)\frac{b^2}{|\log b|^2}.
\ee
\end{lemma}
\begin{proof}
{\bf Step 1} Projection of the equation \fref{equationE} satisfied by {\bf E} onto $\mathcal L^* \boldsymbol \Phi_{0,\hat B_0}$.\\
We take the scaler product of \fref{equationE} with $\mathcal L^* \boldsymbol \Phi_{0,\hat B_0}$ and we reorganize the terms :
\bea
\nonumber \pa_s \left\{ \left< \tilde {\bf Q}_b - {\bf Q} + {\bf E}, \mathcal L^* \boldsymbol \Phi_{0,\hat B_0}\right> \right\} &=& \left< \tilde {\bf Q}_b - {\bf Q} + {\bf E}, \pa_s\mathcal L^* \boldsymbol \Phi_{0,\hat B_0}\right> +  \left< {\bf E}_2, \mathcal L^* \boldsymbol \Phi_{0,\hat B_0}\right> \\
\label{vivelafoire}&+&  \left< -\lsl \Lambda {\bf E}  + {\bf F} + {\bf G}, \mathcal L^* \boldsymbol \Phi_{0,\hat B_0}\right> \\
\nonumber &+& \left( b + \lsl \right)  \left< \Lambda \tilde {\bf Q}_b , \mathcal L^* \boldsymbol \Phi_{0,\hat B_0}\right>,
\eea
where we recall that
\bee
{\bf F} = \tilde{\bf \Psi}_b +  {\bf \Theta}_b(\e,\eta) + {\bf N}(\e,\eta)
\eee
with $\tilde{\bf \Psi}_b$  is defined by \fref{eqerreurpsibtilde} and
\bee
 {\bf \Theta}_b(\e,\eta) &=&  \begin{array}{| c} \nabla. \left( \e \nabla \tilde{\gamma}_b + \tilde{\alpha}_b \nabla \eta \right) \\0 \end{array}, \\
 {\bf N}(\e,\eta) &=&  \begin{array}{| c} \nabla. \left( \e  \nabla \eta \right) \\0 \end{array} \\
\eee
and
\bee
{\bf G} &=&-c_bb^2 \breve{\bf T}.
\eee
{\bf Step 2} Crucial rule of the second decomposition. \\
We use here the second decomposition to execute the error term. Indeed,
\bee
\left< \tilde {\bf Q}_b - {\bf Q} + {\bf E}, \mathcal L^* \boldsymbol \Phi_{0,\hat B_0}\right> = \left< \tilde {\bf Q}_{\hb} - {\bf Q} + \hat {\bf E}, \mathcal L^* \boldsymbol \Phi_{0,\hat B_0}\right> =  \left< \tilde {\bf Q}_{\hb} - {\bf Q}, \mathcal L^* \boldsymbol \Phi_{0,\hat B_0}\right>,
\eee
the last cancellation coming from the orthogonality condition \fref{superortho}. Hence, we can split the left term  of \fref{vivelafoire} in two parts:
\be
\label{nvk}
\pa_s \left\{ \left< \tilde {\bf Q}_b - {\bf Q} + {\bf E}, \mathcal L^* \boldsymbol \Phi_{0,\hat B_0}\right> \right\} = \left<\pa_s  \tilde {\bf Q}_{\hb} , \mathcal L^* \boldsymbol \Phi_{0,\hat B_0}\right> + \left< \tilde {\bf Q}_{\hb} - {\bf Q}, \pa_s \mathcal L^* \boldsymbol \Phi_{0,\hat B_0}\right>
\ee
We estimate both terms separately. First,
\bea
\label{cnrs}&&\left<\pa_s  \tilde {\bf Q}_{\hb} , \mathcal L^* \boldsymbol \Phi_{0,\hat B_0}\right>=\hat b_s \left<\tilde {\bf T}_1 + \hb \pa_b\tilde {\bf T}_1 + 2\hb \tilde {\bf T}_2 + \hb^2 \pa_b \tilde {\bf T}_2,\mathcal L^* \boldsymbol \Phi_{0,\hat B_0}\right> \\
\nonumber &=&\hat b_s \left<{\bf T}_1, \mathcal L^* \boldsymbol \Phi_{0,\hat B_0}\right> + \hb_s\left<\left(\tilde {\bf T}_1- {\bf T}_1 \right) + \hb \pa_b\tilde {\bf T}_1 + 2\hb \tilde {\bf T}_2 + \hb^2 \pa_b \tilde {\bf T}_2,\mathcal L^* \boldsymbol \Phi_{0,\hat B_0}\right>
\eea
Using the bound \fref{bjebbeibei} and the bounds of the Proposition \ref{localization}, we obtain :
\bee
\label{cnrs2}&&\left| \left<\left(\tilde {\bf T}_1- {\bf T}_1 \right) + \hb \pa_b\tilde {\bf T}_1 + 2\hb \tilde {\bf T}_2 + \hb^2 \pa_b \tilde {\bf T}_2,\mathcal L^* \boldsymbol \Phi_{0,\hat B_0}\right>\right| \\
\nonumber&\lesssim& \int _{\tau \leq \hat B_0} \left|\tilde {T}_1- { T}_1 \right| + \hb | \pa_b\tilde {T}_1| + 2\hb |\tilde {T}_2| + \hb^2 |\pa_b \tilde {T}_2| \\
\nonumber&+&  \int _{\hat B_0 \leq \tau \leq 2\hat B_0}\frac{ \left|\nabla\tilde {S}_1- \nabla{ S}_1 \right| + \hb | \pa_b\nabla\tilde {S}_1| + 2\hb |\nabla\tilde {S}_2| + \hb^2 |\pa_b \nabla\tilde {S}_2|}{1+\tau^3} \\
\nonumber&\lesssim& \int_{B_1\leq \tau \leq 2B_1}\frac{1}{\tau^2} + b\int \left[\frac{1 + |\log (\tau\sqrt b)|}{|\log b|}{\bf 1}_{1\leq \tau\leq 6B_0}+\frac{1}{b^2\tau^4|\log b|}{\bf 1}_{ 6B_0\leq \tau\leq 2B_1}\right]\\
\nonumber&+& \int_{B_1\leq \tau \leq 2B_1} \frac{1}{1+\tau^4} + b \int_{1 \tau 2B_1} \frac{\tau \log \tau}{ 1+\tau^3} \\
\nonumber&\lesssim& 1
\eee
Now, the estimate \fref{orthophim} yields
\be
\label{cnrs3}
\left<{\bf T}_1, \mathcal L^* \boldsymbol \Phi_{0,\hat B_0}\right> = \left<\Lambda {\bf Q}, \boldsymbol \Phi_{0,\hat B_0}\right> = \left[ - 32 \pi \log \hat B_0 + O(1) \right]
\ee
Hence, this estimate together \fref{cnrs} and \fref{cnrs2} yields
\be
\label{premierterme}
\left<\pa_s  \tilde {\bf Q}_{\hb} , \mathcal L^* \boldsymbol \Phi_{0,\hat B_0}\right> = \hb_s\left[ - 32 \pi \log \hat B_0 + O(1) \right]
\ee
For the second term of \fref{nvk}, we use the definition of $\Phi_{0,\hat B_0}$ \fref{defphim}, and the bounds of the Proposition \ref{localization}:
\bea
&&\nonumber \left| \left< \tilde {\bf Q}_{\hb} - {\bf Q}, \pa_s \mathcal L^* \boldsymbol \Phi_{0,\hat B_0}\right> \right|=\left| \left<\mathcal L\left( \tilde {\bf Q}_{\hb} - {\bf Q}\right), \pa_s \boldsymbol \Phi_{0,\hat B_0}\right>\right| \\
 \nonumber &\lesssim& |\hat b_s| \left[ \int_{\hat B_0 \leq \tau\leq 2\hat B_0} \tau^2  \mathcal L^{(1)}\left( \tilde {\bf Q}_{\hb} - {\bf Q}\right) + \int_{\hat B_0 \leq \tau \leq 2\hat B_0} \frac{\log \tau}{\tau}  \nabla \mathcal L^{(2)}\left( \tilde {\bf Q}_{\hb} - {\bf Q}\right) \right]\\
\label {ftgre}&\lesssim& |\hat b_s| \left[ \int_{\frac{B_0}4 \leq \tau\leq  4 B_0} \frac{\tau^2}{\tau^4} + \int_{\frac{B_0}4 \leq \tau\leq  4 B_0} \frac{\log \tau}{\tau^4} \right]\lesssim |\hat b_s|
\eea
Hence, using the above bound with \fref{premierterme}, we obtain :
\be
\pa_s \left\{ \left< \tilde {\bf Q}_b - {\bf Q} + {\bf E}, \mathcal L^* \boldsymbol \Phi_{0,\hat B_0}\right> \right\} = \hb_s\left[ - 32 \pi \log \hat B_0 + O(1) \right]
\ee
{\bf Step 3} Estimations of the RHS of \fref{vivelafoire} \\
The leading term of the RHS of \fref{vivelafoire} is ${\bf G}$ defined by
 \bee
{\bf G} &=& -c_bb^2\breve{\bf T}.
\eee
By definition of $\breve{\bf T}$ \fref{breveT1}, we have
\bee
\left<{\bf G}, \mathcal L^* \boldsymbol \Phi_{0,\hat B_0}\right>=  -c_bb^2\left[\left<\mathcal L{\bf T}_1,\Phi_{0,\hat{B}_0}\right>+\left<(\chi_{\frac{B_0}{4}}-1){\bf T}_1,\mathcal L^*\Phi_{0,\hat{B}_0}\right>\right]
\eee
Now, from \fref{bjebbeibei}, we have
\bee
\left| \left<(\chi_{\frac{B_0}{4}}-1){\bf T}_1,\mathcal L^*\Phi_{0,\hat{B}_0}\right> \right| \lesssim \left| \int_{r\leq 2\hat{B}_0}\frac{{\bf 1}_{r\geq \frac{B_0}{4}}}{1+r^2} + \log \hat B_0 \int _{\hat B_0 \leq \tau \leq 2\hat B_0} \frac{1}{1+\tau^4}\right| \lesssim 1
\eee
The estimates \fref{defphim} and the definition of $c_b$ \fref{defcb} yields: 
\bee
\left<{\bf G}, \mathcal L^* \boldsymbol \Phi_{0,\hat B_0}\right>&=&  \frac{2b^2}{|\log b|}\left[1+O\left(\frac{1}{|\log b|}\right)\right]\left[32\pi\log B_0+O(1)\right]\\
& = & 32\pi b^2\left[1+O\left(\frac{1}{|\log b|}\right)\right].
\eee
We estimate like for the proof of \fref{ftgre} :
\bee
\left| \left< \tilde {\bf Q}_{b} - {\bf Q}, \pa_s \mathcal L^* \boldsymbol \Phi_{0,\hat B_0}\right> \right|\lesssim |\hat b_s|,
\eee
and the linear term in ${\bf E}$ with the bootstrap bound \fref{bootsmallh2q} :
\bee
\left| \left< {\bf E}, \pa_s \mathcal L^* \boldsymbol \Phi_{0,\hat B_0}\right> \right|&=& \left| \left< {\bf E_2}, \pa_s \boldsymbol \Phi_{0,\hat B_0}\right> \right| \\
&\lesssim& \left| \frac{(\hat B_0)_s}{\hat B_0}\right |  \left| \int \e_2 \left( r\chi' \frac{\tau}{\hat B_0}\right) \tau^2 + \int \nabla \eta_2 \left( r\chi' \frac{\tau}{\hat B_0}\right) \frac{\log (1+\tau^2)}{1+\tau} \right| \\
&\lesssim& \frac{\hb_s}{b} \| {\bf E}_2\|_{X_Q} \left( \int_{\hat B_0 \leq \tau \leq 2\hat B_0} \frac{\tau^4}{\tau^4} + \int_{\hat B_0 \leq \tau \leq 2\hat B_0} \frac{|\log \tau|^2}{\tau^2} \right)^{\frac{1}{2}}\\
&\lesssim& \frac{\hb_s}{b\sqrt b}  \| {\bf E}_2\|_{X_Q} \lesssim |\hat b_s|.
\eee
Now, we focus on the main liner term, using the bound \fref{estfonamentalebus} :
\bee
\left| \left< {\bf E}_2, \mathcal L^* \boldsymbol \Phi_{0,\hat B_0}\right>\right| \lesssim \int_{\tau \geq \hat B_0} |\e_2| + \log \hat B_0 \int_{\hat B_0 \leq \tau \leq 2\hat B_0} \frac{|\nabla \eta_2|}{1+\tau^3} \lesssim \sqrt b \| {\bf E}_2\|_{X_Q} \lesssim \frac{b^2}{|\log b|}.
\eee
Next, with the bound \fref{bjebbeibei}, the bootstrap bound \fref{bootsmallh2q}, the interpolation bound \fref{coercbase} and the bound \fref{cnkonenoeoeoi3poi} coming from the rough modulation equation :
\bee
\left| \left< \lsl \Lambda {\bf E}, \mathcal L^* \boldsymbol \Phi_{0,\hat B_0}\right> \right| &\lesssim& b\left\{ \int_{\tau \leq 2 \hat B_0} (|\e| + |\tau.\nabla \e|) + \int_{\hat B_0 \leq \tau \leq 2 \hat B_0} \frac{|\nabla \eta| + |\tau. \nabla^2 \eta|}{1+\tau^3}\right\} \\
&\lesssim& \sqrt b (\|\e\|_{L^2} + \|y.\nabla \e\|_{L^2} + \left\|\frac{\nabla \eta }{1+y^3}\right\|_{L^2} + \left\|\frac{\nabla^2 \eta}{1+y^2}\right\|_{L^2})\\
 &\lesssim& C(M)\sqrt b \|{\bf E}_2\|_{X_Q} \lesssim \frac{b^2}{|\log b|}
\eee
We treat the ${\bf F}$ terms separately. To begin, from the Proposition \ref{localization}, we know already the degenerate flux :
\bee
\left| \left< \tilde{\boldsymbol \Psi}_b, \mathcal L^* \boldsymbol \Phi_{0,\hat B_0}\right>\right| \lesssim \frac{b^2}{|\log b|}.
\eee 
Let's focus on the small linear term $\boldsymbol \T_b(\e,\eta)$. We recall that
\bee
 {\bf \Theta}_b(\e,\eta) &=&  \begin{array}{| c} \nabla. \left( \e \nabla \tilde{\gamma}_b + \tilde{\alpha}_b \nabla \eta \right) \\0 \end{array}
\eee
and from the Proposition \ref{localization}, we have the rough bounds, $\forall i \geq 0$
\bee
|r^i\pa_r^i \tilde \alpha_b|&\lesssim& \frac{b}{1+r^2} {\bf 1}_{r \leq 2B_1}\\
|r^i\pa_r^i \nabla \tilde \gamma_b| &\lesssim&  \frac{br(1+|\log r|)}{1+r^2} {\bf 1}_{r \leq 2B_1} 
\eee
Hence, we compute the following rough bound :
\bee
\left| \T^{(1)}_b(\e,\eta) \right| \lesssim b \left[ \frac{1+ |\log r|}{1+r} |\nabla \e| + \frac{1+\log r}{1+r^2} |\e| + \frac{|\Delta \eta|}{1+r^2}+ \frac{|\nabla \eta|}{1+r^3} \right]
\eee
The estimate \fref{estfonamentalebus} and the cancellation $\int \T_b^{(1)} = 0$ yields:
\bee
\left| \left< \boldsymbol \T_b(\e,\eta), \mathcal L^* \boldsymbol \Phi_{0,\hat B_0}\right> \right| &\lesssim& b\int _{r \geq \hat B_0} \left[ \frac{1+ |\log r|}{1+r} |\nabla \e| + \frac{1+\log r}{1+r^2} |\e| + \frac{|\Delta \eta|}{1+r^2}+ \frac{|\nabla \eta|}{1+r^3} \right] \\
&\lesssim& C(M) \sqrt b \|{\bf E}_2\|_{X_Q} \lesssim \frac{b^2}{|\log b|}.
\eee
To conclude, we must treat the nonlinear term defined by
\bee
 {\bf N}(\e,\eta) &=&  \begin{array}{| c} \nabla. \left( \e  \nabla \eta \right) \\0 \end{array}
 \eee
As $\int N^{(1)}(\e,\eta) = 0$, we can use the same strategy like the small linear term :
\bee
\left| \left< \boldsymbol N(\e,\eta), \mathcal L^* \boldsymbol \Phi_{0,\hat B_0}\right> \right| &\lesssim& b\int _{r \geq \hat B_0} |\nabla.(\e\nabla\eta)|\\
&\lesssim& \left(\int |\e|^2 \int |\Delta \eta| ^2 \right)^{\frac{1}{2}} +\left(\int |\nabla \e|^2 \int |\nabla \eta| ^2 \right)^{\frac{1}{2}} 
 \eee
Now, we have
\bee
\int |\Delta \eta| ^2 = \left|\int \nabla \eta \nabla \Delta \eta \right|\lesssim \left(\int |\nabla \Delta \eta|^2 \int |\nabla \eta| ^2 \right)^{\frac{1}{2}} 
\eee
The bootstrap bounds \fref{bootsmallh1}, \fref{bootsmallh2q} and the interpolation bound \fref{coercbase} yield
\bee
\left| \left< \boldsymbol N(\e,\eta), \mathcal L^* \boldsymbol \Phi_{0,\hat B_0}\right> \right| &\lesssim& \frac{b^{\frac32}}{|\log b|} b |\log b|^C \lesssim \frac{b^2}{|\log b|}
\eee
The collection of above estimates yield
\bee
\left| \left< \boldsymbol F, \mathcal L^* \boldsymbol \Phi_{0,\hat B_0}\right> \right| \lesssim \frac{b^2}{|\log b|}
\eee
For the last term induced by the modulation, we use the same strategy because of $\int \Lambda \tilde Q_b^{(1)} = \int \Lambda \tilde \alpha_b = 0$. Thus, from the bound \fref{cnkonenoeoeoi3poi} coming from the rough modulation equations,
\bee
\left| b + \lsl \right|  \left< \Lambda \tilde {\bf Q}_b , \mathcal L^* \boldsymbol \Phi_{0,\hat B_0}\right> &=&\left| b + \lsl \right|  \left< \Lambda \tilde \alpha_b , \mathcal L^* \boldsymbol \Phi_{0,\hat B_0}\right>\\
&\lesssim& b^{\frac32}\int_{\frac{B_0}{4} \leq \tau \leq 6B_0} \frac{b}{1+r^2} \lesssim \frac{b^2}{|\log b|}.
\eee
Injecting all above estimates into \fref{vivelafoire} yields :
\bee
\hb_s[-32\pi \log \hat B_0 + O(1)] = 32\pi b^2 + O\left(\frac{b^2}{|\log b|} + |\hb_s|\right).
\eee
Now, by assumption
 \be
\label{estnieneoifn}
\left|\frac{\log \hat B_0}{|\log b|}-\frac 12\right|=\left|\frac{\log (\hat B_0 \sqrt{b})}{|\log b|}\right|\lesssim \frac{1}{|\log b|},
\ee
Hence,
\bee
\left|\hat{b}_s+\frac{2b^2}{|\log b|}\right|\lesssim C(M)\frac{b^2}{|\log b|^2}.
\eee
This is precisely \fref{poitnzeiboud}, and the Lemma \ref{lemmasharpmod} is proved.
\end{proof} 
\section{$X_Q$ monotonicity}
\label{lyapounov}
The goal of this section is to prove the following monotonicity formula at the $X_Q$ level, which is the keystone of the proof of \fref{bootsmallh2qboot}.
\begin{proposition}[$X_Q$ monotonicity]
\label{htwoqmonton}
There holds:
\bea
\label{eetfonafmetea}
&&\frac{d}{dt}\left\{\frac{\left<\mathcal M {\bf E}_2,{\bf E}_2 \right>}{\l^2}+O\left(\frac{C(M)b^{\frac32}}{|\log b|}\|{\bf E}_2\|_{X_Q}+\frac{C(M)b^3}{\l^2|\log b|^2}\right)\right\}\\
\nonumber &  \leq &  \frac{bC(M)}{\l^4}\left[\frac{b^{\frac32}}{|\log b|}\|{\bf E}_2\|_{X_Q}+\frac{b^3}{|\log b|^2}\right].
\eea
\end{proposition} 
\begin{proof}
To prove this proposition, which is the most technical step of the proof of the Theroem \ref{thmmain}, we shall use the second decomposition of the flow. Indeed, we have seen that the control of $b_s$ \fref{estb} is not enough good for our analysis, due to a too slow decay for the elements of the kernel of $\mathcal L^*$. To circumvent this technical problem, we have introduced the lift parameter $\hb$, whose the control \fref{poitnzeiboud} of the time derivate is better. In the first step, we shall begin to write the equations verify by the error term coming from the second decomposition, and its suitable derivatives. Thus, we shall compute the modified energy identity, whose we control each term in the last step of the proof.
\subsection {Equations verify by $\hat{\bf E}$ and its suitable derivatives:}
We recall the second decomposition of the flow.
\be
\label{defEtilde}
{\bf U} = \begin{array}{| c} u \\v \end{array} = \left( \tilde {\bf Q}_{\hat b}  + \widehat {\bf E}\right)_{\l} = \left(  \ \  \begin{array}{| c} Q + \tilde{\alpha}_{ \hat b} +\hat \e \\ \phi_Q + \tilde{\gamma}_{\hat b} + \hat \eta \end{array} \right)_{\l}.
\ee
Moreover, we have defined
\be
\label{defXi}
{\bf \Xi} = \widehat {\bf E} - {\bf E} = \begin{array} {| c} \zeta \\ \xi \end{array}
\ee
the gap between both decomposition. By definition:
\bee
\pa_s \left( \tilde{\bf Q}_{\hb} + \hat {\bf E}\right) = \pa_s \left( \tilde{\bf Q}_{b} + {\bf E}\right) 
\eee
Hence the equation \fref{equationE} becomes:
\be
\label{equationEtilde}
\pa_s \widehat {\bf E} - \frac{\l_s}{\l} \Lambda \widehat {\bf E} = \mathcal L \widehat {\bf E} + {\bf F}  +  \widehat{ \bf Mod} + {\bf G} + {\bf H} =  \mathcal L \widehat {\bf E} + \widehat{\boldsymbol {\mathcal F}},
\ee
where ${\bf F}$ and ${\bf G}$ are respectively defined by \fref{defF} and \fref{defG},  and the new modulation term is given by
\be
\label{defModtilde}
\widehat{\bf Mod} (s) =  \left( b + \frac{\l_s}{\l}\right) \Lambda \tilde {\bf Q}_{b} - \pa_s \tilde {\bf Q}_{\hat b},
\ee
and
\be
\label{defH}
{\bf H} =  - \mathcal L {\bf \Xi} - \frac{\l_s}{\l} \Lambda {\bf \Xi}.
\ee
The decomposition of the flow in the original variables is given by
\be
\label{defWtilde}
{\bf U} = \left( \tilde {\bf Q}_{\hat b} \right)_{\l} + \widehat {\bf W} = \left(  \ \  \begin{array}{| c} Q + \tilde{\alpha}_{\hat b} \\ \phi_Q + \tilde{\gamma}_{ \hat b} \end{array} \right)_{\l} + \begin{array}{| c}\hat w \\ \hat z \end{array}.
\ee
Hence, $ \widehat {\bf W}$ satisfies the equation:
\be
\label{equationWtilde}
\pa_t \widehat {\bf W} = \mathcal L_{\l} \widehat {\bf W} + \frac{1}{\l^2} \widehat{\boldsymbol {\mathcal F}_{\l}}
 \ee
In the rest of the paper, we use the following notation, in order to ease the clarity of the calculus.
\be
\nabla.{\bf U} = \nabla. \begin{array}{| c} u \\v \end{array} = \begin{array}{| c} \nabla.u \\ \nabla.v \end{array}
\ee
Introduce the differential operator of order one, which appear in the factorization of the operator $\mathcal L$:
\be
\label{defA}
\mathcal A_{\l} {\bf W} = \begin{array}{| c} Q_\l \nabla \mathcal M^{(1)}_\l (w,z) \\ \nabla z - \nabla  \phi_w  \end{array} = \begin{array}{| c} \nabla w + \nabla \phi_{Q_\l} w + Q_\l \nabla z \\ \nabla z - \nabla  \phi_w  \end{array}.
\ee
Indeed, we have this relation between the operators $\mathcal A$ and $\mathcal L$:
\be
\mathcal L_{\l} {\bf W} = \nabla . \left(\mathcal A_{\l} {\bf W} \right)
\ee
In the following, we use the following notations for the suitable derivatives of order one:
\bee
{\bf W}_1 &=& \begin{array} {| c} w_1 \\ z_1 \end{array} =  A_{\l} {\bf W}, \ \ \widehat{\bf W}_1 = \begin{array} {| c} \hat w_1 \\ \hat z_1 \end{array} =  A_{\l} \widehat {\bf W}, \\
{\bf E}_1 &=& \begin{array} {| c} \e_1 \\ \eta_1 \end{array} =  A {\bf E}, \ \ \widehat{\bf E}_1 = \begin{array} {| c} \hat \e_1 \\ \hat \eta_1 \end{array} =  A \widehat {\bf E},
\eee
and these notations for the suitable derivatives of order two:
\bee
{\bf W}_2 &=& \begin{array} {| c} w_2 \\ z_2 \end{array} =  \mathcal L_{\l} {\bf W}, \ \ \widehat{\bf W}_2 = \begin{array} {| c} \hat w_2 \\ \hat z_2 \end{array} =  \mathcal L_{\l} \widehat {\bf W} \\
{\bf E}_2 &=& \begin{array} {| c} \e_2 \\ \eta_2 \end{array} =  \mathcal L {\bf E}, \ \ \widehat{\bf E}_2 = \begin{array} {| c} \hat \e_2 \\ \hat \eta_2 \end{array} =  \mathcal L \widehat {\bf E}
\eee
In the same way, we notice:
\be
{\bf \Xi}_1 = \begin{array} {| c} \zeta_1 \\ \xi_1 \end{array} =  A {\bf \Xi}, \ \ {\bf \Xi}_2 = \begin{array} {| c} \zeta_2 \\ \xi_2 \end{array} =   \mathcal L {\bf \Xi}
\ee
Using these new notations, \fref{equationWtilde} becomes:
\be
\label{equationWtildebis}
\pa_t \widehat {\bf W} = \widehat{\bf W}_2 + \frac{1}{\l^2} \widehat{\boldsymbol {\mathcal F}_{\l}}  = \widehat{\bf W}_2 + \frac{1}{\l^2}\begin{array}{| c} \widehat {\mathcal F}_{1,\l} \\ \widehat {\mathcal F}_{2,\l} \end{array} .
\ee
Now, we introduce a second operator of order one
\be
\label{defV}
V_{\l} \widehat{\bf W} = [ \pa_t, A_{\l} ] \widehat{\bf W} = \begin{array} {| c} - \pa_t \left( \nabla \phi_{Q_\l} \right) \hat w - \pa_t Q_\l \nabla \hat z \\ 0 \end{array} =  \frac{\l_s}{\l} \begin{array} {| c} \left(  \nabla \phi_{\Lambda Q} \right)_{\l} \hat w + \left(\Lambda Q \right)_ \l \nabla \hat z \\ 0 \end{array}.
\ee
Hence $\widehat {\bf W} _1$ and $\widehat {\bf W} _2$ are respectively solutions of:

\bea
\label{equationW1tildebis}
\pa_t \widehat {\bf W}_1 &=& A_\l \widehat{\bf W}_2 + \frac{1}{\l^2} A_\l \widehat{\boldsymbol {\mathcal F}_{\l}} +  V_{\l} \widehat{\bf W}, \\
\label{equationW2tildebis}
\pa_t \widehat {\bf W}_2 &=&  \mathcal L_\l \widehat{\bf W}_2 + \frac{1}{\l^2} \mathcal L_\l \widehat{\boldsymbol {\mathcal F}_{\l}}+ \nabla. \left(V_{\l} \widehat{\bf W} \right).
\eea
\subsection{Modified energy identity}
\bea
\nonumber&&\frac 12 \frac d {dt} \left< \mathcal M_\l \widehat {\bf W}_2, \widehat {\bf W}_2\right> = \left< \pa_t \widehat {\bf W}_2, \mathcal M_\l \widehat {\bf W}_2\right> - \int \frac{\pa_t Q_\l}{2Q_\l^2}\hat w_2^2 \\
\nonumber&=& \left< \mathcal L_\l \widehat{\bf W}_2 + \frac{1}{\l^2} \mathcal L_\l \widehat{\boldsymbol {\mathcal F}_{\l}}+ \nabla. \left(V_{\l} \widehat{\bf W} \right), \mathcal M_\l \widehat {\bf W}_2\right> - \int \frac{\pa_t Q_\l}{2Q_\l^2}\hat w_2^2 \\
\nonumber&=& - \int Q_\l |\nabla \mathcal M^{(1)}_\l (\hat w_2, \hat z_2) |^2 - \int |\Delta \mathcal M^{(2)}_\l (\hat w_2, \hat z_2) |^2 \\
\nonumber&+&  \left<  \widehat {\bf W}_2,  \frac{1}{\l^2} \mathcal M_\l \mathcal L_\l\widehat{\boldsymbol {\mathcal F}_{\l}}\right > - \left< V_{\l} \widehat{\bf W}, \nabla \mathcal M_\l \widehat {\bf W}_2\right> \\
\label{firstenergy}&+& \int \frac{\hat w_2^2}{2\l^2Q_\l^2} \left[ \left( \hat b + \frac{\l_s}{\l}\right) (\Lambda Q)_\l \right] - \int \frac{\hat b  (\Lambda Q)_\l }{2\l^2Q_\l^2}\hat w_2^2.
\eea
For our analysis, the last term has a critical size. Moreover, this term hasn't definite sign. Hence, we must decompose this term in several manageable terms. In this purpose, we compute :
\bea
\label{step1termeembetant}
&& \frac{d}{dt} \left\{ \int \frac{\hat b  (\Lambda Q)_\l }{2\l^2Q_\l^2}\hat w_2 \hat w \right\}  \\ 
\nonumber&=&  \int \frac{d}{dt} \left\{\frac{\hat b  (\Lambda Q)_\l }{2\l^2Q_\l^2}\right\} \hat w_2 \hat w +  \int \frac{\hat b  (\Lambda Q)_\l }{2\l^2Q_\l^2}\hat w_2 \left[ \hat w_2+ \frac{1}{\l^2} \widehat{\mathcal F}_{1,\l} \right]  \\
\nonumber&+&  \int \frac{\hat b  (\Lambda Q)_\l }{2\l^2Q_\l^2} \hat w\left[ \mathcal L^{(1)}_{\l} (\hat w_2, \hat z_2) + \frac{1}{\l^2} \mathcal L^{(1)}_\l \left( \widehat{\mathcal F}_{1,\l}, \widehat{\mathcal F}_{2,\l}\right) + \nabla .V^{(1)}_\l (\hat w,\hat y) \right] .
\eea
But,
\bea
\label{newyork}&&\frac{d}{dt}\left\{\frac{\hb(\Lambda Q)_{\l}}{2\l^2Q_\l^2}\right\} =  \frac{d}{dt}\left\{\frac{\hb}2\left(\frac{\Lambda Q}{Q^2}\right)\left(\frac{r}{\l(t)}\right)\right\}\\
\nonumber& = & \frac{1}{2\l^2}\left[\hb_s\frac{\Lambda Q}{Q^2}-\hb\lsl y\cdot\nabla\left(\frac{\Lambda Q}{Q^2}\right)\right](y)\\
\nonumber& = & \frac{1}{\l^2Q}\left[4\hb\lsl-\hb_s+O\left(\frac{|\hb_s|+\hb|\lsl|}{1+r^2}\right)\right].
\eea
Now, using a intregration by part, we obtain
\bea
\nonumber &&\int \frac{\hat b  (\Lambda Q)_\l }{2\l^2Q_\l^2} \hat w \mathcal L^{(1)}_{\l} (\hat w_2, \hat z_2)= -\hb\int \frac{\hw\mathcal L^{(1)}_{\l} (\hat w_2, \hat z_2)}{\l^2Q_\l}+\int \frac{\hb(2Q+\Lambda Q)_{\l}}{2\l^2Q_\l^2}\hw\mathcal L^{(1)}_{\l} (\hat w_2, \hat z_2)\\
\nonumber & = & \frac{\hb}{\l^2}\int Q_\l \nabla(\mathcal M^{(1)}_\l (\hw_2, \hat z_2))\cdot\nabla\left[\frac{\hw}{Q_\l}+\hat z\right] \\
\label{newyork2}&-& \frac{\hb}{2\l^2}\int Q_\l \nabla (\mathcal M^{(1)}_\l (\hw_2, \hat z_2))\cdot\nabla \left[2\hat z+\frac{(2Q+\Lambda Q)_{\l}}{\l^2Q_\l^2}\hw\right]\\
\nonumber&=& -\frac{\hb}{\l^2} \int \hat w_2 \mathcal M^{(1)}_\l (\hw_2, \hat z_2) -  \frac{\hb}{2\l^2}\int Q_\l \nabla (\mathcal M^{(1)}_\l (\hw_2, \hat z_2))\cdot\nabla \left[2\hat z+\frac{(2Q+\Lambda Q)_{\l}}{\l^2Q_\l^2}\hw\right]
\eea
We have used in the above equality the fundamental degeneracy :
\be
\label{degenaracy}
\frac{\Lambda Q}{Q} + 2 = -\phi_{\Lambda Q} = O \left( \frac{1}{1+r^2}\right).
\ee
Injecting \fref{newyork} and \fref{newyork2} in \fref{step1termeembetant} yield
\bea
\label{step2termeembetant}
&& \frac{d}{dt} \left\{ \int \frac{\hat b  (\Lambda Q)_\l }{2\l^2Q_\l^2}\hat w_2 \hat w \right\}   = \int \frac{\hat b  (\Lambda Q)_\l }{2\l^2Q_\l^2}\hat w_2^2 -\frac{\hb}{\l^2} \int \hat w_2 \mathcal M^{(1)}_\l (\hw_2, \hat z_2)\\ 
\nonumber&+&  \int  \frac{\hat w_2 \hat w}{\l^2Q_\l}\left[4\hb\lsl-\hb_s \right] -  \frac{\hb}{2\l^2}\int Q_\l \nabla (\mathcal M^{(1)}_\l (\hw_2, \hat z_2))\cdot\nabla \left[2\hat z+\frac{(2Q+\Lambda Q)_{\l}}{\l^2Q_\l^2}\hw\right] \\
\nonumber&+&  \int \frac{\hat b  (\Lambda Q)_\l }{2\l^2Q_\l^2} \hat w\left[ \frac{1}{\l^2} \mathcal L^{(1)}_\l \left( \widehat{\mathcal F}_{1,\l}, \widehat{\mathcal F}_{2,\l}\right) + \nabla .V^{(1)}_\l (\hat w,\hat y) \right] +  \int \frac{\hat b  (\Lambda Q)_\l }{2\l^2Q_\l^2}\hat w_2 \left[  \frac{1}{\l^2} \widehat{\mathcal F}_{1,\l} \right]  \\
\nonumber &+& O\left(  \int  \frac{\hat w_2 \hat w}{\l^2Q(1+r^2)}\left[ |\hb_s|+\hb\left|\lsl\right| \right]\right).
\eea
Summing this equality with \fref{firstenergy}, we obtain a first modified energy identity. However, some terms have still critical size for our analysis. To solve this problem, we use the energy identity of $\hat{w}_1$:
\bea
\label{step3termeembetant}
&& -\frac d {dt} \left\{ \int \frac{\hb \hw_1^2}{\l^2Q_\l}\right\} \\
\nonumber&=& - \int \hw_1^2\frac d {dt} \left\{ \frac{\hb}{\l^2Q_\l}\right\} - 2\hb \int \frac{\hw_1}{\l^2Q_\l}\left[ A_\l^{(1)} (\hw_2,\hat z_2) + \frac 1 {\l^2} A_\l^{(1)}(\widehat{\mathcal F}_{1,\l}, \widehat{\mathcal F}_{2,\l}) + V_\l^{(1)}(\hw,\hat z) \right]
\eea
Next, we compute:
\bee
& &- \frac{d}{dt}\left\{\frac{\hb}{\l^2Q_\l}\right\}=-\frac{d}{dt}\left\{\frac{\hb}{Q\left(\frac{r}{\l(t)}\right)}\right\}=\frac{-1}{\l^2}\left[\frac{\hb_s}{Q}+\hb\frac{\lsl y\cdot\nabla Q}{Q^2}\right]\\
& = & \frac{-1}{\l^2Q_\l}\left[4\hb\lsl - \hb_s+O\left(\frac{|\hb_s|+\hb|\lsl|}{1+r^2}\right)\right],\\
\eee
Moreover:
\bee
-2\hb \int \frac{\hw_1}{\l^2Q_\l} A_\l^{(1)} (\hw_2,\hat z_2) = 2 \frac {\hb}{\l^2} \int \nabla.\hw_1 \mathcal M^{(1)}_\l (\hw_2,\hat z_2) = 2\frac{\hb}{\l^2} \int \hat w_2 \mathcal M^{(1)}_\l (\hw_2, \hat z_2)
\eee
Injecting both last terms in \fref{step3termeembetant} yields:
\bea
\label{step4termeembetant}
&& -\frac d {dt} \left\{ \int \frac{\hb \hw_1^2}{\l^2Q_\l}\right\} = 2\frac{\hb}{\l^2} \int \hat w_2 \mathcal M^{(1)}_\l (\hw_2, \hat z_2) + \int \frac{ \hw_1^2}{\l^2Q}\left[4\hb\lsl - \hb_s \right] \\
\nonumber&-& 2\hb \int \frac{\hw_1}{\l^2Q_\l}\left[ \frac 1 {\l^2} A_\l^{(1)}(\widehat{\mathcal F}_{1,\l}, \widehat{\mathcal F}_{2,\l}) + V_\l^{(1)}(\hw,\hat z) \right] +O\left(  \int  \frac{\hat w_1^2}{\l^2Q(1+r^2)}\left[ |\hb_s|+\hb\left|\lsl\right| \right]\right).
\eea
Summing \fref{step2termeembetant} with \fref{step4termeembetant} yields:
\bea
\label{step5termeembetant}
 \nonumber&& \frac{d}{dt} \left\{ \int \frac{\hat b  (\Lambda Q)_\l }{2\l^2Q_\l^2}\hat w_2 \hat w - \int \frac{\hb \hw_1^2}{\l^2Q_\l} \right\}   = \int \frac{\hat b  (\Lambda Q)_\l }{2\l^2Q_\l^2}\hat w_2^2 +\frac{\hb}{\l^2} \int \hat w_2 \mathcal M^{(1)}_\l (\hw_2, \hat z_2)\\ 
\nonumber&+&  \int  \frac{\hat w_2 \hat w + \hw_1^2}{\l^2Q_\l}\left[4\hb\lsl-\hb_s \right] -  \frac{\hb}{2\l^2}\int Q_\l \nabla (\mathcal M^{(1)}_\l (\hw_2, \hat z_2))\cdot\nabla \left[2\hat z+\frac{(2Q+\Lambda Q)_{\l}}{\l^2Q_\l^2}\hw\right] \\
\nonumber&+&  \int \frac{\hat b  (\Lambda Q)_\l }{2\l^2Q_\l^2} \hat w\left[ \frac{1}{\l^2} \mathcal L^{(1)}_\l \left( \widehat{\mathcal F}_{1,\l}, \widehat{\mathcal F}_{2,\l}\right) + \nabla .V^{(1)}_\l (\hat w,\hat y) \right] +  \int \frac{\hat b  (\Lambda Q)_\l }{2\l^2Q_\l^2}\hat w_2 \left[  \frac{1}{\l^2} \widehat{\mathcal F}_{1,\l} \right]  \\
 \nonumber &-&2\hb \int \frac{\hw_1}{\l^2Q_\l}\left[ \frac 1 {\l^2} A_\l^{(1)}(\widehat{\mathcal F}_{1,\l}, \widehat{\mathcal F}_{2,\l}) + V_\l^{(1)}(\hw,\hat z) \right] \\
 &+& O\left(  \int  \frac{|\hat w_2 \hat w| + |\hw_1|^2}{\l^2Q(1+r^2)}\left[ |\hb_s|+\hb\left|\lsl\right| \right]\right).
\eea
By integration by parts, we have:
\bee
\int \frac{\wh_2\wh}{Q_\l}=-\int Q_\l\nabla\mathcal M^{(1)}_\l (\wh,\hat z) \cdot\nabla\left(\frac{\wh}{Q_\l}\right)=-\int\frac{\wh_1^2}{Q_\l}+\int Q_\l\nabla \mathcal M^{(1)}_\l( \wh, \hat z) \cdot\nabla \hat z
\eee
and thus
\be
\label{nioennene}
\int \frac{\wh_2\wh+\wh_1^2}{Q_\l}=\int \wh_1\cdot\nabla \hat z.
\ee
In the same way:
\bee
&&-\int \frac{1}Q_\l\left\{\wh\mathcal L^{(1)}_\l \left( \widehat{\mathcal F}_{1,\l}, \widehat{\mathcal F}_{2,\l}\right)+\wh_2 \widehat{\mathcal F}_{1,\l}\right\}-2\int\frac{1}{Q_\l} \wh_1\cdot \mathcal A^{(1)}_\l \left( \widehat{\mathcal F}_{1,\l}, \widehat{\mathcal F}_{2,\l}\right)\\
& =&  \int Q_\l\nabla \mathcal M^{(1)}_\l \left( \widehat{\mathcal F}_{1,\l}, \widehat{\mathcal F}_{2,\l}\right)\cdot\nabla \left(\mathcal M^{(1)}_\l (\wh, \hat z)- \hat z\right) \\
&+&\int Q_\l\nabla \mathcal M^{(1)}_\l \left( \hw,\hat z\right)\cdot\nabla \left(\mathcal M^{(1)}_\l \left(\widehat{\mathcal F}_{1,\l}, \widehat{\mathcal F}_{2,\l}\right)- \widehat{\mathcal F}_{2,\l} \right)\\
&-&2\int Q_\l\nabla \mathcal M^{(1)}_\l \left( \widehat{\mathcal F}_{1,\l}, \widehat{\mathcal F}_{2,\l}\right)\cdot\nabla\mathcal M^{(1)}_\l (\wh, \hat z)\\
& = & -2\int\wh_1\cdot\nabla  \widehat{\mathcal F}_{2,\l}-2\int Q\nabla \mathcal M^{(1)}_\l \left( \widehat{\mathcal F}_{1,\l}, \widehat{\mathcal F}_{2,\l}\right)\cdot\nabla \hat z.
\eee
Hence:
\bea
\nonumber& & \int \frac{\hat b  (\Lambda Q)_\l }{2\l^2Q_\l^2} \left[ \frac{\hat w}{\l^2} \mathcal L^{(1)}_\l \left( \widehat{\mathcal F}_{1,\l}, \widehat{\mathcal F}_{2,\l}\right) + \frac{\hw_2}{\l^2} \widehat{\mathcal F}_{1,\l} \right] -2\hb \int \frac{\hw_1}{\l^2Q_\l}\left[ \frac 1 {\l^2} A_\l^{(1)}(\widehat{\mathcal F}_{1,\l}, \widehat{\mathcal F}_{2,\l})\right] \\
\nonumber& = & \frac{\hb}{2\l^2}\int\frac{\Lambda Q_\l+2Q_\l}{Q_\l^2}\left[ \frac{\hat w}{\l^2} \mathcal L^{(1)}_\l \left( \widehat{\mathcal F}_{1,\l}, \widehat{\mathcal F}_{2,\l}\right) + \frac{\hw_2}{\l^2} \widehat{\mathcal F}_{1,\l} \right] \\
\label{ppppp}&-& \frac{\hb}{\l^2}\left\{\int\wh_1\cdot\nabla  \widehat{\mathcal F}_{2,\l}+\int Q\nabla \mathcal M^{(1)}_\l \left( \widehat{\mathcal F}_{1,\l}, \widehat{\mathcal F}_{2,\l}\right)\cdot\nabla \hat z\right\}.
\eea
Injecting \fref{nioennene} and \fref{ppppp} in \fref{step5termeembetant} yields:
\bea
\label{step6termeembetant}
 \nonumber&& \frac{d}{dt} \left\{ \int \frac{\hat b  (\Lambda Q +2Q)_\l }{2\l^2Q_\l^2}\hat w_2 \hat w - \int \frac{\hb \hw_1\nabla \hat z}{\l^2} \right\}   = \int \frac{\hat b  (\Lambda Q)_\l }{2\l^2Q_\l^2}\hat w_2^2 +\frac{\hb}{\l^2} \int \hat w_2 \mathcal M^{(1)}_\l (\hw_2, \hat z_2)\\ 
\nonumber&+&  \int  \frac{\hw_1 \nabla \hat z}{\l^2}\left[4\hb\lsl-\hb_s \right] -  \frac{\hb}{2\l^2}\int Q_\l \nabla (\mathcal M^{(1)}_\l (\hw_2, \hat z_2))\cdot\nabla \left[2\hat z+\frac{(2Q+\Lambda Q)_{\l}}{\l^2Q_\l^2}\hw\right] \\
 \nonumber&+&\frac{\hb}{2\l^2}\int\frac{\Lambda Q_\l+2Q_\l}{Q_\l^2}\left[ \frac{\hat w}{\l^2} \mathcal L^{(1)}_\l \left( \widehat{\mathcal F}_{1,\l}, \widehat{\mathcal F}_{2,\l}\right) + \frac{\hw_2}{\l^2} \widehat{\mathcal F}_{1,\l} \right] \\
\nonumber&-& \frac{\hb}{\l^2}\left\{\int\wh_1\cdot\nabla  \widehat{\mathcal F}_{2,\l}+\int Q\nabla \mathcal M^{(1)}_\l \left( \widehat{\mathcal F}_{1,\l}, \widehat{\mathcal F}_{2,\l}\right)\cdot\nabla \hat z\right\} \\
\nonumber&+&  \int \frac{\hat b  (\Lambda Q)_\l }{2\l^2Q_\l^2} \hat w \nabla .V^{(1)}_\l (\hat w,\hat y)  - 2\hb \int \frac{\hw_1}{\l^2Q_\l} V_\l^{(1)}(\hw,\hat z) \\
 &+& O\left(  \int  \frac{|\hat w_2 \hat w| + |\hw_1|^2}{\l^2Q(1+r^2)}\left[ |\hb_s|+\hb\left|\lsl\right| \right]\right).
\eea
To conclude, summing the above equality with \fref{firstenergy} yields the modified energy identity:
\bea
\nonumber&&\frac 12 \frac d {dt}\left\{ \left< \mathcal M_\l \widehat {\bf W}_2, \widehat {\bf W}_2\right> +  \int \frac{\hat b  (\Lambda Q +2Q)_\l }{\l^2Q_\l^2}\hat w_2 \hat w - 2\int \frac{\hb \hw_1\nabla \hat z}{\l^2} \right\} \\
\nonumber&=& - \int Q_\l |\nabla \mathcal M^{(1)}_\l (\hat w_2, \hat z_2) |^2 - \int |\Delta \mathcal M^{(2)}_\l (\hat w_2, \hat z_2) |^2 +\frac{\hb}{\l^2} \int \hat w_2 \mathcal M^{(1)}_\l (\hw_2, \hat z_2)\\
\nonumber&+&  \int  \frac{\hw_1 \nabla \hat z}{\l^2}\left[4\hb\lsl-\hb_s \right] -  \frac{\hb}{2\l^2}\int Q_\l \nabla (\mathcal M^{(1)}_\l (\hw_2, \hat z_2))\cdot\nabla \left[2\hat z+\frac{(2Q+\Lambda Q)_{\l}}{\l^2Q_\l^2}\hw\right] \\
 \nonumber&+&\frac{\hb}{2\l^2}\int\frac{\Lambda Q_\l+2Q_\l}{Q_\l^2}\left[ \frac{\hat w}{\l^2} \mathcal L^{(1)}_\l \left( \widehat{\mathcal F}_{1,\l}, \widehat{\mathcal F}_{2,\l}\right) + \frac{\hw_2}{\l^2} \widehat{\mathcal F}_{1,\l} \right] +  \left<  \widehat {\bf W}_2,  \frac{1}{\l^2} \mathcal M_\l \mathcal L_\l\widehat{\boldsymbol {\mathcal F}_{\l}}\right > \\
\nonumber&-& \frac{\hb}{\l^2}\left\{\int\wh_1\cdot\nabla  \widehat{\mathcal F}_{2,\l}+\int Q\nabla \mathcal M^{(1)}_\l \left( \widehat{\mathcal F}_{1,\l}, \widehat{\mathcal F}_{2,\l}\right)\cdot\nabla \hat z\right\} \\
\nonumber&+&  \int \frac{\hat b  (\Lambda Q)_\l }{2\l^2Q_\l^2} \hat w \nabla .V^{(1)}_\l (\hat w,\hat y)  - 2\hb \int \frac{\hw_1}{\l^2Q_\l} V_\l^{(1)}(\hw,\hat z) - \left< V_{\l} \widehat{\bf W}, \nabla \mathcal M_\l \widehat {\bf W}_2\right>\\
\label{modifiedenergy}&+& \int \frac{\hat w_2^2}{2\l^2Q_\l^2} \left[ \left( \hat b + \frac{\l_s}{\l}\right) (\Lambda Q)_\l \right] +O\left(  \int  \frac{|\hat w_2 \hat w| + |\hw_1|^2}{\l^2Q(1+r^2)}\left[ |\hb_s|+\hb\left|\lsl\right| \right]\right).
\eea
Now, all terms are manageable. We treat each term in the rest of this section. We shall make an intensive use of the bounds of the Lemma \ref{lemmaradiation} and the interpolation bound of the Proposition \ref{interpolation}. 
\subsection{Boundary terms in time}
We must verify that this both terms aren't bigger than the quantity which we would like to control.
\bee
\left| \int \frac{\bh(\Lambda Q+2Q)}{\l^2Q^2}\hat w_2\hat w\right|&\lesssim& \frac{b}{\l^4}\int \frac{|\eh_2\eh|}{(1+r^2)Q}\lesssim \frac{b}{\l^4}\left(\int \frac{\eh_2^2}{Q}\right)^{\frac 12}\left(\int \eh^2\right)^{\frac 12}\\
& \lesssim &\frac{C(M)b}{\l^4}\left[\|\e_2\|_{L^2_Q}^2+\|\zeta_2\|_{L^2_Q}^2+\|\zeta\|_{L^2}^2+\|\e\|_{L^2}^2\right]\\
&\lesssim & \frac{C(M)b^3}{\l^4|\log b|^2}
\eee
The control of the second term is more delicate.
\be
\label{gpaw3}
\left|\frac1{\l^2}\int b\hat w_1\cdot\nabla \hat z \right|  \lesssim  \frac{b}{\l^4}\left( \int \nabla \hat \e \nabla \hat \eta - \int \frac{\nabla Q}{Q} \hat \e \nabla \hat \eta + \int Q |\nabla \hat \eta |^2 \right)\\
\ee
Now,
\bea
\nonumber \int Q |\nabla \hat \eta |^2 &\lesssim& \left( \int | \nabla \eta|^2 \int \frac{|\nabla \eta|^2}{1+\tau^8} \right)^{\frac 12} + \int \frac{\nabla \xi^2}{1+\tau^2} \\
 \label{gpaw} &\lesssim& \|\eta\|_{\dot H^1}\|{\bf E}_2\|_{X_Q} + \int \frac{\nabla \xi^2}{1+\tau^2} \\
 &\lesssim& K^*b^{\frac52}|\log b|^2 + C(M) \frac{b^2}{|\log b|^2}\lesssim C(M) \frac{b^2}{|\log b|^2}
\eea
Hence
\bea
\nonumber&&\left| \int \nabla \hat \e \nabla \hat \eta - \int \frac{\nabla Q}{Q} \hat \e \nabla \hat \eta \right| \lesssim \left(\left\{ \int (1+\tau^2)|\nabla \hat \e|^2 + \int |\hat \e|^2\right\} \int Q |\nabla \hat \eta |^2 \right)^{\frac 12} \\
\nonumber&\lesssim&C(M) \frac{b}{|\log b|} \left[ \int (1+\tau^2)|\nabla \e|^2 + \int |\e|^2 + \int (1+\tau^2)|\nabla \zeta|^2 + \int |\zeta|^2\right]^{\frac12}\\
\label{gpaw2} &\lesssim&C(M) \frac{b}{|\log b|} \left[ \|{\bf E}_2\|_{X_Q}^2 +\frac{b^2}{|\log b|^2}\right]^{\frac12} \lesssim C(M)\frac{b^2}{|\log b|^2}
\eea
Injecting \fref{gpaw}, \fref{gpaw2} in \fref{gpaw3} yields
\be
\left|\frac1{\l^2}\int b\hat w_1\cdot\nabla \hat z \right|  \lesssim C(M)\frac{b^3}{\l^4|\log b|^2}
\ee
\end{proof}
\subsection{Quadratic terms}
\label{quadratic}
We aren't in position to treat the term $\frac{\hb}{\l^2} \int \hat w_2 \mathcal M^{(1)}_\l (\hw_2, \hat z_2)$, which has a wrong sign. Indeed, as $\int \hat w_2 = 0 $, $\mathcal M$ is a positive operator. We shall treat this term in the conclusion. However, both terms $- \int Q_\l |\nabla \mathcal M^{(1)}_\l (\hat w_2, \hat z_2) |^2 - \int |\mathcal M^{(2)}_\l (\hat w_2, \hat z_2) |^2$ are non positive, and they shall be very helpful to estimate some terms in the rest of this proof.
\par
Now, with the modulation equations, we have proved that :
\bee
\left| \frac{\l_s}{\l} \right| \lesssim b, \ \ |\hb_s| \lesssim b^2.
\eee
Thus,
\bee
&&\left|\int\frac{\eh_1\cdot\nabla \hat z}{2\l^6}\left[8\bh\lsl-2\bh_s\right]\right| \lesssim  \frac{b^2}{\l^6}\int |\eh_1\cdot\nabla \hat \eta| \lesssim \frac{C(M)b^4}{\l^6|\log b|^2}.
\eee
We have still proved the last inequality in the last subsection. In the same way, we can control the error term:
\bee
&&\int \frac{|\eh_2\eh|+|\eh_1|^2}{\l^6Q(1+\tau^2)}(|b_s|+b|\lsl+b|+b^2)\\
& \lesssim& \frac{b^2}{\l^6}\int\left[(1+\tau^4)\eh_2^2+(1+\tau^2)|\eh_1|^2+\eh^2\right]\\
& \lesssim & C(M)\frac{b^2}{\l^6}\left[\|{\bf E}_2\|_{X_Q}^2+\frac{b^2}{|\log b|^2}\right].
\eee
Next, we can estimate this term :
\bee
&&\left|\frac{\bh}{2\l^2}\int Q_\l \nabla (\mathcal M^{(1)}_\l (\hw_2, \hat z_2))\cdot\nabla \left[2\hat z+\frac{(2Q+\Lambda Q)_{\l}}{\l^2Q_\l^2}\hw\right]\right|\\
& \leq& \frac{1}{100}\int Q_\l| \nabla (\mathcal M^{(1)}_\l (\hw_2, \hat z_2))|^2+\frac{Cb^2}{\l^6}\int Q\left[|\nabla \hat \eta|^2+\frac{|\nabla \eh|^2}{(1+\tau^4)Q^2}+\frac{ \eh^2}{(1+\tau^6)Q^2}\right]\\
& \leq & \frac{1}{100}\int Q_\l|  \nabla (\mathcal M^{(1)}_\l (\hw_2, \hat z_2))|^2+C(M)\frac{b^2}{\l^6}\left[ \|{\bf E}_2\|_{X_Q}^2+\frac{b^2}{|\log b|^2}\right].
\eee
We are focusing now on all terms depending on the potential $V_\l$. We recall that:
\bee
V_{\l} \widehat{\bf W} = [ \pa_t, A_{\l} ] \widehat{\bf W} =  \frac{\l_s}{\l} \begin{array} {| c} \left(  \nabla \phi_{\Lambda Q} \right)_{\l} \hat w + \left(\Lambda Q \right)_ \l \nabla \hat z \\ 0 \end{array}.
\eee
and thus, we have the following bound
\be
\left|V_{\l} \widehat{\bf W} \right| \lesssim \frac{b}{\l^2}\left[ \frac{\hat \e}{1+\tau^3} + \frac{\nabla \hat \eta}{1+\tau^4}\right]
\ee
Hence,
\bee
 &&\left| \int \frac{\hat b  (\Lambda Q)_\l }{2\l^2Q_\l^2} \hat w \nabla .V^{(1)}_\l (\hat w,\hat y)  \right| \lesssim \frac{b^2}{\l^6}\int (1+\tau^4) \hat \e \left[ \frac{\hat \e}{1+\tau^4} + \frac{\nabla \hat \eta}{1+\tau^5} + \frac{\nabla \hat \e}{1+\tau^3} + \frac{\nabla^2 \hat \eta}{1+\tau^4}  \right] \\
 &\lesssim& \frac{b^2}{\l^6}\left[ \|{\bf E}_2\|_{X_Q}^2 + \int |\zeta|^2 + \int (1+\tau^2)|\nabla \zeta|^2 + \int \frac{|\nabla \xi|^2}{1+\tau^2} +  \int |\Delta \xi|^2\right] \\
 &\lesssim& \frac{b^2}{\l^6}\left[ \|{\bf E}_2\|_{X_Q}^2 + \frac{b^2}{|\log b|^2}\right] 
\eee
Next,
\bee
 \left|\hb \int \frac{\hw_1}{\l^2Q_\l} V_\l^{(1)}(\hw,\hat z)\right| &\lesssim& \frac{b^2}{\l^6}\int (1+\tau^4) \nabla \hat \e \left[ \frac{\hat \e}{1+\tau^3} + \frac{\nabla \hat \eta}{1+\tau^4} \right] \\
 &\lesssim& \frac{b^2}{\l^6}\left[ \|{\bf E}_2\|_{X_Q}^2 + \int |\zeta|^2 + \int (1+\tau^2)|\nabla \zeta|^2 + \int \frac{|\nabla \xi|^2}{1+\tau^2}\right] \\
 &\lesssim& \frac{b^2}{\l^6}\left[ \|{\bf E}_2\|_{X_Q}^2 + \frac{b^2}{|\log b|^2}\right] 
\eee
As $V^{(2)}_\l (\hat w_2,\hat z_2) = 0$, we have for the last term depending on $V_\l$
\bee
&&\left| \left< V_{\l} \widehat{\bf W}, \nabla \mathcal M_\l \widehat {\bf W}_2\right> \right| = \left| \int V^{(2)}_l(\hat w_2,\hat z_2)   \nabla \mathcal M_\l^{(1)}(\hat w_2,\hat z_2) \right|\\
&\leq& \frac{1}{100} \int Q |\nabla \mathcal M_\l^{(1)}(\hat w_2,\hat z_2) |^2 + O\left(\frac{b}{\l^2} \int (1+\tau^4)\left[ \frac{|\hat \e|^2}{1+\tau^6} + \frac{|\nabla \hat \eta|^2}{1+\tau^6}\right] \right)\\
&\leq& \frac{1}{100} \int Q |\nabla \mathcal M_\l^{(1)}(\hat w_2,\hat z_2) |^2 + \frac{b}{\l^2} O\left( \|{\bf E}_2\|_{X_Q}^2 + \int |\zeta|^2+ \int \frac{|\nabla \xi|^2}{1+\tau^2} \right) \\
&\leq& \frac{1}{100} \int Q |\nabla \mathcal M_\l^{(1)}(\hat w_2,\hat z_2) |^2 + C(M) \frac{b^2}{\l^6}\left[ \|{\bf E}_2\|_{X_Q}^2 + \frac{b^2}{|\log b|^2}\right] 
\eee
To conclude this subsection dedicated to the quadratic terms, we must estimate the last term. We must carefully study this term because of the very bad estimation coming from \fref{estlambda} and \fref{deformation}:
\bee
\left| \hb + \lsl \right| \lesssim \frac{b}{|\log b|}.
\eee
We recall the decomposition of the gap $\boldsymbol \Xi$ using in the proof of the lemma \fref{lemmaradiation}:
\bee
{\boldsymbol \Xi}_2 = \mathcal L \boldsymbol \Xi = \mathcal L \boldsymbol \Xi_{sm} + (b-\hb)\Lambda {\bf Q} + \boldsymbol {\mathcal R}
\eee
with the estimation
\bee
\boldsymbol{\mathcal R} = (b-\hat b) \begin{array} {|c} O\left(\frac{{\bf 1}_{r \geq B_1}}{1+r^4}\right) \\
O\left(\frac{{\bf 1}_{r \geq B_1}}{1+r^2}\right) 
\end{array}
\eee
Hence,
\bee
&&\left|\int \frac{\wh_2^2}{\l^2Q_\l^2}\left[ \left(\lsl+\bh\right)(\Lambda Q)_{\l}\right]\right| \lesssim  \frac{b}{|\log b|\l^6}\left| \int \frac{(\e_2+\mathcal L^{(1)} {\boldsymbol \Xi}_{sm}+\boldsymbol{\mathcal R}^{(1)}+(b-\hb)\Lambda Q)^2}{Q^2}\Lambda Q\right|\\
& \lesssim & \frac{b}{|\log b|\l^6}\left[\|{\bf E}_2\|_{X_Q}^2+\|\mathcal L {\boldsymbol \Xi}_{sm}\|_{X_Q}^2+\|\boldsymbol {\mathcal R}\|_{X_Q}^2+\frac{b^2}{|\log b|^2}\left|\int\left(\frac{\Lambda Q}{Q}\right)^2\Lambda Q\right|\right.\\
& + & \left.\frac{b}{|\log b|}\left|\int(\e_2+\mathcal L^{(1)} {\boldsymbol \Xi}_{sm})\frac{(\Lambda Q)^2}{Q^2}\right|+\frac{b^2}{|\log b|^2}\int_{r\geq B_1}\frac{1}{1+r^4}\right]\\
& \lesssim & \frac{b}{\l^6}\left[\frac{b^{\frac 32}}{|\log b|}\|{\bf E}_2\|_{X_Q}+\frac{b^3}{|\log b|^2}+\frac{b^2}{|\log b|^2}\left|\int\left(\frac{\Lambda Q}{Q}\right)^2\Lambda Q\right|+\frac{b}{|\log b|^2}\left|\int(\e_2+\mathcal L^{(1)} {\boldsymbol \Xi}_{sm})\frac{(\Lambda Q)^2}{Q^2}\right|\right].
\eee
Now, using that $\Lambda^2Q = \nabla.(r\Lambda Q)$, we have that 
\bee
\nabla \phi_{\Lambda^2Q} =  r \Lambda Q
\eee
Now, using the radial representation of Poisson field, we have:
\bee
\phi_{\Lambda^2Q} (0) = \int_0^{\infty} \Lambda^2Q \log r r dr = 0.
\eee 
Thus,
\bee
\phi_{\Lambda^2Q}(r) = \int_0^r r \nabla.(rQ)rdr = r^2Q.
\eee
Using the explicit formula of $\Lambda^iQ$, for $0 \leq i \leq 2$, we obtained:
\be
\label{Mlambda2Q}
\mathcal M (\Lambda^2 Q,r^2Q) = \left| \begin{array}{l} \frac{\Lambda^2 Q}{Q} + r^2Q  \\ r^2Q - \phi_{\Lambda ^2 Q}\end{array} \right. = \left| \begin{array}{l} \left(\frac{\Lambda Q}{Q}\right)^2  \\ 0\end{array} \right. 
\ee
We are now in position to compute:
\bee
\int\left(\frac{\Lambda Q}{Q}\right)^2\Lambda Q = \left< \mathcal M (\Lambda^2 Q,r^2Q),\Lambda {\bf Q}\right> = \left<  \left| \begin{array}{l} \Lambda^2 Q  \\ r^2Q\end{array} \right. ,\mathcal M (\Lambda {\bf Q})\right> = -2\int \Lambda^2Q = 0.
\eee
We recall the following degeneracy :
\bee
\left( \frac{\Lambda Q}{Q} \right)^2 = 4 + O \left( \frac{1}{1+\tau^2}\right)
\eee
Hence, with \fref{estimatonebisbis}
\bee
\left|\int \mathcal L^{(1)} {\boldsymbol \Xi}_{sm}\frac{(\Lambda Q)^2}{Q^2}\right| \lesssim \int \frac{\left|\mathcal L^{(1)} {\boldsymbol \Xi}_{sm} \right|}{1+\tau^2} \lesssim \|\mathcal L^{(1)} {\boldsymbol \Xi}_{sm}\|_{L^2} \lesssim \frac{b^2}{|\log b|}.
\eee
Finally, using the cancellation $\int \e_2 = 0$, we obtain
\bee 
&&\int \e_2\frac{(\Lambda Q)^2}{Q^2} = \left< {\bf E}_2,\mathcal M(\Lambda^2{\bf Q}+2\Lambda {\bf Q})\right>=\left< \mathcal M {\bf E}_2, \Lambda^2{\bf Q}+2\Lambda {\bf Q}\right> \\
&=& \int \mathcal M^{(1)} (\e_2,\eta_2) \nabla.(r(\Lambda Q + 2Q)) +  \int \nabla \mathcal M^{(2)} (\e_2,\eta_2) \nabla\left(r^2Q + \frac{8}{1+r^2}\right) \\
&=&- \int \nabla \mathcal M^{(1)} (\e_2,\eta_2) r(\Lambda Q + 2Q)) - \int \Delta \mathcal M^{(2)} (\e_2,\eta_2) \left(r^2Q + \frac{8}{1+r^2}\right)
\eee
Thus,
\bee
\left| \int \e_2\frac{(\Lambda Q)^2}{Q^2}\right| &\lesssim&  \left(\int Q|\nabla( \mathcal M^{(1)} (\e_2,\eta_2))|^2\right)^{\frac 12}\left(\int \frac{r^2(\Lambda Q+2Q)^2}{Q}\right)^{\frac 12} \\
&+&  \left(\int |\Delta( \mathcal M^{(2)} (\e_2,\eta_2))|^2\right)^{\frac 12}\left(\int \left[r^2Q + \frac{8}{1+r^2}\right]^2\right)^{\frac 12}\\
&\lesssim& \left[ \int Q|\nabla( \mathcal M^{(1)} (\hat \e_2,\hat \eta_2))|^2 + \int |\Delta( \mathcal M^{(2)} (\hat \e_2,\hat \eta_2))|^2 + \frac{b^4}{|\log b|^2}\right]^{\frac 12},
\eee
where we have use \fref{degenrateonebis} in the last line. Hence,
\bee
&&\frac{b^2}{\l^6|\log b|^2}\left| \int \e_2\frac{(\Lambda Q)^2}{Q^2}\right| \leq \frac{1}{100\l^6}\left| \int \e_2\frac{(\Lambda Q)^2}{Q^2}\right|^2 + O\left(\frac{b^4}{\l^6|\log b|^4} \right)\\
&\leq& \frac{1}{100\l^6}\left[ \int Q|\nabla( \mathcal M^{(1)} (\hat \e_2,\hat \eta_2))|^2 + \int |\Delta( \mathcal M^{(2)} (\hat \e_2,\hat \eta_2))|^2\right]+ O\left(\frac{b^4}{\l^6|\log b|^4} \right)
\eee
The collection above bounds yields the admissible control:
\bee
&&\left|\int \frac{\wh_2^2}{\l^2Q_\l^2}\left[ \left(\lsl+\bh\right)(\Lambda Q)_{\l}\right]\right| \\
& \lesssim &  \frac{b}{\l^6}\left[b^{\frac 32}\|{\bf E}_2\|_{X_Q}+\frac{b^3}{|\log b|^2}\right]+\frac{1}{100\l^6}\left[ \int Q|\nabla( \mathcal M^{(1)} (\hat \e_2,\hat \eta_2))|^2 + \int |\Delta( \mathcal M^{(2)} (\hat \e_2,\hat \eta_2))|^2\right].
\eee
To conclude the proof of the Proposition \ref{htwoqmonton}, we shall focus on the term depending on $\widehat{\boldsymbol {\mathcal F}_{\l}}$. We recall that:
\be
\label{ftgyhujikolpm}
 \widehat{\boldsymbol{\mathcal F}_\l} =   \left({\widehat{\bf Mod}}\right)_\l + {\bf H}_\l + {\bf G}_\l + \left(\tilde{\bf \Psi}_b\right)_\l +  \left({\bf \Theta}_b(\e,\eta)\right)_\l + \left({\bf N}(\e,\eta)\right)_\l
\ee
As the dependence on $\widehat{\boldsymbol {\mathcal F}_\l}$ in \fref{modifiedenergy} is linear, we shall estimate each term of the decomposition \fref{ftgyhujikolpm} in  separate subsections. 
\subsection{$\tilde{\boldsymbol \Psi}_b$ terms}
First, we use the bounds \fref{roughboundltaowt2} and \fref{cneneoneonoenoet} coming from the Proposition \ref{localization} to estimate:
\bee
&&\left|\left<\widehat{\bf W}_2, \mathcal M_\l \mathcal L_\l \tilde{\boldsymbol \Psi}_b\right> \right| = \left|\left< \mathcal M_\l \widehat{\bf W}_2,\mathcal L_\l \tilde{\boldsymbol \Psi}_b\right> \right| \\
&\lesssim&\left| \int \nabla  \mathcal M_\l^{(1)}(\hat w_2, \hat z_2) Q \nabla \mathcal M_\l^{(1)} \left(\tilde{\Psi}^{(1)}_{b,\l},\tilde{\Psi}^{(2)}_{b,\l}\right)\right| + \left|\int \nabla \mathcal M_\l^{(2)}(\hat w_2, \hat z_2) \nabla \mathcal L_\l^{(2)} \left(\tilde{\Psi}^{(1)}_{b,\l},\tilde{\Psi}^{(2)}_{b,\l}\right)  \right|\\
&\leq&\frac{1}{100}\left[ \int Q|\nabla( \mathcal M_\l^{(1)} (\hat w_2,\hat z_2))|^2 + \int |\Delta( \mathcal M_\l^{(2)} (\hat w_2,\hat z_2))|^2\right] \\
&+& O\left(\frac{1}{\l^6} \left[ \int Q \left|\nabla \mathcal M_\l^{(1)} \left(\tilde{\Psi}^{(1)}_{b,\l},\tilde{\Psi}^{(2)}_{b,\l}\right)\right|^2 + \int \left|\mathcal L_\l^{(2)} \left(\tilde{\Psi}^{(1)}_{b,\l},\tilde{\Psi}^{(2)}_{b,\l}\right)  \right|^2  \right] \right) \\
&\leq&\frac{1}{100}\left[ \int Q|\nabla( \mathcal M_\l^{(1)} (\hat w_2,\hat z_2))|^2 + \int |\Delta( \mathcal M_\l^{(2)} (\hat w_2,\hat z_2))|^2\right]  + O\left(\frac{b^4}{\l^6|\log b|^2} \right).
\eee
Now, with the bound \fref{roughboundltaowt} and the degeneracy \fref{degenaracy}, we obtain
\bee
&&\left| \frac{\hb}{2\l^2}\int\frac{\Lambda Q_\l+2Q_\l}{Q_\l^2}\left[ \frac{\hat w}{\l^2} \mathcal L^{(1)} \left(\tilde{\Psi}^{(1)}_{b,\l},\tilde{\Psi}^{(2)}_{b,\l}\right)\right]  \right| \\
&\lesssim&\frac{b}{\l^6}\int (1+\tau^2) \left\{\hat \e\mathcal L^{(1)} \left(\tilde{\Psi}^{(1)}_{b},\tilde{\Psi}^{(2)}_{b}\right) + \hat \e_2 \tilde{\Psi}^{(1)}_{b} \right\} \\
&\lesssim& \frac{b}{\l^6} \left [ \left\|\mathcal L^{(1)} \left(\tilde{\Psi}^{(1)}_{b},\tilde{\Psi}^{(2)}_{b}\right) \right\|_{L^2_Q}\|\hat \e\|_{L^2} + \| \tilde{\Psi}^{(1)}_{b}\|_{L^2}\|\hat \e_2\|_{L^2_Q} \right ] \\
&\lesssim& \frac{b}{\l^6} \frac{b^2}{|\log b|} \left[ \|{\bf E}_2 \|_{X_Q} + \|\zeta\|_{L^2} + \|\zeta\|_{L^2_Q}\right] \lesssim \frac{b}{\l^6} \frac{b^2}{|\log b|} \left[ \|{\bf E}_2 \|_{X_Q} + \frac{b}{|\log b|}\right]
\eee
With \fref{roughboundltaowt2} we can control the following term
\bee
&&\left|\frac{b}{\l^6} \int \hat \e_1 \nabla\tilde{\Psi}^{(2)}_{b}  \right| \lesssim \frac{b}{\l^6} \left(\int (1+\tau^2)|\hat \e_1|^2 \int \frac{\left| \nabla\tilde{\Psi}^{(2)}_{b}\right|^2}{1+\tau^2} \right)^{\frac12} \\
& \lesssim &  \frac{b}{\l^6}\frac{b^{\frac 52}}{|\log b|}\left[C(M) \|{\bf E}_2\|_{X_Q}+\left(\int(1+\tau^2) |\nabla \zeta|^2+\int |\zeta|^2+\int\frac{ |\nabla \xi|^2}{1+r^2}\right)^{\frac 12}\right]\\
& \lesssim & \frac{b}{\l^6}\left[b^2 \|{\bf E}_2\|_{X_Q}+\frac{b^3}{|\log b|^2}\right].
\eee
We recall the bound \fref{gpaw}
\bee
\int Q |\nabla \hat \eta|^2 \lesssim C(M) \frac{b^2}{|\log b|}
\eee
Using this bound together the bound \fref{cneneoneonoenoet} :
\bee
\left|\frac{b}{\l^6} \int \nabla \hat \eta Q \nabla \mathcal M^{(1)} \left(\tilde{\Psi}^{(1)}_{b},\tilde{\Psi}^{(2)}_{b}\right)  \right| &\lesssim& \frac{b}{\l^6} \left(\int Q |\nabla \hat \eta|^2 \int Q \left|\nabla  \mathcal M^{(1)} \left(\tilde{\Psi}^{(1)}_{b},\tilde{\Psi}^{(2)}_{b}\right) \right|^2  \right)^{\frac12}\\
&\lesssim& C(M) \frac{b}{\l^6}\frac{b^4}{|\log b|^2}
\eee
\subsection{${\bf G}$ terms}
We recall that ${\bf G} = -c_bb^2 \breve{\bf T}$ with
\bee
 \left| \begin{array} {l} \breve T^{(1)}\\ \nabla \breve T^{(2)}\end{array}\right. = \chi_{\frac{B_0}{4}} \left| \begin{array} {l} T_1 \\\nabla S_1\end{array}\right..
\eee
First, using that $\mathcal L {\bf T}_1 = \Lambda {\bf Q}$, the cancellation $\int \hat \e_2=0$ and $\mathcal M \Lambda {\bf Q} = \left| \begin{array} {l} -2 \\ 0\end{array}\right.$, we obtain :
\bee
\left<\widehat {\bf E}_2,\mathcal M\mathcal L\breve{\bf T}\right> = - \left<\widehat {\bf E}_2,\mathcal M\mathcal L[{\bf T}_1 - \breve{\bf T}]\right>
\eee
Moreover, 
\bee
\widehat {\bf E}_2 ={\bf E}_2 + \mathcal L {\boldsymbol \Xi}_{sm} + (b-\hb) \Lambda Q  + \boldsymbol{\mathcal R}. 
\eee
where $ \boldsymbol{\mathcal R}$ satisfies the bound \fref{mathcalR}. With this decomposition, we obtain:
\bee
\left<\widehat {\bf E}_2,\mathcal M\mathcal L\breve{\bf T}\right> =  - \left<{\bf E}_2 + \mathcal L {\boldsymbol \Xi}_{sm} + \boldsymbol{\mathcal R},\mathcal M\mathcal L[{\bf T}_1 - \breve{\bf T}]\right>
\eee
We compute now $\mathcal L[{\bf T}_1 - \breve{\bf T}]$:
\bee
\mathcal L[{\bf T}_1 - \breve{\bf T}] &=& \left| \begin{array}{l} \nabla. \{\nabla(T_1(1-\chi_{\frac{B_0}{4}})) + T_1\nabla S_1 (1-\chi_{\frac{B_0}{4}})^2 \} \\  \nabla. \{(\nabla S_1)(1-\chi_{\frac{B_0}{4}})\} - (1-\chi_{\frac{B_0}{4}})T_1 \end{array}\right.
\eee
Hence :
\bee
\left| \mathcal L^{(1)}[{\bf T}_1 - \breve{\bf T}]\right| \lesssim \frac{1}{1+\tau^4} {\bf 1}_{\tau \geq \frac{B_0}{4}}.
\eee
and
\bee
\nabla \mathcal M^{(2)}(\mathcal L[{\bf T}_1 - \breve{\bf T}]) &=&- (\chi_{\frac{B_0}{4}})'\phi_{\Lambda Q} + ((\chi_{\frac{B_0}{4}})'\nabla S_1)' - \chi_{\frac{B_0}{4}})' T_1 - T_1\nabla S_1\chi_{\frac{B_0}{4}}(1-\chi_{\frac{B_0}{4}}) \\
&-&(1-\chi_{\frac{B_0}{4}})\left\{ \nabla \phi_{\Lambda Q}-\nabla T_1 - T_1\nabla S_1\right\}
\eee
By construction of the profiles $T_1$ and $S_1$, we have
\bee
\nabla \phi_{\Lambda Q}-\nabla T_1 - T_1\nabla S_1 = 0.
\eee
Hence, we obtain the estimation
\bee
\left| \nabla \mathcal M^{(2)}(\mathcal L[{\bf T}_1 - \breve{\bf T}]) \right|\lesssim \frac{\sqrt b}{r} {\bf 1}_{\frac{B_0}{4} \leq \tau \leq \frac{B_0}{2}}.
\eee
Thus
\bee
&&\left<\widehat{\bf W}_2, \mathcal M_\l \mathcal L_\l {\bf G}_\l \right>\lesssim\frac{b^2}{\l^6|\log b|}\left| \left<{\bf E}_2 + \mathcal L {\boldsymbol \Xi}_{sm} + \mathcal R,\mathcal M\mathcal L[{\bf T}_1 - \breve{\bf T}]\right>\right|\\
&\lesssim& \frac{b^2}{\l^6|\log b|}\left( \|\e_2\|_{L^2_Q} +\left\| \mathcal L^{(1)} {\boldsymbol {\Xi}_{sm}} \right\|_{L^2_Q}+\left\| \mathcal R^{(1)} \right\|_{L^2_Q}  \right)\left\|\mathcal L^{(1)}[{\bf T}_1 - \breve{\bf T}]\right\|_{L^2_Q} \\
&+&\frac{b^2}{\l^6|\log b|}\left( \|\nabla \eta_2\|_{L^2}+\left\| \mathcal L^{(2)} {\boldsymbol {\Xi}_{sm}} \right\|_{L^2} +\left\| \mathcal R^{(2)}\right\|_{L^2}  \right) \left\|\nabla \mathcal M^{(2)}(\mathcal L[{\bf T}_1 - \breve{\bf T}]) \right\|_{L^2} \\
&\lesssim&\frac{b^{\frac52}}{\l^6|\log b|} \left( \|{\bf E}_2\|_{X_Q} +\|\mathcal L {\boldsymbol {\Xi}_{sm}} \|_{X_Q}+\|\mathcal R \|_{X_Q} \right) \lesssim\frac{b^{\frac52}}{\l^6|\log b|} \left( \|{\bf E}_2\|_{X_Q} +\frac{b^{\frac32}}{|\log b|}\right) .
\eee
Other terms depending on ${\bf G}$ in \fref{modifiedenergy} can be treated in brute force. Indeed, using \fref{coercbase} and Lemma \ref{lemmaradiation}:
\bee
b\left|\int\frac{\Lambda Q+2Q}{Q^2}(\eh\mathcal L^{(1)}{\bf G}+\eh_2 G^{(1)})\right| & \lesssim & b\frac{b^2}{|\log b|}\int \frac{1}{(1+r^2)Q}\left[\frac{|\eh|}{1+r^4}+\frac{|\eh_2|}{1+r^2}\right]\\
& \lesssim & b\frac{b^2}{|\log b|}\left[\|\e\|_{L^2}+\|\e_2\|_{L^2_Q}+\|\zeta_2\|_{L^2_Q}+\|\zeta\|_{L^2}\right]\\
& \lesssim & b\left[C(M)\frac{b^2}{|\log b|}\|{\bf E}_2\|_{X_Q}+\frac{b^3}{|\log b|^2}\right].
\eee
Similarily, using \fref{normeWQ2} and Lemma \ref{lemmaradiation}:
\bee
b\left|\int \eh_1\cdot\nabla G^{(2)}\right|& \lesssim & b\frac{b^2}{|\log b|}\int \frac{1+|\log r|}{1+r}\left[|\e_1|+|\nabla \zeta|+Q|\nabla \xi|+\frac{|\zeta|}{1+r}\right]\\
& \lesssim & b\frac{b^2}{|\log b|}\left(\int (1+r^2)|\e_1|^2+|\zeta|^2+(1+r^2)|\nabla \zeta|^2+\frac{|\nabla \xi|^2}{1+r^2}\right)^{\frac 12}\\
& \lesssim & \frac{b^3}{|\log b|}\left[C(M)\|{\bf E}_2\|_{X_Q}+\frac{b}{|\log b|}\right],
\eee
and using an integration by parts:
\bee
b\left|\int \nabla \hat \eta \cdot Q\nabla \mathcal M^{(1)}{\bf G}\right|& =&b\left|(Q\Delta \hat \eta+\nabla Q\cdot\nabla \hat \eta,\mathcal M E_2)\right|\\
& \lesssim & b\frac{b^2}{|\log b|}\int \left[\frac{|\Delta \hat \eta|}{1+r^4}+\frac{|\nabla \hat \eta|}{1+r^5}\right](1+r^2)\\
& \lesssim & b\left[\frac{b^2}{|\log b|}C(M)\|{\bf E}_2\|_{X_Q}+\frac{b^3}{|\log b |^2}\right].
\eee 
\subsection{{\bf H} terms :}
We recall that:
\bee
{\bf H} =  - \mathcal L {\bf \Xi} - \frac{\l_s}{\l} \Lambda {\bf \Xi}.
\eee
and the decomposition :
\bee
\mathcal L {\bf \Xi} = \mathcal L \boldsymbol \Xi_{sm} + (b-\hb)\Lambda {\bf Q} + \boldsymbol {\mathcal R}
\eee
where $ \boldsymbol \Xi_{sm} $ is defined in \fref{xism} and we have the estimation
\bee
\boldsymbol{\mathcal R} = (b-\hat b) \begin{array} {|c} O\left(\frac{{\bf 1}_{r \geq B_1}}{1+r^4}\right) \\
O\left(\frac{{\bf 1}_{r \geq B_1}}{1+r^2}\right) 
\end{array}.
\eee
In the following, we notice 
\be
\label{defh1}
{\bf H}_1 = {\bf H} + (b-\hat b)\Lambda {\bf Q} =  -\mathcal L \boldsymbol \Xi_{sm} - \boldsymbol {\mathcal R}- \frac{\l_s}{\l} \Lambda {\bf \Xi}.
\ee
Moreover from \fref{estlambda} and \fref{deformation} :
\bee
\left| \lsl \right| \lesssim b\ \ \mbox{and} \ \ \left| b-\hb \right| \lesssim \frac{b}{|\log b|}.
\eee
Using together \fref{degenrateoneter} and \fref{degenrateonebis} yields
\be
\label{borneb4surh}
\int Q \left| \nabla \mathcal M^{(1)} (H^{(1)},H^{(2)})\right|^2 + \int \left| \Delta \mathcal M^{(2)} (H^{(1)},H^{(2)}) \right|^2 \lesssim \frac{b^4}{|\log b|^2}.
\ee
The bounds \fref{estimatonebis}, \fref{estimatone} and \fref{estimatonebisbis} yields
\be
\label{borne4surh1}
\int \frac{\left| \nabla H^{(2)}_1 \right|^2}{1+\tau^2}  + \int \left| H^{(1)}_1\right|^2 \lesssim  \frac{b^4}{|\log b|^2}.
\ee
We are in position to estimate the ${\bf H}$ terms in \fref{modifiedenergy}:
\bee
\left| \left< \widehat{\bf{E}_2}, \mathcal M \mathcal L {\bf H}\right>\right| &=& \left| \left< \mathcal M \widehat{\bf{E}_2}, \mathcal L {\bf H}\right>\right| \\
&\lesssim& \left(\int Q\left| \nabla \mathcal M^{(1)}(\hat \e_2, \hat \eta_2)\right|^2\int Q \left| \nabla \mathcal M^{(1)} (H^{(1)}_1,H^{(2)}_1)\right|^2\right)^{\frac12} \\
&+& \left(\int \left| \Delta \mathcal M^{(2)}(\hat \e_2, \hat \eta_2)\right|^2  \int \left| \Delta \mathcal M^{(2)} (H^{(1)}_1,H^{(2)}_1) \right|^2\right) \\
&\lesssim& \frac {1}{100}\left[ \int Q\left| \nabla \mathcal M^{(1)}(\hat \e_2, \hat \eta_2)\right|^2 + \int \left| \Delta \mathcal M^{(2)}(\hat \e_2, \hat \eta_2)\right|^2\right] \\
&+&\int Q \left| \nabla \mathcal M^{(1)} (H^{(1)}_1,H^{(2)}_1)\right|^2 + \int \left| \Delta \mathcal M^{(2)} (H^{(1)}_1,H^{(2)}_1) \right|^2 \\
&\lesssim&  \frac {1}{100}\left[ \int Q\left| \nabla \mathcal M^{(1)}(\hat \e_2, \hat \eta_2)\right|^2 + \int \left| \Delta \mathcal M^{(2)}(\hat \e_2, \hat \eta_2)\right|^2\right] + \frac{b^4}{|\log b|^2}
\eee
In the same way, using \fref{coercbase}
\bee
b\left|\int Q \nabla \mathcal M^{(1)} (H^{(1)}_1,H^{(2)}_1) \nabla \hat \eta  \right| &\lesssim& \frac {1}{100} \int Q\left| \nabla \mathcal M^{(1)}(\hat \e_2, \hat \eta_2)\right|^2 + b^2\int Q |\nabla \hat \eta|^2 \\
 &\lesssim& \frac {1}{100} \int Q\left| \nabla \mathcal M^{(1)}(\hat \e_2, \hat \eta_2)\right|^2 +C(M)\frac{b^4}{|\log b|^2}
\eee
Finally, 
\bee
&&\left| b\int \frac{\Lambda Q + 2Q}{Q^2} \hat \e \mathcal L^{(1)}(H^{(1)}_1,H^{(2)}_1) \right| =\left| b\int \nabla \left\{\frac{\Lambda Q + 2Q}{Q^2} \hat \e\right\}Q\nabla \mathcal M^{(1)}(H^{(1)}_1,H^{(2)}_1) \right| \\
&\lesssim&b \left( \int Q \left| \nabla \mathcal M^{(1)} (H^{(1)}_1,H^{(2)}_1)\right|^2\right)^{\frac12}\left(\int |\nabla \eh|^2+\frac{|\eh|^2}{1+r^2}\right)^{\frac12}\\
& \lesssim & \frac{b^3}{|\log b|^2}\left[\frac{b}{|\log b|}+C(M)\|{\bf E}_2\|_{X_Q}\right].
\eee
For the two last terms, we must use the decomposition \fref{defh1}, in order to use the structure of $\Lambda Q$. Without this structure, the terms are critical size for our analysis and thus are unmanageable. We shall highlight an additionnal algebra in order to use the dissipation, to control this pathological term.
\par First, we are focusing on the ${\bf H}_1$ term, using the estimation \fref{borne4surh1}.
\bee
&&\left| b\int \hat \e_1 \nabla H^{(2)}_1 \right| \\
&\lesssim& b^2 \left( \int \frac{\left| \nabla H^{(2)}_1 \right|^2}{1+\tau^2}\right)^{\frac12} \left[ \int (1+\tau^2)|\e_1|^2 + \int (1+\tau^2)|\nabla \zeta|^2 + \int |\zeta|^2 + \int \frac{|\nabla \xi|^2}{1+\tau^2}\right]^{\frac12}\\
& \lesssim & \frac{b^3}{|\log b|^2}\left[\frac{b}{|\log b|}+C(M)\|{\bf E}_2\|_{X_Q}\right].
\eee
Moreover
\bee
&&\left| b\int \frac{\Lambda Q + 2Q}{Q^2} \hat \e_2 H^{(1)}_1 \right| \lesssim b \|\hat \e_2\|_{L^2_Q} \|H^{(1)}_1\|_{L^2} \lesssim\frac{b^3}{|\log b|^2}\left[\frac{b}{|\log b|}+C(M)\|{\bf E}_2\|_{X_Q}\right].
\eee
The last step is to study this both term for $(b-\hb)\Lambda Q$. Its crucial to study them together. Hence,
\bee
b(b-\hat{b})\left[\int \frac{\Lambda Q+2Q}{Q^2}\eh_2\Lambda Q-2\int \eh_1\cdot\nabla \phi_{\Lambda Q}\right]=  b(b-\hat{b})\int \eh_2\left[\frac{\Lambda Q(\Lambda Q+2Q)}{Q^2}+2\phi_{\Lambda Q}\right].
\eee
Here we have an additional algebra. Indeed:
\bee
\frac{\Lambda Q(\Lambda Q+2Q)}{Q^2}+2\phi_{\Lambda Q}
=\left(\frac{\Lambda Q}{Q}\right)^2 + \mathcal M^{(1)} (\Lambda Q, \phi_{\Lambda Q})
=\left(\frac{\Lambda Q}{Q}\right)^2-4.
\eee
As $\int \hat e_2=0$,
\bee
\int \eh_2\left[\frac{\Lambda Q(\Lambda Q+2Q)}{Q^2}+2\phi_{\Lambda Q}\right] = \int \left(\frac{\Lambda Q}{Q}\right)^2\eh_2 \\
 \eee
 We have still estimate this term in the subsection \ref{quadratic}, and we have proved
 \bee
 \left| \int \left(\frac{\Lambda Q}{Q}\right)^2\eh_2 \right| \leq  \frac {1}{100} \int Q\left| \nabla \mathcal M^{(1)}(\hat \e_2, \hat \eta_2)\right|^2 +C(M)\frac{b^4}{|\log b|^2}.
\eee
This concludes the estimations of ${\bf H}$ terms.
\subsection{Small linear $\boldsymbol \T_b$ terms}
We recall that:
\bee
 {\bf \Theta}_b(\e,\eta) &=&  \begin{array}{| c} \nabla. \left( \e \nabla \tilde{\gamma}_b + \tilde{\alpha}_b \nabla \eta \right) \\0 \end{array} =  \begin{array}{| c} \nabla \e \nabla \tilde{\gamma}_b + \e \Delta \tilde{\gamma}_b + \tilde{\alpha}_b \Delta \eta + \nabla \tilde{\alpha}_b \nabla \eta \\0 \end{array}.
 \eee
Hence
\bee
 \nabla \phi_{{\Theta}^{(1)}_b} =  \e \nabla \tilde{\gamma}_b + \tilde{\alpha}_b \nabla \eta.
\eee
As $\T_b^{(2)}=0$, we have :
\bee
\mathcal M {\bf \T}_b = \left| \begin{array}{l} \frac{\T_b^{(1)}}{Q}\\ - \phi_{{\Theta}^{(1)}_b}\end{array}\right.
\eee
In the proof of the lemma \ref{lemmasharpmod}, we have obtain the rough bounds :
\bee
\left|  \nabla \phi_{{\Theta}^{(1)}_b} \right| &\lesssim&  b \left[ \frac{1}{1+r} |\e| + \frac{|\nabla \eta|}{1+r^2} \right], \\
\left| \T^{(1)}_b \right| &\lesssim& b \left[ \frac{1+ |\log r|}{1+r} |\nabla \e| + \frac{1+\log r}{1+r^2} |\e| + \frac{|\Delta \eta|}{1+r^2}+ \frac{|\nabla \eta|}{1+r^3} \right], \\
\left|\nabla \T^{(1)}_b \right| &\lesssim& b \left[ \frac{1+ |\log r|}{1+r} |\nabla^2 \e| + \frac{1+\log r}{1+r^2} |\nabla \e| + \frac{1+\log r}{1+r^3} |\e| + \frac{|\nabla \Delta \eta|}{1+r^2}+ \frac{|\Delta \eta|}{1+r^3}+ \frac{|\nabla \eta|}{1+r^4} \right].
\eee
Hence, we obtain the bounds:
\be
\label{borneb4surtheta}
\int Q \left| \nabla \mathcal M^{(1)} (\T_b^{(1)},\T_b^{(2)})\right|^2 + \int \left| \Delta \mathcal M^{(2)} (\T_b^{(1)},\T_b^{(2)}) \right|^2 \lesssim C(M)b^2 \left\| {\bf E}_2 \right\|^2_{X_Q}
\ee
and
\be
\label{borne4surtheta1}
\int \frac{| \nabla \T_b^{(2)} |^2}{1+\tau^2}  + \int \left| \T_b^{(1)}\right|^2 \lesssim C(M)b^2 \left\| {\bf E}_2 \right\|^2_{X_Q}.
\ee
From the last subsection, these bounds are enough small with respect to $b$ to ensure the control of ${\bf \T}_b$ terms. 
\subsection{Modulation terms}
We recall that :
\bee
\widehat{\bf Mod} =  \left( b + \frac{\l_s}{\l}\right) \Lambda \tilde {\bf Q}_{b} - \pa_s \tilde {\bf Q}_{\hat b}
\eee
From the lemma \ref{lemmeparam} and \ref{lemmasharpmod}, we have the following bound:
\be
\label{superborne}
\left| \lsl + b \right| + |\hat b_s|\lesssim C(M) \frac{b^2}{|\log b|}.
\ee
We split the modulation terms in three parts in order to use the structure of the elements of the kernel of the linearized operator $\mathcal L$:
\bee
\widehat{\bf Mod} =  \widehat{\bf Mod_1}+\widehat{\bf Mod_2}+\widehat{\bf Mod_3}
\eee
with
\bee
\widehat{\bf Mod_1} &=& \left( b + \frac{\l_s}{\l}\right) \Lambda \tilde {\bf Q}_{b}, \\
\widehat{\bf Mod_2} &=& -\hat b_s {\bf T}_1, \\
\widehat{\bf Mod_3} &=& -\hb_s\left(\tilde{\bf T}_1-{\bf T}_1+\hb\frac{\partial \tilde{\bf T}_1}{\partial b}+2\hb\tilde{\bf T}_2+\hb^2\frac{\partial \tilde{\bf T}_2}{\partial b}\right).
\eee
From the Proposition \ref{localization}, and \fref{superborne} we have :
\bee
\left|\widehat{Mod_1}^{(1)} \right| &\lesssim&C(M) \frac{b^2}{|\log b|}\left[\frac{1}{1+\tau^4} + b^2 (1+|\log \tau|) {\bf 1}_{\tau \leq 2B_1}\right], \\
\left| \nabla \widehat{Mod_1}^{(2)} \right| &\lesssim& C(M) \frac{b^2}{|\log b|}\left[\frac{1}{1+\tau^3} +  b^2 \tau(1+|\log \tau|) {\bf 1}_{\tau \leq 2B_1}\right], \\
\left|\widehat{Mod_2}^{(1)} \right| &\lesssim& C(M) \frac{b^2}{|\log b|}\left[\frac{1}{1+\tau^2}\right], \\
\left| \nabla \widehat{Mod_2}^{(2)} \right| &\lesssim& C(M) \frac{b^2}{|\log b|}\left[\frac{\tau}{1+\tau^2}\right], \\
\left|\widehat{Mod_3}^{(1)} \right| &\lesssim& C(M) \frac{b^2}{|\log b|}\frac{\bf{1}_{\tau \geq B_1}}{1+\tau^2}, \\
&+& C(M) \frac{b^3}{|\log b|} \left\{  \tau^2{\bf 1}_{\tau\le1} +\frac{1 + |\log (\tau\sqrt b)|}{|\log b|}{\bf 1}_{1\leq \tau\leq 6B_0}+ \frac{1}{b^2\tau^4|\log b|}{\bf 1}_{\tau\geq 6B_0}\right\}, \\
\left| \nabla \widehat{Mod_3}^{(2)} \right| &\lesssim& C(M) \frac{b^2}{|\log b|}\left[{\bf 1}_{\tau \geq B_1}\frac{\tau}{1+\tau^2} +b (1+|\log \tau|) {\bf 1}_{\tau \leq 2B_1}\right] 
\eee
With the above estimations, it is easy to prove :
\bea
\label{boundmod2}
\left\| \frac{ \nabla \widehat{Mod_2}^{(2)}}{1+\tau^2}\right\|_{L^2}^2+ \left \|\mathcal L \widehat{\bf Mod_2}\right\|^2_{X_Q}  + \left\|\tau^i\pa_\tau^i \widehat{Mod_2}^{(1)}\right\|_{L^2}^2 &\lesssim& C(M)\frac{b^4}{|\log b|^2}, \\
\label{boundmod3}
\left\| \frac{ \nabla \widehat{Mod_3}^{(2)}}{1+\tau^2}\right\|_{L^2}^2+ \left \|\mathcal L \widehat{\bf Mod_3}\right\|^2_{X_Q}  + \left\|\tau^i\pa_\tau^i \widehat{Mod_3}^{(1)}\right\|_{L^2}^2 &\lesssim& C(M)\frac{b^5}{|\log b|^2},\\
\label{boundmod1}
\left\| \frac{ \nabla \widehat{Mod_1}^{(2)}}{1+\tau^2}\right\|_{L^2}^2+ \left\|\tau^i\pa_\tau^i \widehat{Mod_1}^{(1)}\right\|_{L^2}^2 &\lesssim& C(M)\frac{b^4}{|\log b|^2}, \\
\label{boundmod1bis}
 \left \|\mathcal L \widehat{\bf Mod_1}\right\|^2_{X_Q}  &\lesssim& C(M)\frac{b^5}{|\log b|^2}.
\eea
Remark that we have used for the last bound the cancellation $\mathcal L \Lambda {\bf Q}=0$.
 \par
 Now, we can verify that all modulation terms are manageable in the modified energy \fref{modifiedenergy}. From \fref{degenrateoneone},  \fref{boundmod3}, \fref{boundmod1bis}:
 \bee
\left|\left<\widehat{\bf E}_2,\mathcal M\mathcal L \widehat{\bf Mod}\right>\right|& \lesssim &\left|\left<\mathcal M{\boldsymbol \Xi}_2,\mathcal L \widehat{\bf Mod}\right>\right|+\|{\bf E}_2\|_{X_Q}\left[\left \|\mathcal L \widehat{\bf Mod_2}\right\|_{X_Q}+\left \|\mathcal L \widehat{\bf Mod_3}\right\|_{X_Q}\right]\\
& \lesssim &\left(\|{\bf E}_2\|_{X_Q} + \frac{b^{\frac32}}{|\log b|}\right)\left[\left \|\mathcal L \widehat{\bf Mod_2}\right\|_{X_Q}+\left \|\mathcal L \widehat{\bf Mod_3}\right\|_{X_Q}\right]\\
& \lesssim &C(M) b\left[\frac{b^{\frac 32}}{|\log b|}\|{\bf E}_2\|_{X_Q} +\frac{b^3}{|\log b|^2}\right].
\eee
Moreover, from the bounds \fref{estimatonebis}, \fref{estimatone} and the interpolation bound \fref{coercbase} :
\bee
&&b \left| \int \frac{\Lambda Q + 2Q}{Q^2} \left[ \hat \e \mathcal L^{(1)}\left(\widehat{Mod}^{(1)},\widehat{Mod}^{(2)} \right) + \hat \e_2 \widehat{Mod}^{(1)}\right] \right|\\
&\lesssim& b \left[ \|\e\|_{L^2}\left\|\mathcal L \widehat{\bf Mod} \right\|_{X_Q} + \left\{\left\| {\bf E}_2 \right\|_{X_Q} +\left\| {\boldsymbol \Xi}_2 \right\|_{X_Q}  \right\}\left\| \widehat{\bf Mod} \right\|_{L^2} \right]\\
& \lesssim &C(M) b\left[\frac{b^{\frac 32}}{|\log b|}\|{\bf E}_2\|_{X_Q} +\frac{b^3}{|\log b|^2}\right].
\eee
Similarly,
\bee
b\left| \hat \e_1\nabla \widehat{Mod}^{(2)}  \right| &\lesssim& b \left\| \frac{ \nabla \widehat{Mod}^{(2)}}{1+\tau^2}\right\|_{L^2} \left[ \int (1+\tau^2)|\e_1|^2+ \int (1+\tau^2)|\nabla \zeta|^2 + \int |\zeta|^2 + \int \frac{|\nabla \xi|^2}{1+\tau^2}\right] \\
& \lesssim &C(M) \frac{b^{3}}{|\log b|}\left[\|{\bf E}_2\|_{X_Q} +\frac{b}{|\log b|}\right].
\eee
Finally, using the bootstrap bound \fref{bootsmallh1} and \fref{bootsmallh2q}
\bee
&&b\left|\int \nabla\hat \eta \cdot Q\nabla\mathcal M^{(1)}\left(\widehat{Mod}^{(1)},\widehat{Mod}^{(2)} \right) \right| \lesssim b\int|\nabla \hat \eta|\left||\nabla\widehat{Mod}^{(1)}|+\frac{|\widehat{Mod}^{(1)}|}{1+\tau}+\frac{|\nabla\widehat{\Mod}^{(2)}|}{1+\tau^4}\right|\\
& \lesssim & b\left(\int \frac{|\nabla \eta|^2}{1+\tau^2}+\frac{|\nabla \xi|^2}{1+\tau^2}\right)^{\frac 12}\left(\int(1+\tau^2)|\nabla \widehat{Mod}^{(1)}|^2+|\widehat{Mod}^{(1)}|^2+\frac{|\nabla \widehat{\Mod^{(2)}}|^2}{1+\tau^2}\right)^{\frac 12}\\
& \lesssim & C(M) \frac{b^{3}}{|\log b|}\left(\|\eta\|_{\dot H^1}\|\eta_2\|_{\dot H^1} +\frac{b^2}{|\log b|^2}\right)^{\frac 12}\\
& \lesssim & C(M) \frac{b^{3}}{|\log b|}\left(K^*b^{\frac52}|\log b|^2+\frac{b^2}{|\log b|^2}\right)^{\frac 12}\\
& \lesssim &C(M) \frac{b^4}{|\log b|^2}
 \eee
\subsection{Non-linear ${\bf N}$ terms}
We recall that
\bee
 {\bf N}(\e,\eta) &=&  \begin{array}{| c} \nabla. \left( \e  \nabla \eta \right) \\0 \end{array} \ \ \mbox{and} \ \ \nabla \phi_{N^{(1)}}(\e,\eta) = \e  \nabla \eta. 
\eee
To control the non-linear $\bf N$ terms, we shall prove the following bounds :
\be
\label{borneb4surn}
\int Q \left| \nabla \mathcal M^{(1)} (N^{(1)},N^{(2)})\right|^2 + \int \left| \Delta \mathcal M^{(2)} (N^{(1)},N^{(2)}) \right|^2 \lesssim C(M)\frac{b^4}{|\log b|^2}
\ee
and
\be
\label{borneb4surn1}
\int \frac{| \nabla N^{(2)} |^2}{1+\tau^2}  + \int \left| N^{(1)}\right|^2 \lesssim C(M) \frac{b^4}{|\log b|^2}
\ee
We have still proved that it is a sufficient condition to be sure that these terms are manageable for our analysis. To prove \fref{borneb4surn} and \fref{borneb4surn1}, we compute:
\bee
\left\{\begin{array}{l}\nabla \mathcal M^{(1)} (N^{(1)},N^{(2)})\\\nabla \mathcal M^{(2)} (N^{(1)},N^{(2)}) \end{array}\right. =\left\{\begin{array}{l}\nabla \left( \frac{N^{(1)}(\e,\eta)}{Q}\right)\\ -\nabla \phi_{N^{(1)}}(\e,\eta) \end{array}\right.  =\left\{\begin{array}{l}\nabla \left( \frac{\nabla.(\e\nabla \eta)}{Q}\right)\\ -\e \nabla \eta \end{array}\right.  
\eee
Hence
\bee
&&\int Q \left| \nabla \mathcal M^{(1)} (N^{(1)},N^{(2)})\right|^2 + \int \left| \Delta \mathcal M^{(2)} (N^{(1)},N^{(2)}) \right|^2 \\
&=&  \int \left| N^{(1)}\right|^2 + \int \frac{\left[ Q \nabla \phi_{N^{(1)}}(\e,\eta) + \nabla \phi_Q N^{(1)}(\e,\eta) + \nabla\e(\Delta \eta + \nabla^2 \eta) + \e \nabla \Delta \eta \right]^2}{Q}\\
&\lesssim& \int \frac{\e^2|\nabla \eta|^2}{1+r^4} + \int (1+r^2)|N^{(1)}(\e,\eta)|^2 + \int (1+r^4)\left[\e^2| \nabla \Delta \eta|^2+  |\nabla \e|^2\left(|\Delta \eta|^2 + |\nabla^2\eta|^2\right) \right]
\eee
We shall estimate separately each term. First, using \fref{normeinftynablaeta}:
\bee
\left|  \int \frac{\e^2|\nabla \eta|^2}{1+r^4}\right| \lesssim \|\nabla \eta \|^2_{L^{\infty}}\int \e^2&\lesssim& K^* b^2|\log b|^6 \|{\bf E}_2\|^2_{X_Q}\\
&\lesssim& C(M) \frac{b^4}{|\log b|^2}.
\eee
Now, using the definition of ${\bf N}$, and the bounds \fref{normeinftynablaeta} and \fref{normeinftyye}
\bee
\left| \int(1+r^2)|N^{(1)}(\e,\eta)|^2 \right| &\lesssim& \int (1+r^2)|\nabla\e|^2|\nabla \eta|^2 + \int (1+r^2)|\e|^2|\Delta \eta|^2 \\
&\lesssim& \|\nabla \eta \|^2_{L^{\infty}} \int (1+r^2)|\nabla\e|^2 +  \|(1+r) \e \|^2_{L^{\infty}} \int |\Delta\eta|^2 \\
&\lesssim& K^* b^2|\log b|^6 \|{\bf E}_2\|^2_{X_Q}\\
&\lesssim& C(M) \frac{b^4}{|\log b|^2}.
\eee
With the bootstrap bound \fref{bootsmallh3} and \fref{normeinftyye}, we obtain
\bee
\left| \int (1+r^4) \e^2| \nabla \Delta \eta|^2 \right| \lesssim \|(1+r) \e \|^2_{L^{\infty}} \int (1+r^2)| \nabla \Delta \eta|^2 &\lesssim&  K^* \frac{b}{\sqrt{|\log b|}} \|{\bf E}_2\|^2_{X_Q}\\
&\lesssim& C(M) \frac{b^4}{|\log b|^2}.
\eee
Finally, the bound \fref{normeinftydeltaeta} give :
\bee
\left| \int (1+r^4) |\nabla \e|^2\left(|\Delta \eta|^2 + |\nabla^2\eta|^2\right) \right| &\lesssim& \left[\|(1+r) \Delta \eta \|^2_{L^{\infty}} + \|(1+r) \nabla^2 \eta \|^2_{L^{\infty}} \right] \int (1+r^2)| \nabla \e|^2 \\
&\lesssim&  K^* \frac{b}{\sqrt{|\log b|}} \|{\bf E}_2\|^2_{X_Q}\\
&\lesssim& C(M) \frac{b^4}{|\log b|^2}.
\eee
Chosen the function $\delta$ in \fref{positivitybzero} enough small with respect to $K^*$, the above estimate concludes the proof of \fref{borneb4surn} and \fref{borneb4surn1}, and thus the proof of all terms depending on $\boldsymbol {\mathcal F}$.
\subsection{Conclusion}
Injecting all above estimates in \fref{modifiedenergy} yields :
\bee
\nonumber &&\frac 12\frac{d}{dt}\left\{\frac{\left<\mathcal M \widehat{\bf E}_2,\widehat{\bf E}_2\right>}{\l^4}+O\left(\frac{C(M)b^3}{\l^4|\log b|^2}\right)\right\}\\
 &\lesssim&  \frac{bC(M)}{\l^6}\left[\frac{b^{\frac32}}{|\log b|}\|{\bf E}_2\|_{X_Q}+\frac{b^3}{|\log b|^2}\right]+\frac{\hb}{\l^6} \int \hat \e_2 \mathcal M^{(1)}_\l (\hat \e_2, \hat \eta_2).
\eee
We are now in position to treat the last term. In this purpose, we shall multiply the above inequality by $\l^2$, and we use the following bound coming form \fref{estlambda} and \fref{deformation} :
\bee
\left|\frac{\l_s}{\l} + \hat b \right| \lesssim \frac{b}{|\log b|},
\eee
in order to obtain
\bea
\nonumber &&\frac 12\frac{d}{dt}\left\{\frac{\left<\mathcal M \widehat{\bf E}_2,\widehat{\bf E}_2\right>}{\l^2}+O\left(\frac{C(M)b^3}{\l^4|\log b|^2}\right)\right\} \lesssim \frac{bC(M)}{\l^4}\left[\frac{b^{\frac32}}{|\log b|}\|\e_2\|_{L^2_Q}+\frac{b^3}{|\log b|^2}\right] \\
\label{bonbon3}&+&\frac{1}{\l^4}\left|\frac{\l_s}{\l} + \hat b \right|  \int \hat \e_2 \mathcal M^{(1)}_\l (\hat \e_2, \hat \eta_2)+\frac{1}{\l^4}\left|\lsl\right|\frac{C(M)b^3}{|\log b|^2}.
\eea
Using the bounds  \fref{degenrateoneone} and \fref{noenoeno}, we obtain:
\bea
\nonumber\left|\left<\mathcal M \widehat{\bf E}_2,\widehat{\bf E}_2\right>\right| &=& \left|\left<\mathcal M{\bf E}_2,{\bf E}_2\right> +2 \left<\mathcal M \boldsymbol{\Xi}_2,{\bf E}_2\right> +  \left<\mathcal M \boldsymbol{\Xi}_2,\boldsymbol{\Xi}_2\right> \right|\\
\label{bonbon}&\lesssim& \left|\left<\mathcal M{\bf E}_2,{\bf E}_2\right>\right| + \frac{b^{\frac32}}{|\log b|}\|{\bf E}_2\|_{X_Q} + \frac{b^3}{|\log b|^2}.
\eea
In the same way, the bounds  \fref{degenrateoneone2} and \fref{noenoeno2} yield:
\bea
\nonumber \int \hat \e_2 \mathcal M^{(1)}_\l (\hat \e_2, \hat \eta_2) &=& \left|\int \e_2 \mathcal M^{(1)}_\l (\e_2, \eta_2)+2\int \e_2 \mathcal M^{(1)}_\l (\zeta_2, \xi_2)+\int \zeta_2 \mathcal M^{(1)}_\l (\zeta_2, \xi_2) \right| \\
\label{bonbon1}&\lesssim& \left|\int \e_2 \mathcal M^{(1)}_\l (\e_2, \eta_2)\right| + \frac{b^{\frac32}}{|\log b|}\|\e_2\|_{L^2_Q} + \frac{b^3}{|\log b|^2}.
\eea
Using the definition of the operator $\mathcal M$ and the interpolation bound \fref{coercbase}, we have:
\bea \nonumber
 \left|\int \e_2 \mathcal M^{(1)}_\l (\e_2, \eta_2)\right| &=& \left|\int \frac{|\e_2|^2}{Q} + \int \e_2\eta_2\right| \\
 \nonumber &\lesssim&  \int \frac{|\e_2|^2}{Q} + \left( \int \frac{|\e_2|^2}{Q}\right)^{\frac12} \left( \int Q|\Delta \eta|^2 + \int Q |\e|^2\right)^{\frac12} \\
\label{bonbon2} &\lesssim& \|{\bf E}_2\|_{X_Q}^2.
\eea
Injecting the bounds \fref{bonbon}, \fref{bonbon1} and \fref{bonbon2} in \fref{bonbon3} yield \fref{eetfonafmetea}. This conclude the proof of the Proposition \ref{htwoqmonton}
\section{Proof of the Proposition \ref{propboot}}
\label{superbootstrap}
The bound \fref{poitnzeiboud} implies $\hb_s <0$. Using together \fref{deformation} prove the upper bound of \fref{poitnzeiboud}. We prove the non-cancellation of $b$ by contradiction. Suppose that there exist a time $s^*<T$, where $T$ is the maximal time, where the bounds of the bootstrap hold true, such that $b(s^*)=0$. The bound \fref{bootsmallh2q} and the interpolation bound \fref{coercbase} imply that $\e(s^*) = 0$. Hence, by conservation of the mass, $\int u(s^*) = \int Q = \int u_0$. This is a contradiction with the initial small super critical mass \fref{intiialmass}. This concludes the proof of \fref{poitnzeiboud}.
\subsection{$L^1$ bound \fref{bootsmallloneboot}}
\label{proofpropboot}
\begin{lemma}[$L^1$  bound]
\label{htwoqmontonbis}
There holds:
\be
\label{lonebound}
\int|\e|<\frac 12(\delta^*)^{\frac 14}.
\ee
\end{lemma}
\begin{proof}
We introduce the decomposition $$u=\frac{1}{\l^2}(Q+\et)\left(t,\frac{x}{\l(t)}\right)\ \ \mbox{i.e.}\ \ \et=\e+\qbt-Q$$ and we split the perturbation $\et$ in two parts:
\be
\label{decompfoejejo}
\et=\et_<+\et_>, \ \ \et_{<}=\et{\bf 1}_{|\et|<Q},\ \  \et_>=\et{\bf 1}_{\et>Q}.
\ee 
Using the interpolation bound \fref{normeinftyye}, the bootstrap bound \fref{positiv} and the bounds \fref{esttonepropt}, \fref{esttwot} coming from the construction of the approximate profile, we obtain: 
\be
\label{etlinftiysmall}
\|\et\|_{L^{\infty}}\lesssim \|\e\|_{L^{\infty}}+|b|\lesssim \delta(\alpha^*).
\ee 
Choose a small constant $\eta^*< \delta(\alpha^*)$. Let r such that $|\et_<(r)|>\eta^*Q(r)$. Then $$\delta(\alpha^*)\gtrsim\|\et\|_{L^{\infty}}>|\et_<(r)|>\eta^* Q(r).$$ Thus there exist $r(\alpha^*)\to +\infty$ as $\alpha^*\to 0$, such that $r(\alpha^*)<r$. 
\\
Hence
 \bee
 \int |\et_<| & \leq & \int  |\et_<|{\bf 1}_{\eta^*Q<|\et|< Q}+\int |\et_<|{\bf 1}_{|\et_<|\leq \eta^*Q}\lesssim \int_{r\geq r(\alpha^*)}Q+\eta^*\int Q\\
& \lesssim& \delta(\alpha^*)+\eta^*\lesssim \delta(\alpha^*).
\eee
 According to the conservation of the mass, the bound \fref{intiialmass} implies $$\int \et<\alpha^*.$$
Now, using that by definition $\et_>>0$: 
 \be
 \label{cneiocneneo}
 \int |\et| \lesssim \delta(\alpha^*).
 \ee
 This bound together with the smallness of $b$ \fref{positiv} and the decay of the profiles $\tilde T_1$ and $\tilde T_2$ \fref{esttonepropt}, \fref{esttwot} imply:
 $$\int |\e|\lesssim \int |\et|+\sqrt{b}\lesssim \delta(\alpha^*)+\sqrt{\delta^*}<\frac 12(\delta^*)^{\frac 14}
$$ 
and concludes the proof of Lemma \ref{htwoqmontonbis}.
\end{proof}
\subsection{$\dot H^1$ bound \fref{bootsmallh1boot}} 
We prove this bound with the following monotonicity formula :
\be
\label{monoh1}
\frac d {dt} \left\{ \frac{1}{\l^2}\int |\nabla \eta|^2  + O\left( \frac{b^2|\log b|^2}{\l^2}\right)\right\} \leq \sqrt{K^*} \frac{b^3 |\log b|^6}{\l^4}
\ee
Assume \fref{monoh1}. Notice $\mathcal E_1 = \int |\nabla \eta|^2$. We recall the bounds coming from the modulation equations \fref{estlambda} and \fref{poitnzeiboud}:
\be
\label{supersuper}
|\hat b_s| \lesssim \frac{b^2}{|\log b|}, \ \ |\l\l_t + b| \lesssim C(M) \frac{b^2}{|\log b|}.
\ee
Integrating \fref{monoh1} in time between 0 ant $t^*$ yields:
\bea
\label{poivre}
&&\mathcal E_1(t) \leq \frac{\l^2(t)}{\l^2(0)} \left[ \mathcal E_1(0) + O \left( b(0)^2|\log b(0)|^2\right) \right] \\
\nonumber&+&\sqrt{K^*} \l^2(t) \int_0^{t^*} \frac{b^3 |\log b|^6}{\l^4}dt + O \left( b(t)^2|\log b(t)|^2\right).
\eea
With the bounds \fref{supersuper}, we estimate the last term :
\bee
\int_0^{t^*} \frac{b^3 |\log b|^6}{\l^4} &=&\int_0^{t^*} -\l_t \frac{\hat b^2 |\log \hat b|^6}{\l^3}dt + O\left( \int_0^{t^*}   \frac{b^4 |\log b|^2}{\l^4} \right) \\
&=& \left[ \frac{\hat b^2 |\log \hb|^6}{2\l^2}\right]_0^{t^*} - \int_0^t \frac{1}{2\l^4}\frac{d}{ds}\left[ \hb^2|\log \hb|^6\right]dt+ O\left( \int_0^{t^*}   \frac{b^4 |\log b|^2}{\l^4} \right) \\
&=& \left[ \frac{\hat b^2 |\log \hb|^6}{2\l^2}\right]_0^{t^*} + O\left( \int_0^{t^*} \frac{b^3 |\log b|^4}{\l^4} \right) \\ 
\eee
Thus
\be
\l^2(t) \int_0^{t^*} \frac{b^3 |\log b|^6}{\l^4} \lesssim \hat b(t)^2 |\log \hb(t)|^6 + \left(\frac{\l(t)}{\l(0)}\right)^2\hat b(0)^2 |\log \hb(0)|^6
\ee
Injecting this bound in \fref{poivre}, using the smallness of $\mathcal E_1(0)$, we obtain :
\bee
\mathcal E_1(t) \leq \sqrt{K^*} \left\{ \hat b(t)^2 |\log \hb(t)|^6 + \left(\frac{\l(t)}{\l(0)}\right)^2\hat b(0)^2 |\log \hb(0)|^6 \right\}.
\eee
Using the bounds \fref{supersuper}, we prove without difficulty that :
\be
\frac{d}{ds}\left[ \frac{\hb^2|\log \hb|^6}{\l^2}\right] >0.
\ee
This fact coupled with the measure \fref{deformation} of the gap betwenn $b$ and $\hb$ conclude the proof of \fref{bootsmallh1boot}.
\\ {\it Proof of \fref{monoh1} :} 
We use the second decomposition and as the norm $\dot H^1$ of the flux is invariant by scaling, we have the relation :
\bee
\int |\nabla \hat \eta|^2 = \int |\nabla  \hat z|^2.
\eee 
Moreover, $\hat z$ verifies the equation :
\bee
\pa_t \hat z = \Delta \hat z - \hat w + \frac 1 {\l^2} \left[ \Psi_b^{(2)} + \widehat \Mod^{(2)} \right]
\eee
Hence
\bea
\nonumber \frac d{dt} \frac12 \int |\nabla \hat z|^2 &=& - \int \pa_t \hat z \Delta \hat z \\
\nonumber &=& - \int |\Delta \hat z|^2 + \int \Delta \hat z \hat w - \frac 1 {\l^2} \int \Delta\hat z \left[  \tilde \Psi_b^{(2)} + \widehat \Mod^{(2)} \right] \\
\label{mlkjh0} &=& - \int |\Delta \hat z|^2 + \int \Delta\hat  z \hat w + \frac 1 {\l^2} \int \nabla\hat \eta \nabla \left[ \tilde \Psi_b^{(2)} + \widehat \Mod^{(2)} \right]. 
\eea
Now, using the bootstrap bound \fref{bootsmallh2q} and the interpolation bound \fref{coercbase} : 
\be
\label{mlkjh1}
\int \Delta\hat z\hat w \leq \frac 18 \int |\Delta\hat z|^2 + 4 \int\hat w^2 \leq  \frac 18 \int |\Delta\hat z|^2 + \frac{K^*}{\l^2} \frac{b^3}{|\log b|^2}.
\ee
Using Cauchy-Schwartz together the bounds \fref{bootsmallh1} and \fref{roughboundltaowt3} yields
\be
\label{mlkjh2}
\left| \int \nabla\hat \eta \nabla  \tilde \Psi_b^{(2)} \right| \lesssim \left( \int | \nabla\hat \eta|^2  \int | \nabla  \tilde \Psi_b^{(2)}|^2\right)^{\frac12} \lesssim \sqrt{K^*} b^3 |\log b|^{6}
\ee
The bootstrap bound on the modulation parameters \fref{cnkonenoeoeoi3poi} and the decay of $\nabla \tilde S_1$ \fref{nablaS1taillet} and $\nabla \tilde S_2$ \fref{nablaS2part} yields :
\bee
\int | \nabla\widehat \Mod^{(2)}|^2 &\lesssim& \left| b + \frac{\l_s}{\l}\right|^2 \int |\nabla \phi_{\Lambda Q} + b \nabla \tilde S_1 + b^2 \nabla \tilde S_2 |^2 + |\hat b_s|^2 \int |\nabla \tilde S_1 + 2b \nabla \tilde S_2 |^2 \\
&\lesssim& \frac{b^4}{|\log b|}.
\eee
Hence,
\be
\label{mlkjh3}
\left| \int \nabla \hat \eta\nabla\widehat \Mod^{(2)} \right| \lesssim \left( \int | \nabla \hat \eta|^2  \int | \nabla\widehat \Mod^{(2)}|^2\right)^{\frac12} \lesssim\sqrt{K^*} b^3 |\log b|^5
\ee
Injecting \fref{mlkjh1}, \fref{mlkjh2} and \fref{mlkjh3} in \fref{mlkjh0} 
\bee
\frac d {dt} \left\{\int |\nabla \hat \eta|^2 \right\} \leq \sqrt{K^*} \frac{b^3 |\log b|^6}{\l^2},
\eee
Now the bound \fref{vivementlafin} yields 
\bee
\frac d {dt} \left\{\int |\nabla \eta|^2 + O\left( b^2|\log|^2 \right) \right\} \leq \sqrt{K^*} \frac{b^3 |\log b|^6}{\l^2},
\eee
Dividing this inequality by $\l^2$, and using the bound \fref{supersuper}, we obtain \fref{monoh1}.
\subsection{$\dot H^2$ bound}
To prove this bound \fref{bootsmallH2}, we can use the second decomposition. Hence, we have the equation:
\bee
\pa_t  \hat z = \Delta \hat  z - \hat  w + \frac{1}{\l^2} \left[ \tilde \Psi_b^{(2)} +  \widehat{Mod}^{(2)} \right]
\eee
and
\bee
\pa_t \nabla  \hat z = \nabla \Delta  \hat z - \nabla  \hat w + \frac{1}{\l^3} \left[ \nabla \tilde \Psi_b^{(2)} + \nabla  \widehat{Mod}^{(2)} \right]
\eee
Next, we compute
\bee
\frac 12 \frac d {dt} \int |\Delta  \hat z|^2 &=& \int \nabla. \left\{ \nabla \pa_t \hat  z\right\} \Delta \hat  z \\
&=& - \int \nabla \pa_t \hat  z \nabla \Delta \hat  z \\
&=& - \int \left\{ \nabla \Delta \hat  z - \nabla \hat  w + \frac{1}{\l^3} \left[ \nabla \tilde \Psi_b^{(2)} + \nabla  \widehat{Mod}^{(2)} \right]\right\} \nabla \Delta \hat  z \\
&=& - \int \left| \nabla \Delta  \hat z \right|^2 + \int  \nabla \hat  w \nabla \Delta  \hat z + \frac{1}{\l^3} \int \left[ \nabla \tilde \Psi_b^{(2)} + \nabla  \widehat{Mod}^{(2)} \right] \nabla \Delta  \hat \eta
 \eee
Now, using the bootstrap bound \fref{bootsmallh2q} and the interpolation bound \fref{coercbase}, we obtain:
\bee
\left|\int  \nabla  \hat \e \nabla \Delta  \hat \eta \right| \leq \frac{1}{10} \int \left| \nabla \Delta  \hat \eta \right|^2 + \int |\nabla \hat \e|^2 + \int |\nabla \zeta|^2 \lesssim K^*\frac{b^3}{|\log b|^2} + \frac{1}{10} \int \left| \nabla \Delta  \hat \eta \right|^2.
\eee
Now, from the bound of the Proposition \ref{localization},
\bee
\int \left|\Delta  \tilde \Psi_b^{(2)}\right|^2 \lesssim b^5 |\log b|^2
\eee 
and from the bound \fref{estlambda} and \fref{poitnzeiboud}
\bee
\int |\Delta  \widehat{Mod}^{(2)} |^2 &\lesssim& \left| b + \frac{\l_s}{\l}\right|^2 \int |\Delta \phi_{\Lambda Q} + b \Delta \tilde S_1 + b^2 \Delta \tilde S_2 |^2 + |\hat b_s|^2 \int |\Delta \tilde S_1 + 2b \Delta \tilde S_2 |^2 \\
&\lesssim& \frac{b^4}{|\log b|^2}.
\eee
Hence
\bee
\left|  \frac{1}{\l^4} \int \left[ \nabla \tilde \Psi_b^{(2)} + \nabla  \widehat{Mod}^{(2)} \right] \nabla \Delta  \hat \eta \right| &=& \left|  \frac{1}{\l^4} \int \left[ \Delta \tilde \Psi_b^{(2)} + \Delta  \widehat{Mod}^{(2)} \right] \Delta  \hat \eta \right| \\
&\lesssim&\frac{1}{\l^4} \frac{b^2}{|\log b|} \|\Delta \hat \eta\|_{L^2} \lesssim \frac{\sqrt{K^*}}{\l^4} \frac {b^3} {|\log b|}.
\eee
The above estimations together \fref{estimatonebis}  and \fref{estimatone} give the monotonicity formula:
\be
\frac{d}{dt}\left[\frac{1}{\l^2}\int |\Delta \eta|^2 + O\left(\frac{b^2}{\l^2|\log b|^2}\right) \right] \lesssim \frac{\sqrt{K^*}}{\l^4} \frac {b^3} {|\log b|^2}.
\ee
Using the same proof as the $\dot H^1$ level, {\it ie} dividing by $\l^2$ the monotonicity formula and integrating in time, we prove the bound \fref{bootsmallH2}. The proof is left to the reader.
\subsection{$\dot H^3$ bound:}
We use exactly the same approach than the last subsection.
Now, we focus on this equation :
\bee
\pa_t \Delta \hat z = \Delta^2\hat z - \Delta \hat w + \frac{1}{\l^4}\left[\Delta \tilde \Psi_b^{(2)} + \Delta \widehat Mod^{(2)} \right].
\eee
Next, we have
\bee
&&\frac12 \frac d{dt} \int (1+r^2)|\nabla \Delta \hat z|^2 = \int (1+r^2)\nabla \pa_t \Delta \hat z \nabla \Delta \hat z \\
&=& -\int \pa_t \Delta \hat z \nabla.\left( (1+r^2) \nabla \Delta \hat z \right) \\
&=& -\int \left\{  \Delta^2\hat z - \Delta \hat w + \frac{1}{\l^4}\left[\Delta \tilde \Psi_b^{(2)} + \Delta \widehat Mod^{(2)} \right]\right\} \left\{ 2r \Delta \hat z + (1+r^2) \nabla \Delta \hat z \right\} \\
&=&- \int (1+r^2)|\Delta^2 \hat z|^2 + \int (1+r^2)\Delta^2 \hat z\left\{ - \Delta \hat w + \frac{1}{\l^4}\left[\Delta \tilde \Psi_b^{(2)} + \Delta \widehat Mod^{(2)} \right] \right\} \\
&-& \int 2r \Delta \hat z \left\{  \Delta^2\hat z - \Delta \hat w + \frac{1}{\l^4}\left[\Delta \tilde \Psi_b^{(2)} + \Delta \widehat Mod^{(2)} \right]\right\} 
\eee
Using the bound of the last subsection and the bootstrap bound \fref{bootsmallh2q} :
\bee
\int  (1+r^2) \left| \Delta \hat \e \right|^2 + \int (1+r^2)\left|\Delta  \tilde \Psi_b^{(2)}\right|^2 + \int  (1+r^2) |\Delta  \widehat{Mod}^{(2)} |^2 \lesssim K^* \frac{b^3}{|\log b|^2}.
\eee
In the above estimation, it is very important to see that $\int |\nabla \tilde S_1|^2 \sim \log b$, but $\int (1+r^2) |\Delta \tilde S_1|^2 \lesssim1$.
\par
Next, we have, using the bootstrap bound \fref{bootsmallh3}:
\bee
\left| \int 2r \Delta \hat z \Delta^2\hat z \right| \leq \frac{1}{10} \int (1+r^2)|\Delta^2 \hat z|^2 + O \left(\int |\Delta \hat \eta|^2 \right) \leq \frac{1}{10} \int (1+r^2)|\Delta^2 \hat z|^2 + O \left(K^*\frac{b^2}{|\log b|} \right).
\eee
Thus, using the bounds \fref{estimatonebis}  and \fref{estimatone}, we obtain:
\be
\frac{d}{dt}\left[\frac{1}{\l^2}\int(1+r^2) |\nabla \Delta \eta|^2 + O\left(\frac{b^2}{\l^4|\log b|^2}\right) \right] \lesssim \frac{\sqrt{K^*}}{\l^4} \frac {b^2} {\sqrt{|\log b|}}.
\ee
Here, the bound is enough large, and we can integrate in time, without dividing by a power of $\l$, in order to prove \fref{bootsmallH3}. The proof is left to the reader.
\subsection{$X_Q$ bound \fref{bootsmallh2qboot} }
The proof of this bound is identical as \cite{raphael2012b}. We use the same strategy that for the $\dot H^1$ bound. We sketch the argument for the sake of completeness.
In the last section, we have proved the following monotonicity formula, where we have use the bootstrap bound \fref{bootsmallh2q}:
\be
\label{nknvkonbrono}
\frac{d}{dt}\left\{\frac{(\mathcal M \e_2,\e_2)}{\l^2}+O\left(C(M)\sqrt{K^*}\frac{b^3}{\l^2|\log b|^2}\right)\right\}\lesssim  \sqrt{K^*}\frac{b^4}{\l^4|\log b|^2}.
\ee
First, we integrate in time the last term between 0 and $t^*$.
\bee
\int_0^{t^*}\frac{b^4}{\l^4|\log b|^2}dt& = & \int_0^{t^*} -\l_t \frac{\hb^3}{\l^3|\log \hb|^2}dt+O\left(\int_0^{t^*}\frac{b^5}{\l^4|\log b|^5}dt\right)\\
& = & \left[\frac{\hb^3}{2\l^2|\log \hb|^2}\right]_0^{t^*}-\int_0^{t^*}\frac{1}{2\l^4}\frac{d}{ds}\left[\frac{\hb^3}{|\log \hb|^2}\right]dt+O\left(\int_0^{t^*}\frac{b^5}{\l^4|\log b|^5}dt\right)\\
& = &  \left[\frac{\hb^3}{2\l^2|\log \hb|^2}\right]_0^{t^*}+O\left(\int_0^{t^*}\frac{b^4}{\l^4|\log b|^3}dt\right)
\eee
Thus, we obtain: 
\bee
\l^2(t)\int_0^{t^*}\frac{b^4}{\l^4|\log b|^2}dt\lesssim \frac{b^3(t)}{|\log b(t)|^2}+\left(\frac{\l(t)}{\l(0)}\right)^2\frac{b^3(0)}{|\log b(0)|^2}
\eee
Using the interpolation bound \fref{coerclwotht} and integrating \fref{nknvkonbrono} in time yield
\bea
\label{cneonneeonoe}
\delta(M)\|\e_2(t)\|^2_{L^2_Q}& \lesssim & (\mathcal M\e_2(t),\e_2(t))\\
\nonumber & \lesssim & C(M)\left\{\frac{\l^2(t)}{\l^2(0)}\left[\|\e_2(0)\|_{L^2_Q}^2+\sqrt{K^*}\frac{b^3(0)}{|\log b(0)|^2}\right]+\sqrt{K^*}\frac{b^3(t)}{|\log b(t)|^2}\right\}
\eea
for some small enough universal constants $\delta(M),C(M)>0$ independent of $K^*$. Moreover, \fref{supersuper} implies 
\be
\label{neiovneneonevoneo}
\frac{d}{ds}\left\{\frac{\hb^3}{\l^2|\log \hb|^2}\right\}>0
\ee
and thus \fref{cneonneeonoe} and the smallness of $\|{\bf E}_2(0)\|_{X_Q}$ \fref{initialsmalneesb}  at the initial time yield: 
\bee
\delta(M)\|\e_2(t)\|^2_{L^2_Q}& \lesssim & C(M)\sqrt{K^*}\left[\l^2(t)\frac{b^3(0)}{\l^2(0)|\log b(0)|^2}+\frac{b^3(t)}{|\log b(t)|^2}\right]\\
& \lesssim &  C(M)\sqrt{K^*}\left[\l^2(t)\frac{\hb^3(0)}{\l^2(0)|\log\hb(0)|^2}+\frac{b^3(t)}{|\log b(t)|^2}\right]\\
& \lesssim & C(M)\sqrt{K^*}\frac{\hb^3(t)}{|\log \hb(t)|^2}\lesssim  C(M)\sqrt{K^*}\frac{b^3(t)}{|\log b(t)|^2}
\eee
 and \fref{bootsmallh2qboot} follows for $K^*=K^*(M)$ large enough. This concludes the proof of Proposition \ref{propboot}.
 \section{Proof of the Theorem \ref{thmmain}}
 \label{lastsection}
  We are now in position to conclude the proof of Theorem \ref{thmmain}. The proof follows similar lines as in \cite{raphael2012b},we sketch the argument for the sake of completeness.
 \begin{proof}
  {\bf Step 1} Proof of the blow-up at a finite time $T_0$.
 \\
 \\ Let $u_0\in \mathcal O$ and $u\in \mathcal C([0,T),\mathcal E)$ be the corresponding solution to \fref{kps} with lifetime $0<T\leq +\infty$, then the estimates of Proposition \ref{propboot} hold on $[0,T)$. Observe from \fref{neiovneneonevoneo} the bound $$\l^2\lesssim b^3$$ from which using \fref{supersuper} 
 \be
 \label{poissonrougee}
 -\l\l_t\gtrsim b\gtrsim C(u_0)\l^{\frac 23} 
 \ee 
 and thus
  \bee
   -(\l^{\frac 43})_t\gtrsim C(u_0)>0 
   \eee
    implies that $\l(t)$ touches zero in some finite time $0<T_0<+\infty$.\par
Moreover, as $b$ is a non increasing function, the bounds of Proposition \ref{propboot}, the Hardy bound \fref{harfylog}, the construction of the approximate solution, whose the estimates are available in the Proposition \ref{localization}, and the unique decomposition \fref{decompe} ensure that $$\|\e(t)\|_{H^2}+\|\Delta \eta(t)\|_{L^2}\ll1\ \ \mbox{for}\ \  0\leq t<T_0.$$
Furthermore,
  $$\lim_{t\uparrow T_0}\|u(t)\|_{H^2} + \|\Delta \eta(t)\|_{L^2}=+\infty$$ and thus using a standard argument from Cauchy theory, the solution blows up at $T=T_0<+\infty$.  Moreover, the bound \fref{supersuper} yields: $$|\lambda\lambda_t|=\left|\lsl\right|\lesssim 1 \ \ \mbox{and thus}\ \ \lambda(t)\lesssim \sqrt{T-t},$$ and thus from \fref{defst}: 
\be
\label{glovaltime}
s(t)\to+\infty\ \ \mbox{as}\ \ t\to T_0.
\ee
 
 {\bf Step 2} Computation of the rate of the concentration. \\\\
 In the Lemma \ref{lemmeparam} and \ref{lemmasharpmod}, we have obtained equation of modulation parameters $(b,\hat b,\l)$, depending on a suitable norm for the error term. In the bootstrap, we have obtained a enough good bound for this error \fref{bootsmallh2qboot} in order to be in position now to reintegrate this equation as $s \to \infty$, {\it ie} in the vinicity of the blow-up time.
 \be
 \label{eqsdbivbbie}
 \left|\hat{b}_s+\frac{2\bh^2}{|\log \bh|}\right|\lesssim\frac{\bh^2}{|\log \bh|^2}.
 \ee
 This equation is the same as in \cite{raphael2012}. Thus, we can use the same strategy to obtain the blow-up speed. First we multiply \fref{eqsdbivbbie} by $\frac{|\log \hb|}{\hb^2}$ and obtain: $$\frac{\bh_s\log \bh}{\bh^2}=-2+O\left(\frac{1}{|\log \bh|}\right).$$ We use $$\left(\frac{\log t}{t}+\frac{1}{t}\right)'=-\frac{\log t}{t^2}$$ to conclude after integration: $$-\frac{\log \bh+1}{\bh}=2s+O\left(\int_0^s\frac{d\sigma}{|\log \bh|}\right).$$
 We obtain the equations
 \bee
 \frac{- \log \bh}{ \bh} = \frac{s}{2} \left(1 + O\left(\frac{1}{\log s}\right) \right) \ \ \mbox{and} \ \ \bh = - \frac{-2 \log \bh}{s}\left(1 + O\left(\frac{1}{\log s}\right) \right),
 \eee
  which implies $$\hb(s)=\frac{\log s}{2s}\left(1+o(1)\right), \ \ \log \bh=\log\log s-\log s +O(1).$$ Finally, combining both above estimates, we obtain the development as $s \to \infty$:
 $$\hb(s)=\frac{1}{2s}\left(\log s-\log \log s\right)+O\left(\frac{1}{s}\right).$$
 Using the bound of the gap betwenn $b$ and $\hb$ \fref{deformation}, together the bound on the law $\l$ \fref{estlambda} yield 
 \be
 \label{cbeiheovejiugveuig}
 -\lsl=\hb+O\left(\frac{b}{|\log b|}\right)=\frac{1}{2s}\left(\log s-\log \log s\right)+O\left(\frac{1}{s}\right).
 \ee
Integrating in time this estimation, we obtain
$$-\log \l=\frac14\left[(\log s)^2-2\log s \log \log s\right]+O(\log s)=\frac{(\log s)^2}4\left[1-\frac{2\log\log s}{\log s}+O\left(\frac1{\log s}\right)\right].$$
 Hence, 
 $$\sqrt{|\log \l|}=\frac{\log s}2\left[1-\frac{\log\log s}{\log s}+O\left(\frac1{\log s}\right)\right]$$ and thus:
$$e^{2\sqrt{|\log \l|}+O(1)}=\frac{s}{\log s}, \ \ s=\sqrt{|\log \l|}e^{2\sqrt{|\log \l|}+O(1)}.$$ We use these relations to rewrite the modulation equation \fref{cbeiheovejiugveuig}: $$-s\lsl=-\sqrt{|\log \l|}e^{2\sqrt{|\log \l|}+O(1)}\l\l_t=\sqrt{|\log \l|}+O(1)$$ and thus 
\be
\label{poissonrouge}
-\l\l_te^{2\sqrt{|\log \l|}}=e^{O(1)}.
\ee
 The time integration with boundary condition $\l(T_0)=0$ yields $$\lambda(t)=\sqrt{T_0-t}e^{-\sqrt{\frac{|\log (T_0-t)|}{2}}+O(1)}\ \ \mbox{as}\ \ t\to T_0,$$ this is \fref{rate}.\\ \\
 {\bf Step 3} Strong convergence \fref{boievbebeo} \\ 
 \\
Injecting the bound \fref{poissonrouge} in \fref{poissonrougee} implies
$$b(t)\to 0\ \ \mbox{as}\ \ t\to T_0.$$ 
From the estimation of the Proposition \ref{localization}, the above convergence implies $$\|{\bf Q}- \tilde {\bf Q}_b\|_{W_Q} = \|\tilde{\boldsymbol \Upsilon}_b\|_{W_Q}\to 0\ \ \mbox{as}\ \ t\to T_0.$$ 
and the strong convergence \fref{boievbebeo} now follows from the Proposition \ref{propboot}.\\
This concludes the proof of Theorem \ref{thmmain}.
 \end{proof}




\begin{appendix}


\section{Hardy bounds}
In this section, we prove logarithmic Hardy inequalities for radial functions $u \in H^3_{rad}(\mathbb R^2)$. They are, with the explicit knowledge of the repulsive structure of the linearized operator $\mathcal L$, the keystone of the proof of the Proposition \ref{interpolationhtwo}, describing the coercivity of this operator under additional orthogonality conditions. This is standard weighted Hardy inequalities, however theirs proofs are displayed for the reader's convenience. 
\begin{lemma}[Weighted Hardy inequality]
\label{weightedhardy}
There holds the Hardy bounds: 
\be
\label{hardyboundbis}
\forall \alpha>-2, \ \  \int r^{\alpha+2}|\pa_rv|^2\geq \frac{(2+\alpha)^2}{4} \int r^{\alpha}v^2,
\ee
$\forall R>2$, $\forall v\in\dot{H}^1_{rad}(\RR^2)$ and $\gamma> 0$, there holds the following controls:
\bea
\label{harfylog}
 \int_{r\le R} \frac{|v|^2}{r^2(1+|\log r|)^2}rdr&\lesssim&\int_{1\leq r\leq2} |v|^2 +  \int_{r\le R} |\nabla v|^2, \\
\label{hardylog'}
 \int_{1 \leq r\le R} \frac{|v|^2}{r^{\gamma+2}(1+|\log r|)^2}rdr&\lesssim&\int_{1\leq r\leq2} |v|^2 +  \int_{1\leq r\le R} \frac{|\nabla v|^2}{r^\gamma(1+|\log r|)^2}, \\
\label{hardylevel1}
\int\frac{v^2}{r^2(1+r^4)(1+|\log r|)^2} &\lesssim& \int \frac{|\nabla v|^2}{r^{4}(1+|\log r|)^2}-\int\frac{|v|^2}{1+r^8}
\eea
$\forall v\in H^2_{rad}(\RR^2)$ and $\gamma\in [0,2[$, there holds the bounds:
\be
\label{hardylevel2}
\int \frac{|\nabla v|^2}{r^{4}(1+|\log r|)^2} + \int \frac{|\nabla^2 v|^2}{r^{2}(1+|\log r|)^2} \lesssim \int \frac{|\Delta v|^2}{r^{2}(1+|\log r|)^2}
\ee
$\forall v\in H^3_{rad}(\RR^2)$, there holds the Hardy bounds:
\be
\label{hardylevel3}
 \int \frac{|\Delta v|^2}{r^2(1+|\log r|)^2} - \int \frac{|\Delta v|^2}{1+r^4}\lesssim  \int |\nabla(\Delta v)|^2,
\ee
\end{lemma}
\begin{proof}
\fref{hardyboundbis} is a simple consequence of the following integration by parts, for $v \in \mathcal C_c^{\infty} (\mathbb R^2)$ :
$$\frac{\alpha+2}{2}\int r^{\alpha} v^2=-\int r^{\alpha+1}v\pa_rv\leq \left(\int r^{\alpha}v^2\right)^{\frac 12}\left(\int r^{\alpha+2}(\pa_rv)^2\right)^{\frac 12}$$
Now, for $v \in \mathcal C_c^{\infty} (\mathbb R^2)$, let's prove \fref{harfylog}. For this purpose, let the radial function $f(r)=-\frac{1}{r(1+|\log (r)|)}$ so that 
\bee
\nabla \cdot f= \left| \begin{array}{l}\frac{1}{r^2(1+|\log r|)^2} \ \ \mbox {for} \ \ r \geq1 \\ \frac{-1}{r^2(1+|\log r|)^2} \ \ \mbox {for} \ \ r \leq1 \end{array}\right.,
\eee and integrate by parts to get, with $\e>0$: 
\bea
\nonumber & & \int_{\e\le r\le R} \frac{|v|^2}{r^2(1+|\log r|)^2}rdr  =-  \int_{\e\leq r\leq 1} |v|^2\nabla \cdot f rdr + \int_{1\leq r\leq R} |v|^2\nabla \cdot f rdr\\
\nonumber & = &- \left[\frac{|v|^2}{1+|\log (r)|}\right]_{1}^R + \left[\frac{|v|^2}{1+|\log (r)|}\right]_{\e}^1 +2\int_{r\le R} v\partial_r v \frac{1}{r(1+|\log r|)}rdr\\
\label{stepwzotfp} & \lesssim & |v(1)|^2+\left(\int_{r\le R} \frac{|v|^2}{r^2(1+|\log r|)^2}rdr\right)^{\frac{1}{2}}\left(\int_{r\le R} |\nabla v|^2rdr\right)^{\frac{1}{2}}.
\eea
On other hand,
\bee
|v(1)|^2 \leq \|v\|^2_{L^\infty_{1\leq r \leq 2}} \lesssim \int_{1\leq r\leq2} |v|^2+\int_{1\leq r\leq2} |\nabla v|^2
\eee
 Injecting this into \fref{stepwzotfp} and letting $\e\to 0$ yields \fref{harfylog}. 
To prove \fref{hardylog'}, let $f(r)=-\frac{{\bf e}_r}{(1+|\log (r)|)r^{\gamma+1}}$ so that 
 $$
 \nabla \cdot f=\frac{2}{(1+|\log r|)^5 r^{\gamma+2}}+\frac{\gamma }{(1+|\log r|)^2 r^{\gamma+2}},
 $$ and integrate by parts to get: 
\bea
\label{stepwzofp}
\nonumber & & \gamma \int_{1\le r\le R} \frac{|v|^2}{r^{\gamma+2}(1+|\log r|)^2}  \le  \int_{1\leq r\le R} |v|^2\nabla \cdot f \\
\nonumber & = &- \left[\frac{|v|^2}{r^\gamma(1+|\log r|)^2}\right]_{1}^R +2\int_{1\leq r\leq R} v\partial_r v \frac{1}{r^{\gamma+1}
(1+|\log r|)^2}\\
& \lesssim & |v(1)|^2+\left(\int_{1\leq r\leq R} \frac{|v|^2}{r^{\gamma+2}(1+|\log r|)^2}\right)^{\frac{1}{2}}\left(\int_{1 \leq r \leq R} \frac{|\nabla v|^2}{r^{\gamma}(1+|\log r|)^2}\right)^{\frac{1}{2}}\notag.
\eea
We concludes the proof of  \fref{hardylog'} using the same way as the last. The bound \fref{hardylevel1} is a simple consequence of the two last bounds.To prove \fref{hardylevel2}, we compute :
\bee
 \int \frac{|\Delta v|^2}{y^{2}(1+|\log y|)^2} = \int \frac{|\nabla v|^2}{y^{4}(1+|\log y|)^2} + \int \frac{|\nabla^2 v|^2}{y^{2}(1+|\log y|)^2} + 2 \int \frac{\pa_r v \pa_{rr} v }{y^{3}(1+|\log y|)^2}.
\eee
Now,
\bee
&&2 \int_\e^R \frac{\pa_r v \pa_{rr} v }{r^{3}(1+|\log r|)^2} =  \int_\e^R \frac{\pa_r (( \pa_{r} v) ^2) }{r^{3}(1+|\log r|)^2} \\
&=& \left[ \frac{( \pa_{r} v) ^2 }{r^{2}(1+|\log r|)^2}\right]_\e^R - \int_\e^R  |\pa_r v |^2\pa_r \left(\frac{1}{r^{2}(1+|\log r|)^2}\right)dr
\eee
 For $u \in H^2_{rad} (\mathbb R^2)$, all integrals in the above equality are absolutely convergent. Moreover there exists a sequence $R_n \underset{n \rightarrow + \infty}{\rightarrow} +\infty$ such that
 \bee
 \left[ \frac{( \pa_{r} v) ^2 }{r^{2}(1+|\log r|)^2}\right]_{\e}^{R_n} \underset{n \rightarrow + \infty}{\rightarrow} 0 - \frac{|\pa_{r} v(\e)|^2 }{\e^{2}(1+|\log \e|)^2} <0.
 \eee
 Remarking that $\pa_r \left(\frac{1}{r^{2}(1+|\log r|)^2} \right) <0$, this implies that, $\forall \e >0$ 
 \bee
 2 \int_\e^R \frac{\pa_r v \pa_{rr} v }{r^{3}(1+|\log r|)^2} < 0.
 \eee
 This concludes the proof of \fref{hardylevel2}.
The bound \fref{hardylevel3} comes directly from the bound \fref{harfylog}. 
\end{proof}
\section{On the Poisson field}
In this section, we shall write two technical Lemma on the Poisson field. The proof of the first lemma is available in \cite{raphael2012b}. 
\begin{lemma}[Interpolation estimates]
\label{lemmainterpolation}
Let $u\in L^2_Q$, with $\int u=0$ then
\be
\label{improvedlinfty}
\|\nabla \phi_u\|_{L^{2}}\lesssim \|u\|_{L^2_Q}.
\ee
\end{lemma}
\begin{lemma}
\label{phideltav}
Let $v \in H^1_{rad}$ smooth, such that $\int (1+|\log r|^2) |\nabla v|^2 < \infty$. Then $\phi_{\Delta v}=v$.
\end{lemma}
\begin{proof}
Let $v$ satisfying the conditions of the lemma. Let $U=\phi_{\Delta v}$. In the framework of the radial functions, there exist a constant $\alpha$ such that $U = v + \alpha$. Using the convolution representation, we have:
\bee
\phi_{\Delta v} (0) = \int_0^\infty \log (r) \Delta v(r) r dr.
\eee
Remark that this integral is absolutely convergent. Now, let $0<\e<R<+\infty$. 
\bee
\int_{\e}^R \log (r) \Delta v(r) r dr &=& \left[ r\log(r)v'(r)\right]_\e^R - \int _\e^R v'(r)dr \\
&=& R\log R v'(R) - v(R) -\e \log \e  v'(\e)+ v(\e).
\eee
As v is smooth, we have
\be
\label{ones}
\lim \limits_{\e \rightarrow 0} \e \log \e  v'(\e) = 0.
\ee
From the hypothesis of the Lemma:
 \bee
  \int |v|^2 + \int (1+|\log r|^2) |\nabla v|^2 < \infty
 \eee
Hence, there exists a sequence $R_n \rightarrow + \infty$ such that 
\be
\label{seconds}
\lim \limits_{n \rightarrow +\infty} R_n\log R_n  v' (R_n)- v(R_n) = 0.
\ee
This yields together \fref{ones} :
$\phi_{\Delta v} (0) = v(0)$, and $\alpha=0$. This concludes the proof of the lemma \fref{phideltav}.
\end{proof}
\section{Interpolation bounds}
In this section, using the bootstrap bounds of the subsection \ref{apzoeirutymm}, we obtain interpolation bounds, which are the keystone of the proof of the monotonicity formulas.
 \begin{proposition}[Interpolation bounds]
 \label{interpolation}
(i) $W_Q$ bound :
 \bea
 \label{coercbase} && \int (1+r^4) |\Delta \e|^2 + \int (1+r^2) |\nabla \e|^2 + \int \e^2 + \int \frac{|\nabla \phi_\e|^2}{r^2(1+|\log r|)^2} \\
\nonumber&+& \int |\nabla \Delta \eta|^2 + \int \frac{|\Delta \eta|^2}{r^2(1+|\log r|)^2} + \int \frac{|\nabla \eta|^2}{r^4(1+|\log r|)^2} \lesssim C(M) \|{\bf E}_2\|_{W_Q}^2
 \eea
 \be
 \label{normeWQ2}
  \int (1+r^2) |\nabla \e_1|^2 \lesssim C(M) \|{\bf E}_2\|_{W_Q}^2 
 \ee
 (ii) $L^2$ bound :{\footnote {We recall that $K^*$ is the constant of the bootstrap.}}
 \bea
\label{wnueg} \int (1+r^2) |\nabla \Delta \eta |^2 &\lesssim&K^* \frac{b}{\sqrt{|\log b|}}\\
\label{wnuegg} \int |\Delta \eta |^2 &\lesssim& K^*\frac{b^2}{|\log b|}\\
\label{wnueggg}  \int |\nabla \eta |^2 &\lesssim& K^*b^2|\log b|^6
 \eea
 (iii) $L^{\infty}$ bound: 
 \bea
\label{normeinftyye} \|(1+r) \e\|_{L^{\infty}}^2 &\lesssim& C(M) \|{\bf E}_2\|^2_{W_Q} \\
\label{normeinftynablaeta} \|\nabla \eta\|_{L^{\infty}}^2 &\lesssim&K^* b^2|\log b|^6 \\
\label{normeinftydeltaeta} \|(1+r)\Delta \eta\|_{L^{\infty}}^2 +\|(1+r)\nabla^2 \eta\|_{L^{\infty}}^2 &\lesssim&K^*  \frac{b}{\sqrt{|\log b|}}
 \eea
 \end{proposition}
 \begin{proof}
 The proof of \fref{coercbase} is a simple consequence of the Proposition \ref{interpolationhtwo} and the bootstrap bound \fref{bootsmallh2q}. The bound \fref{normeWQ2} comes from the following inequality and the bound \fref{coercbase} :
 \bee
  \int (1+r^2) |\nabla \e_1|^2 \lesssim  \int (1+r^2) |\nabla \e|^2 + \int \frac{\e^2}{1+r^2} + \int \frac{|\nabla \eta|^2}{1+r^6}.
 \eee
 The bound \fref{wnueg},  \fref{wnuegg} and  \fref{wnueggg} are exactly bounds of the bootstrap. 
 To prove \fref{normeinftyye}, we have near the origin the estimate, using Sobolev and the bound \fref{coercbase} :
 \bee
 \|\e\|_{L^{\infty}_{0\leq r \leq 1}}^2 \lesssim  \|\e\|_{L^{2}_{0\leq r \leq 1}}^2 +  \| \nabla \e\|_{L^{2}_{0\leq r \leq 1}}^2 \lesssim C(M) \|{\bf E}_2\|^2_{W_Q} 
 \eee
 Now, let $f(r) = (1+r)\e(r)$ and $a \in ]1,2[$ such that:
 \bee
 f(a)^2 \leq \int_1^2 |f|^2
 \eee
Let $y \leq 1$. Then
\bee
f(y)^2 = f(a)^2 + \int_a^y \pa_r(f^2)dr.
\eee
Using Cauchy-Schwarz,
\bee
\left| \int_a^y \pa_r(f^2)dr\right|^2 \lesssim \int_1^y \frac{f^2}{r^2}\int_1^y |\pa_r f|^2.
\eee
Using the definition of f, we obtain
\bee
 \|(1+r) \e\|_{L^{\infty}}^2 \lesssim \int |\e|^2 + \int (1+r^2) |\nabla \e|^2 \lesssim C(M) \|{\bf E}_2\|^2_{W_Q} 
\eee
This concludes the proof of \fref{normeinftyye}. Finally, to prove \fref{normeinftynablaeta} and \fref{normeinftydeltaeta}, we use the same strategy with the bound \fref{wnueg}, \fref{wnuegg} and \fref{wnueggg}. This concludes the proof of the Proposition \ref{interpolation}.
  \end{proof}
 \end{appendix}

\end{document}